\documentclass[11pt,reqno]{amsart}
\usepackage[margin=1in,letterpaper]{geometry}
\usepackage{graphicx}
\usepackage{amssymb}
\usepackage{amsthm}
\usepackage{mathrsfs}
\usepackage{stmaryrd}
\usepackage{accents} 
\usepackage{enumitem} 
\usepackage[bookmarksopen,bookmarksdepth=3]{hyperref} 
\makeatletter\let\over\@@over\makeatother

\numberwithin{equation}{section}
\theoremstyle{plain} 
\newtheorem{theorem}{Theorem}[section] 
\newtheorem{proposition}[theorem]{Proposition} 
\newtheorem{corollary}[theorem]{Corollary}
\newtheorem{lemma}[theorem]{Lemma}
\theoremstyle{remark}
\newtheorem{remark}[theorem]{Remark}

\newcommand{\be}{\begin{equation}}
\newcommand{\ee}{\end{equation}}%
\newcommand{\bse}{\begin{subequations}}
\newcommand{\ese}{\end{subequations}}

\newcommand{\supp}[1]{\operatorname{supp}{#1}}
\newcommand{\sech}{\operatorname{sech}} 
\newcommand{\jump}[1]{\left\llbracket{#1}\right\rrbracket}

\newcommand{\kernel}{\operatorname{ker}}

\newcommand{\id}{\operatorname{id}}
\newcommand{\FSproj}{\mathcal{Q}}
\newcommand{\FP}{\mathfrak{W}}

\newcommand{\R}{\mathbb{R}} 
\newcommand{\placeholder}{\,\cdot\,}
\newcommand{\maps}{\colon}         
\newcommand{\by}{\times}         
\newcommand{\sub}{\subset}         
\newcommand{\n}[2][]{#1\lVert #2 #1\rVert}
\newcommand{\abs}[2][]{#1\lvert #2 #1\rvert}
\newcommand{\dell}{\partial}
\newcommand{\ina}{\textup{~in~}}
\newcommand{\ona}{\textup{~on~}}
\newcommand{\fora}{\textup{~for~}}
\newcommand{\asa}{\textup{~as~}}
\renewcommand{\l}{\kappa} % Maybe this is what was meant?

\newcommand{\even}{\mathrm{e}}      
\newcommand{\bdd}{\mathrm{b}}

\newcommand\nl{\mathcal N}   % nonlinear term
\newcommand\base{{\Omega^\prime}} % base of the cylinder
\newcommand\A{\mathcal A}    % principle term for elliptic problem
\newcommand\B{\mathcal B}    % lower-order term
\newcommand\G{\mathcal G}    % boundary term
\newcommand\K{\mathcal K}    % boundary term
\newcommand\F{\mathscr F}    % the nonlinear problem

\newcommand\strainW{\mathcal W}    % the nonlinear problem

\newcommand\homholder{\mathring C} % homogeneous H\"older spaces
\newcommand\Xspace{{\mathscr X}}
\newcommand\Yspace{{\mathscr Y}}
\newcommand\Hspace{{\mathscr H}}

\newcommand\atXspace{\mathfrak{X}}  % spaces for Amick--Turner fixed-point
\newcommand\atYspace{\mathfrak{Y}}
\newcommand\atZspace{\mathfrak{Z}}
\newcommand\atF{\mathfrak{F}}

\newcommand\polydyn{p}
\newcommand\polyflow{\tilde q}
\newcommand\polynew{q}
\newcommand\polyconj{\mathcal P}

\newcommand\flow{\mathcal S}
\newcommand\flowdiff{\tilde{\mathcal S}}
\newcommand\flowred{s}
\newcommand\flowscaled{S}

\begin{document}

\title[Center manifolds for quasilinear PDE]{Center manifolds without a phase space for quasilinear problems in elasticity, biology, and hydrodynamics}
\date{\today}

\author[R. M. Chen]{Robin Ming Chen}
\address{Department of Mathematics, University of Pittsburgh, Pittsburgh, PA 15260} 
\email{mingchen@pitt.edu}  

\author[S. Walsh]{Samuel Walsh}
\address{Department of Mathematics, University of Missouri, Columbia, MO 65211} 
\email{walshsa@missouri.edu}

\author[M. H. Wheeler]{Miles H. Wheeler}
\address{Faculty of Mathematics, University of Vienna, Vienna 1080, Austria}
\email{miles.wheeler@univie.ac.at}

\begin{abstract}
In this paper, we present a novel center manifold reduction theorem for quasilinear elliptic equations posed on infinite cylinders.  This is done without a phase space in the sense that we avoid explicitly reformulating the PDE as an evolution problem.  Under suitable hypotheses, the resulting center manifold is finite dimensional and captures all sufficiently small bounded solutions.  Compared with classical methods, the reduced ODE on the manifold is more directly related to the original physical problem and also easier to compute. The analysis is conducted directly in H\"older spaces, which is often desirable for elliptic equations.  

We then use this machinery to construct small bounded solutions to a variety of systems.    These include heteroclinic and homoclinic solutions of the anti-plane shear problem from nonlinear elasticity; exact slow moving invasion fronts in a two-dimensional Fisher--KPP equation; and hydrodynamic bores with vorticity in a channel.  The last example is particularly interesting in that we find solutions with critical layers and  distinctive ``half cat's eye'' streamline patterns.

\end{abstract}
\maketitle

\setcounter{tocdepth}{1} 
\tableofcontents

\section{Introduction} \label{introduction section}

Our basic objective in this paper relates to a classical problem: characterizing small bounded solutions of a quasilinear elliptic PDE posed on an unbounded cylinder $\Omega = \mathbb{R} \times \base$. The base of the cylinder $\base \subset \mathbb{R}^{n-1}$ is a bounded and connected $C^{2+\alpha}$ domain for some $\alpha \in (0,1)$, and the dimension $n \geq 2$.    For simplicity, say that $0 \in \base$.  

As a fairly representative example, we initially focus on the following quasilinear PDE:
\begin{equation}
  \left\{ \begin{aligned} 
    \nabla \cdot \A(y, u, \nabla u, \lambda)  + \B(y, u, \nabla u, \lambda) & = 0 \qquad \textrm{in } \Omega \\
    \G(y, u, \nabla u, \lambda) & = 0 \qquad \textrm{on } \partial\Omega,
  \end{aligned} \right. \label{main elliptic PDE} 
\end{equation}
where spatial coordinates in $\Omega$ are written $(x,y)$ for $x \in \mathbb{R}$ and $y \in \base$.  Here, $\lambda \in \mathbb{R}$ is a parameter, while $u = u(x,y) \in C^{2+\alpha}(\overline{\Omega})$ is the unknown.   We ask that the functions $\A = \A(y,z, p, \lambda)$, $\B = \B(y, z, p, \lambda)$, and $\G = \G(y, z,p, \lambda)$ are uniformly $C^{M+4}$ in their arguments for a fixed integer $M \geq 2$.  Moreover, we assume that the interior equation is uniformly elliptic in the sense that there exists $\theta > 0$ such that
\begin{equation}
  \label{nonlinear ellipticity}
   \sum_{i,j} \A_{i,p_j}(y,z,p,\lambda) q_i q_j \geq \theta |q|^2 \qquad \textrm{for all } y \in \base,  ~p, q \in \mathbb{R}^n,~z, \lambda \in \mathbb{R}. 
\end{equation}
The boundary condition is taken to be uniformly oblique in that there exists $\chi > 0$ such that 
\begin{equation}
  \label{nonlinear obliqueness}
  -N(y) \cdot \G_{p}(y, z, p, \lambda) \geq \chi \qquad \textrm{for all } y \in \base, ~p \in \mathbb{R}^n, ~z,  \lambda \in \mathbb{R},
\end{equation}
where $N = (0, N^\prime) \in \mathbb{R}^{n}$ denotes the outward unit normal to $\Omega$ on $\partial\Omega = \mathbb{R} \times \partial \base$.  Note that since the coefficients in \eqref{main elliptic PDE} are independent of $x$, the full nonlinear problem  is invariant under axial translation.

Borrowing terminology from dynamical systems, we say a solution $(u,\lambda)$ of \eqref{main elliptic PDE} is \emph{homoclinic} if $u$ limits to a fixed function as $|x| \to \infty$, and we call it \emph{heteroclinic} provided $u$ has  distinct limits as $x \to \pm\infty$.  Beyond their intrinsic mathematical importance, equations of the form \eqref{main elliptic PDE} arise in a surprisingly diverse array of physical settings.  Of particular interest to us is their connection to traveling waves in nonlinear elasticity, mathematical biology, and especially hydrodynamics.   In those contexts, homoclinic solutions are referred to variously as \emph{pulses}, \emph{solitons}, or \emph{solitary waves}, while and heteroclinics correspond to \emph{fronts} or \emph{bores}.   Although the techniques we develop are equally well-suited to both these types of solutions, our emphasis will be on fronts because they are more difficult to construct.   An ulterior motive for this choice is that, in a forthcoming paper, we will present a global bifurcation theory for heteroclinics.

The unboundedness of $\Omega$ seriously complicates the task of finding these solutions.  For example, it is well-known that the relevant linearized operators fail to be Fredholm in unweighted H\"older spaces, which precludes the direct application of bifurcation theoretic techniques. For semilinear problems,  monotonicity methods have proven to be effective; see, for example, Berestycki and Nirenberg \cite{berestycki1992fronts}, A. Volpert, V. Volpert, and V. Volpert \cite{volpert1994book}, and the references therein.  By contrast, in the quasilinear setting, the predominant approach is to reformulate \eqref{main elliptic PDE} as a \emph{spatial} dynamical system (that is, treating $x$ as an evolution variable), and use infinite-dimensional invariant manifold theory.  Seeking small bounded solutions, we might hope to construct a finite-dimensional center manifold and study the bounded orbits of a reduced equation there. Beginning with the pioneering work of Kirchg\"assner \cite{kirchgassner1982wavesolutions} and Mielke \cite{mielke1986reduction,mielke1988reduction} in the 1980s, this basic strategy has been built upon and applied to great effect by many authors; see, for example, the book of Haragus and Iooss \cite{haragus2011book} for historical overview or \cite{dias2003handbook} for applications to water waves.   

While the Mielke--Kirchg\"assner approach is quite general and very powerful, it is not perfectly suited to every problem.  For many systems, such as \eqref{main elliptic PDE}, the reformulation as an evolution equation contorts the PDE in an unnatural way.  In particular, accommodating nonlinear boundary conditions typically requires one or more implicit changes of dependent variables.  This is certainly possible to do, but it adds an additional layer of complexity to the already involved process of computing the reduced ODE.   More importantly, it obscures the relationship between the equation on the center manifold and the physical problem.  Another potentially limiting factor is that the above theory is formulated in relatively weak Sobolev spaces in the transversal variable $y$ due to its reliance on so-called optimal regularity estimates.   When studying elliptic PDEs, it is often  desirable to work directly in spaces of H\"older continuous functions.   

Recently, Faye and Scheel \cite{faye2016center} introduced an alternative technique that ameliorates some of these issues.  Rather than reformulate the problem as an evolution equation, they instead perform a delicate fixed point argument in exponentially weighted Sobolev spaces. This furnishes what they call a center manifold ``without a phase space.'' Indeed, the manifold is parameterized by the components of the solution in the kernel of the linearized operator rather than initial data.  This permits them to treat certain non-local problems --- which was their original intent --- and also greatly simplifies the arduous task of computing the reduced equation.  Unfortunately, the Faye--Scheel method is fundamentally restricted to semilinear problems, and it appears to be ill-adapted to H\"older spaces.  

As one of the main contributions of this paper, we present a new center manifold reduction theorem that is specialized to treat quasilinear elliptic problems of the form \eqref{main elliptic PDE}, as well as more general ones.   The analysis is conducted entirely in H\"older spaces and, like Faye--Scheel, the reduced equation can be computed with comparatively elementary methods.  For heteroclinic solutions, one must expand the reduction function to cubic order, and so these differences in complexity are especially salient.  A particularly attractive feature of this machinery is that one has the freedom to choose the projection involved in the definition of the center manifold.  For instance, when we study surface water waves, we can arrange for the reduced ODE to directly govern the free boundary.  In this way,  the physical context remains in view even as we restrict to the center manifold.  
%Among the many benefits of this that qualitative features of orbits to the reduced system lift in a completely transparent way to solutions of the full PDE.     

The most technically challenging part of constructing a center manifold invariably involves solving a fixed point  problem in weighted spaces and then verifying that the solution depends smoothly on the parameters.  For this, we are deeply indebted to a paper of Amick and Turner \cite{amick1994center}, where bounds and Fr\'echet differentiability of superposition operators in exponentially weighted H\"older spaces is painstakingly worked out.   
In fact, these authors developed their own center manifold reduction based on the above estimates and a point-wise in $x$ spectral splitting approach.  We use their ideas to construct a preliminary center manifold, and then reconfigure it in the style of Faye and Scheel, obtaining the simplified expansion procedure and freedom of projection choice.  

%Briefly put: the classical spatial dynamics approach relies on invariant manifold theory and infinite-dimensional evolution equations; Faye--Scheel, on the other hand, uses purely functional analysis techniques; while Amick and Turner have a more PDE-centric approach that relies on point-wise in $x$ separation of variables.    What we present takes the very general machinery of Amick and Turner, specializes it to quasilinear elliptic PDE, then 

The second part of the present paper consists of three nontrivial applications of our center manifold reduction theorem. These problems were selected both for their physical significance and to illustrate different aspects of the methodology.   First, we prove the existence of homoclinic and heteroclinic solutions to the anti-plane shear equations from nonlinear elasticity.   Second, we verify the existence of slow moving fronts in a two-dimensional Fisher--KPP system with absorbing boundary conditions.

Finally, and most substantially, we construct small rotational bores in a channel.  These are heteroclinic solutions of the full two-dimensional incompressible Euler equations, with two immiscible layers of constant density fluid separated by a free boundary.  A major novelty is that we allow for constant vorticity as well as critical layers.

\subsection*{Notation}  Here we record some notational conventions followed throughout the rest of the paper.  Let  $U \subset \R^n$ be a cylinder in dimension $n \geq 2$.  For $k \in \mathbb{N}$, $\alpha \in (0,1)$, $\mu \in \mathbb{R}$, and a function $f \in C^k(U)$, we define the exponentially weighted H\"older norm
\begin{align*}
  \n f_{C_\mu^{k+\alpha}(U)}
  := \sum_{\abs \beta \le k} \n{w_\mu \partial ^\beta f}_{C^0(U)}
  + \sum_{\abs \beta = k} \n{w_\mu |\partial^\beta f|_{\alpha} }_{C^0(U)},
\end{align*}
where $w_\mu(x) := \sech{(\mu x)}$ is an exponential weight function and $|f|_{\alpha}$ is the local H\"older seminorm
\begin{align*}
  |f|_{\alpha}(x,y) := \sup_{\substack{|(r,s)| < 1\\ (x+r, y+s) \in U}}
  \frac{\abs{f(x+r, y+s)-f(x,y)}}{\abs{(r,s)}^\alpha}.
\end{align*}
We denote by 
\[ C_\mu^{k+\alpha}(\overline{U}) := \left\{ f \in C^{k+\alpha}(\overline{U}) : \| f \|_{C_\mu^{k+\alpha}(U)} < \infty \right\}.\] 
Occasionally, we will also work with the space of uniformly bounded H\"older continuous functions $C_{\bdd}^{k+\alpha}(\overline{U})$, which is defined as $C_\mu^{k+\alpha}(\overline{U})$ with $\mu = 0$.  Since $U$ is unbounded, this a proper subset of the space 
$C^{k+\alpha}(\overline{U})$ of functions which are merely \emph{locally} H\"older continuous up to the boundary.  Finally, for $k \geq 1$, and $\mu$ and $\alpha$ as above, we define the homogeneous seminorm 
  \begin{align*}
   |f|_{\homholder_\mu^{k+\alpha}(U)}
  := \sum_{1 \leq \abs \beta \le k} \n{w_\mu \partial ^\beta f}_{C^0(U)}
  + \sum_{\abs \beta = k} \n{w_\mu |\partial^\beta f|_{\alpha} }_{C^0(U)},
\end{align*}
and say $f \in \homholder_\mu^{k+\alpha}(\overline{U})$ provided that $|f|_{\homholder_\mu^{k+\alpha}(\overline{U})} < \infty$.

%\subsection{Formulation} \label{formulation section}

% by replacing $\A$, $\B$, and $\G$ by 
%\[ \tilde{\A} := \frac{1}{a_{11}} \A, \qquad \tilde{\B} :=  \tilde{\A} \cdot \nabla a_{11} + \frac{1}{a_{11}} \B, \qquad \textrm{and} \qquad \tilde{\G} := \frac{1}{a_{11}} \G. \]
%Notice that, because $a_{11}$ is independent of $x$, $\tilde{\A}$ and $\tilde{\B}$ also exhibit the reflection symmetry property \eqref{explicit reflection assumption}.  

%Associated to \eqref{linear form hypothesis} is a transversal elliptic operator set on $\base$, namely $L^\prime = (L_1^\prime, L_2^\prime)$ defined by 
%\begin{equation}
% \label{transverse problem}
%  \begin{split} 
%    L_1^\prime w &:=  \nabla^\prime \cdot \left( a^\prime(y,\lambda) \nabla^\prime w + b(y,\lambda) w\right) - b(y,\lambda) \cdot \nabla^\prime w+ c(y, \lambda) w \\
%    L_2^\prime w & := \left( -N^\prime(y) \cdot \left( a^\prime(y,\lambda) \nabla^\prime w + b(y, \lambda) w \right) + g(y,\lambda) w \right)\Big|_{\partial \base}  
%  \end{split}
%\end{equation}
%for all $w = w(y) \in C_\bdd^{2+\alpha}(\overline{\base})$.  As $\base$ is a bounded and smooth domain in $\mathbb{R}^{n-1}$, standard elliptic theory ensures that, for each $\lambda \in \mathbb{R}$, the spectrum $\Sigma(\lambda)$ of $L^\prime$ is  $\{ \nu_k \}_{k = 0}^\infty$, with each $\nu_k = \nu_k(\lambda)$ being an eigenvalue of finite multiplicity and $\nu_k \to -\infty$ as $k \to \infty$.  Moreover, there is an orthonormal basis of $L^2(\base)$ comprised by the corresponding eigenfunctions $\{ \varphi_k \}_{k=0}^{\infty}$.  

\subsection{Statement of results}

Written as an abstract operator equation, the elliptic problem \eqref{main elliptic PDE} takes the form 
\begin{equation}
  \F(u, \lambda) = 0,\label{abstract nonlinear equation} 
\end{equation}
where 
\[ \F = (\F_1, \F_2) \colon C_\bdd^{2+\alpha}(\overline{\Omega})  \times \mathbb{R} \longrightarrow C_\bdd^{0+\alpha}(\overline{\Omega}) \times C_\bdd^{1+\alpha}(\partial \Omega).\]
By this convention, $\F_1$ represents the equation in the interior, whereas $\F_2$ corresponds to the boundary condition.  

It is well-known that families of ``long waves'' can be found bifurcating from ``trivial'' $x$-independent solutions at certain critical parameter values (often connected to so-called dispersion relations).  This intuition motivates the following structural assumptions on $\F$.  First, suppose that there exists a family of trivial solutions parameterized by $\lambda$; for simplicity, this can be stated as
\begin{equation}
  \F(0, \lambda) = 0 \qquad \textrm{for all } \lambda \in \mathbb{R}. \label{F trivial solution assumption} 
\end{equation}
We will study solutions near $(u,\lambda) = (0,0)$, which leads us to consider the linearized operator $L := \F_u(0,0)$.  We make two hypotheses on $L$.  First,
\begin{equation}
  L := \F_u(0,0)  \textrm{ is formally self-adjoint with a co-normal boundary condition.}  \label{F symmetry assumption} 
\end{equation}
Second, we make a spectral assumption on the transversal linearized operator
\[ L^\prime := L |_{C^{2+\alpha}(\overline{\base})} : C^{2+\alpha}(\overline{\base}) \to C^{0+\alpha}(\overline{\base}) \times C^{1+\alpha}(\partial \base),\]
which results from restricting $L$ to acting on $x$-independent functions. As $\base$ is a bounded and smooth domain in $\mathbb{R}^{n-1}$, standard elliptic theory ensures that the spectrum of $L^\prime$ consists of finite multiplicity eigenvalues $\nu_0 > \nu_1 > \cdots$ with  $\nu_k \to -\infty$ as $k \to \infty$.  Moreover, there is an orthonormal basis of $L^2(\base)$ comprised by the corresponding eigenfunctions $\{ \varphi_k \}_{k=0}^{\infty}$. Our final assumption is that 
\begin{equation}
  \label{lambda0 assumption}  \nu_0 = 0 \textrm{ is a simple eigenvalue}.
\end{equation} 
This is the sense in which  the parameter value $\lambda =0$ is critical. 
%Naturally, this can easily be adapted to the case where \eqref{lambda0 assumption} occurs at a nonzero value $\lambda_0$.  

%Finally, we impose a symmetry condition on the linearized problem at $(0,0)$:
%\begin{equation}
%  \F_u(0,0) \textrm{ commutes with reflections in $x$.}  \label{F symmetry assumption} 
%\end{equation}
%Note that since the coefficients in \eqref{main elliptic PDE} are independent of $x$, the full nonlinear operator $\F(\placeholder, \lambda)$ commutes  with axial translations.  

\begin{theorem}[Center manifold reduction] \label{reduction theorem} 
  Consider the quasilinear elliptic PDE \eqref{main elliptic PDE} posed on the infinite cylinder $\Omega$.  Assume that it has a family of trivial solutions \eqref{F trivial solution assumption}, its linearization satisfies \eqref{F symmetry assumption},  and that $\lambda = 0$ is a critical parameter value in that the corresponding transversal linearized problem has the spectral behavior \eqref{lambda0 assumption}.  Fix $\mu \in (0,\sqrt{|\nu_1|/2})$ and an integer $M \ge 2$. Then there exist neighborhoods $U \sub  C_\bdd^{2+\alpha}(\overline{\Omega}) \times \mathbb{R}$ and $V \sub  \mathbb{R}^3$ of the origin and a coordinate map $\Psi = \Psi(A,B,\lambda)$ satisfying
  \begin{align} \label{Psi flatness}
    \Psi \in C^{M+1}(\mathbb{R}^3 , C_\mu^{2+\alpha}(\overline{\Omega})),
    \qquad \Psi(0,0,\lambda) = \Psi_A(0,0,\lambda) = \Psi_B(0,0,\lambda) = 0 \text{ for all $\lambda$},
    % Miles: actually we're going to assume C^{M+1} in the next Theorem...
  \end{align}
  such that the following hold.
  \begin{enumerate}[label=\rm(\alph*)]
  \item Suppose that $(u, \lambda) \in U$ solves \eqref{main elliptic PDE}.  Then $v(x) := u(x,0)$ solves the second-order ODE 
    \begin{equation}
      \label{reduced ODE} v'' = f(v, v', \lambda)
    \end{equation}
    where $f \maps \R^3 \to \R$ is the $C^{M+1}$ mapping
    \begin{equation}
      f(A,B,\lambda) :=  \frac {d^2}{dx^2}\bigg|_{x=0} \Psi(A, B, \lambda)(x,0).
      \label{reduced f definition}
    \end{equation}
  \item \label{reduction theorem recovery} Conversely, if $v \colon \mathbb{R} \to \mathbb{R}$ satisfies the ODE \eqref{reduced ODE} and $(v(x), v'(x), \lambda) \in V$ for all $x$, then $v = u(\placeholder, 0)$ for a solution $(u,\lambda) \in U$ of the PDE \eqref{main elliptic PDE}.  Moreover, 
    \begin{equation}\label{FS expression} 
      u(x+\tau, y) = \frac{v(x)}{\varphi_0(0)} \varphi_0(y) + \frac{v'(x)}{\varphi_0(0)} \tau \varphi_0(y) + \Psi(v(x), v'(x), \lambda)(\tau,y),
    \end{equation}
    for all $\tau \in \mathbb{R}$.  Here, recall that $\varphi_0$ generates the kernel of $L^\prime$.  
  \end{enumerate}
\end{theorem}

\begin{remark}\label{Phi remark}
  Setting $\tau=0$ in \eqref{FS expression} and normalizing $\varphi_0$ so that $\varphi_0(0)=1$, we obtain
  \begin{equation*}%\label{FS expression simple}
    u(x, y) = v(x) \varphi_0(y) + \Phi(v(x),v'(x),\lambda,y),
  \end{equation*}
  where $\Phi \in C^{2+\alpha}(\R^4,\R)$ is given by
  \begin{align}\label{Phi asymptotics}
    \Phi(A,B,\lambda,y) := \Psi(A, B, \lambda)(0,y) = O\Big((\abs A + \abs B)(\abs A + \abs B + \abs\lambda)\Big).
  \end{align}
\end{remark}

%As remarked upon above, Theorem~\ref{reduction theorem} is a melange of the center manifold theory of Amick--Turner \cite{amick1989small,amick1994center}, the more recent work of Faye--Scheel \cite{faye2016center}, and several new ingredients.  
Let us draw attention again to the fact that the ODE \eqref{reduced ODE} relates in a transparent way to the original PDE \eqref{main elliptic PDE}.  For example, when studying free boundary problems, we may pick coordinates on $\base$ so that the graph of $v$ parametrizes the interface.  

Another advantage of our approach --- which it inherits from Faye--Scheel --- is the comparative simplicity of deriving the reduced equation.  This can be seen in the next result, which says essentially that $\Psi$ in \eqref{Psi flatness} and $f$ in \eqref{reduced ODE} can be determined through a na\"ive formal asymptotic expansion.  

\begin{theorem}[Reduced equation] \label{expansion theorem} 
  In the setting of Theorem~\ref{reduction theorem}, the coordinate map $\Psi$ admits the Taylor expansion
  \begin{equation}
    \Psi(A,B, \lambda) = \sum_{\substack{2 \le i+j+k \leq M \\ i + j \ge 1}} \Psi_{ijk} A^i B^j \lambda^k + O\left( (|A| + |B|) (|A| + |B| + |\lambda|)^M \right) \qquad \mathrm{in}~ C_\mu^{2+\alpha}(\overline{\Omega}),\label{Psi taylor expansion} 
  \end{equation}
  where the coefficients $\Psi_{ijk}$ are the unique functions in $C_\mu^{2+\alpha}(\overline{\Omega})$ that satisfy 
  \begin{enumerate}[label=\rm(\roman*)]
  \item \label{expansion theorem projection} $\Psi_{ijk}(0,0) = \partial_x \Psi_{ijk}(0,0) = 0$.
  \item \label{expansion theorem Gateaux} For all $i+j+k \le M$, the formal G\^{a}teaux derivative 
    \begin{equation}\label{Gateaux derivative}
      \partial^i_A \partial^j_B \partial^k_\lambda \Big|_{(A,B,\lambda) = 0} \F \left(A\varphi_0 + Bx \varphi_0 + \Psi(A,B,\lambda) \right) = 0.
    \end{equation}
  \end{enumerate}
\end{theorem}
\begin{remark}\label{Gateaux remark}
  By introducing an appropriate cut-off function, we may consider the G\^{a}teaux derivative of $\F$ in \eqref{Gateaux derivative} as the Fr\'echet derivative of a modified $\F$. In practice, however, this distinction is unimportant when using \eqref{Gateaux derivative} to calculate the $\Psi_{ijk}$. Further details can be found in Lemma~\ref{lem cutoff F} and Section~\ref{general strategy sec}. 
\end{remark}
\begin{remark}\label{projection remark}
  As mentioned in the introduction, we actually have considerable freedom in choosing the linear relationship $v = \mathcal V u$ between the original unknown $u$ and the quantity $v$ governed by the reduced ODE \eqref{reduced ODE} in Theorem~\ref{reduction theorem}. Like Faye and Scheel \cite{faye2016center}, we have found pointwise evaluation $\mathcal V u(x) := u(x,0)$ to be the most convenient for calculations, but our proofs also apply to, for instance,
  \begin{equation}
    \label{choice for v}
    \mathcal Vu(x) := \int_\base u(x,y)\, dy 
    \quad 
    \text{or}
    \quad
    \mathcal Vu(x) := \int_{x-1}^{x+1}\int_\base u(s,y)\, dy\, ds. 
  \end{equation}
  Besides slightly alterning the very final step in the proof of Theorem~\ref{expansion theorem} in Section~\ref{proof of center manifold section}, the only other modification is that Theorem~\ref{expansion theorem}~\ref{expansion theorem projection} becomes $\FSproj \Psi_{ijk} = 0$, where the operator 
  \begin{align*}
    \FSproj w(x,y) := \frac{\mathcal V u(0)}{\varphi_0(y)} \varphi_0(y)
    + \frac{\mathcal Vu'(0)}{\varphi_0(y)} x \varphi_0(y)
  \end{align*}
  is a bounded projection from $C^{2+\alpha}_\mu(\overline \Omega)$ onto the kernel of $L$, here thought of as a mapping between weighted H\"older spaces.
\end{remark}

We also obtain the following theorem relating the linearized problem at any small non-trivial solution of the PDE \eqref{main elliptic PDE} to the linearization of the reduced ODE \eqref{reduced ODE}.

\begin{theorem}[Linearization and reduction] \label{nondegeneracy theorem} In the setting of Theorem~\ref{reduction theorem}~\ref{reduction theorem recovery}, if $\dot u \in C_\bdd^{2+\alpha}(\overline{\Omega})$ is a solution to the linearized PDE 
\[ \F_u(u,\lambda) \dot u = 0, \]
then $\dot v := \dot u(\placeholder, 0)$ satisfies the linearized reduced ODE
\begin{equation}
  \dot v^{\prime\prime} = f_{(A,B)}(v, v^\prime,\lambda) \cdot \left( \dot v, \dot v^\prime\right).
  \label{linearized reduced ODE}
\end{equation}
\end{theorem}

%This is an extremely important fact for global bifurcation theory, which relies crucially on compactness properties of the linearized operator.  In particular, one must have that $\F_u(u,\lambda)$ is Fredholm index $0$ for all nontrivial  $(u,\lambda)$.  This is quite sensitive as $\F_u(0,0)$ is not even Fredholm.  
The above theorem allows us to, among other things, calculate the dimension of the kernel of $\F_u(u,\lambda)$ using only information about the planar system \eqref{linearized reduced ODE}.  Indeed, Theorem~\ref{nondegeneracy theorem} tells us that the linearizations of the PDE and reduced ODE are compatible in that uniqueness of bounded solutions to the latter implies invertibility properties for the former.

Analogous results to Theorem~\ref{nondegeneracy theorem} can be found in \cite[Theorem 4.1(ii)]{wheeler2013solitary} and \cite[Theorem 5.1(ii)]{chen2018existence}, for example.  There the authors must carefully linearize each step in the center manifold construction.  By contrast,  our proof of Theorem~\ref{nondegeneracy theorem} relies on a soft analysis argument that avoids this rather tedious process through an extension of Theorem~\ref{reduction theorem} to diagonal elliptic systems.

In the remainder of the paper, we use Theorem~\ref{reduction theorem} and Theorem~\ref{expansion theorem} to  construct homoclinic and heteroclinic solutions to three quasilinear elliptic problems arising in quite different physical settings.     This includes anti-plane shear equilibria for a nonlinear elastic model with live body forces, and slow-moving invasion fronts for a two-dimensional Fisher--KPP equation with reactive boundary conditions.  To keep the presentation here compact, we defer stating these results and discussing the relevant history until later.

Our last application is to water waves.  Specifically, we study a system consisting of two incompressible fluids at constant density governed by the Euler equations.  They are separated by a free boundary and confined to a infinitely long horizontal channel. Steady traveling solutions to this problem are often referred to as \emph{internal waves}, and they are observed frequently in coastal flows \cite{perry1965large}.  We prove the existence of several families of front-type internal waves, which in hydrodynamics are known as (smooth) bores.   From a physical standpoint, bores are interesting because they are a genuinely stratified phenomenon:  one can show that no bores exist in constant density fluids~\cite{wheeler2013solitary}.  As heteroclinic connections, they also require considerably more finesse to construct. 

 Numerical studies of bores have been carried out by a number of authors \cite{turner1988broadening,lamb1998conjugate,grue2000breaking}, but very few rigorous results are currently available.  The earliest work is due to Amick and Turner \cite{amick1989small}, who used a precursor to the center manifold reduction in \cite{amick1994center} to characterize all small bounded solutions to the system assuming the flow in each layer is irrotational.   Later, Mielke \cite{mielke1995homoclinic} obtained an analogous result by applying traditional spatial dynamics techniques.   Using direct fixed point arguments, Makarenko  \cite{makarenko1992bore} gave an alternative construction for small-amplitude bores in the same setting, and later studied the continuously stratified case \cite{makarenko1999conjugate}.  

We not only prove the existence of irrotational bores, but in addition allow constant vorticity in the upper layer.  In the latter case, many of these waves will have \emph{critical layers} --- a line of particles in the fluid that are moving with the same horizontal speed as the wave itself.  It is well-known that this can create interesting streamline patterns, such as the famous cat's eyes in the periodic setting~\cite{wahlen2009critical,ehrnstrom2010multiple,constantin2016critical,ribeiro2017beneath}.  We find many families of waves feature a striking ``half cat's eye''; see Figure~\ref{fig:cats} and Theorem~\ref{cats eye theorem}.  To the best of our knowledge, this configuration has never been observed before.  Indeed, it is commonly thought that surface solitary waves in constant density water can \emph{never} have critical layers. Our results show that the heuristic fails for internal fronts.  
\begin{figure}%[hb]
  \centering
  \hspace*{3em} % to make it look ``centered'', obviously needs to be tweaked as margins change
  \includegraphics[scale=1.1]{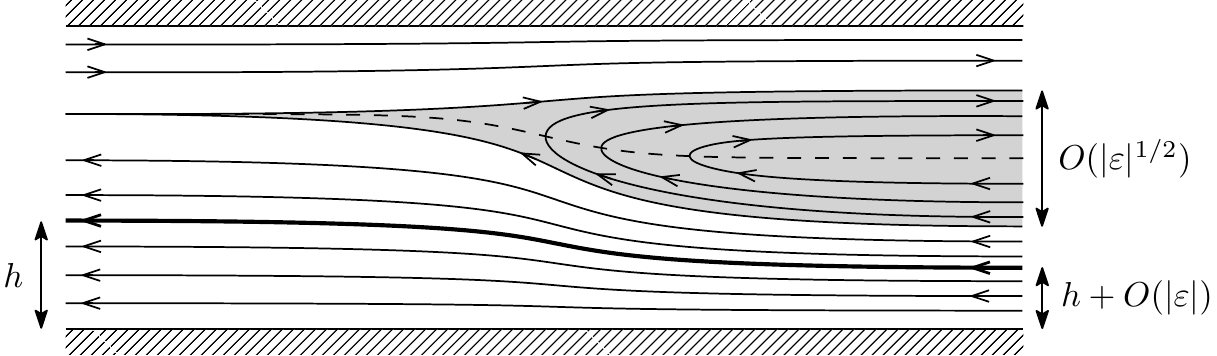}
  \caption{A smooth bore with a ``half cat's eye'' streamline pattern. The two fluid regions are bounded above and below by rigid walls and separated by a sharp interface (shown in bold). The dashed curve is the critical layer, above which particles move to the right and below which they move to the left (in the moving frame). Inside the shaded region (the ``eye''), the streamlines are bounded from the left and unbounded to the right, whereas outside they are unbounded in both directions.}
  \label{fig:cats}
\end{figure}

The analysis required for the water wave problem is several orders of magnitude more involved than the previous two examples.  It is here that the elegance of the expansion in Theorem~\ref{expansion theorem} and the choice of projection in the definition of the manifold are exploited most fully.  For instance, we are able to give a very simple proof that the free surface is monotonically decreasing and the streamlines have the expected pattern.  

An essential part of each of the above problems is identifying a parameter regime that admits front-type solutions.  For elasticity, we are able to exploit symmetry properties of the equation, whereas for the Fisher--KPP we take advantage of the robustness of the well-studied one-dimensional model.  Neither of these simplifications are available for water waves.  Instead, we make strong use of the theory of \emph{conjugate flows}; see Section~\ref{sec:conjugate}.  

\subsection{Plan of the article}

The proofs of Theorem~\ref{reduction theorem} and Theorem~\ref{expansion theorem} are carried out in Section~\ref{general theory section}.  First, in Section~\ref{linear theory section}, we establish some basic facts regarding the linear elliptic operator $L = \F_u(0,0)$.  Then, in Section~\ref{reformulation fixed point section}, the PDE is rewritten as a fixed point problem in the style of Amick--Turner.  Over the course of Section~\ref{at hypothesis section} and Section~\ref{truncation section}, we verify the hypotheses of that general theory, which yields a center manifold, but does not directly furnish the reduced equation for $v$ in \eqref{reduced ODE}.  In Section~\ref{proof of center manifold section}, we complete the proof using a near-identity change of variables to convert locally to the Faye--Scheel formulation, which gives us the liberty to choose the projection in the definition of the manifold, and also leads to the Taylor expansion \eqref{Psi taylor expansion}. 

For the benefit of the reader, Section~\ref{general strategy sec} contains a gentle explanation of the general strategy for actually computing the reduced equation and finding heteroclinic or homoclinic solutions.  While this is in principle deducible from \eqref{reduced ODE} and \eqref{Psi taylor expansion}, there are certain choices that are not immediately obvious but greatly simplify the process.  

In Section~\ref{extensions section}, we consider a number of extensions of Theorem~\ref{reduction theorem} and Theorem~\ref{expansion theorem} to other types of elliptic equations.  We also provide the proof of Theorem~\ref{nondegeneracy theorem}.  

The application to nonlinear elasticity can be found in Section~\ref{sec ap}, while Section~\ref{sec fkpp} contains our results on invasion fronts in two-dimensional Fisher--KPP.  We devote Section~\ref{sec ww} to proving the existence of internal bores with vorticity.  

Finally, two appendices are included.  Appendix~\ref{at appendix} provides a brief statement of Amick and Turner's fixed point theory that is sufficient for proving Theorem~\ref{reduction theorem}.   In Appendix~\ref{calculation appendix}, we collect some further details regarding the calculation of the reduced equations in the first two applications.  

%Analogous results to Theorem~\ref{nondegeneracy theorem} can be found in \cite[Theorem 4.1(ii)]{wheeler2013solitary} and \cite[Theorem 5.1(ii)]{chen2018existence}, for example.  There the authors must carefully linearize each step in the center manifold construction.  By contrast,  our proof of Theorem~\ref{nondegeneracy theorem} relies on a soft analysis argument that avoids this rather tedious process through an extension of Theorem~\ref{reduction theorem} to diagonal elliptic systems.  The proof is given in Section~\ref{extensions section}.  

\section{Center manifolds for quasilinear elliptic PDE on a cylinder} \label{general theory section}

Let us begin by fixing some notation.  Recalling that $\alpha \in (0,1)$ is the H\"older exponent introduced earlier, we set
\[ \Xspace := C^{2+\alpha}(\overline{\Omega}), \qquad \Yspace = \Yspace_1 \times \Yspace_2 := C^{0+\alpha}(\overline{\Omega}) \times C^{1+\alpha}(\partial \Omega). \]
Note that these are spaces of functions which are only \emph{locally} H\"older up to the boundary; the corresponding spaces of uniformly bounded functions will be designated with a subscript $\bdd$.  Likewise, for $\mu \geq 0$, we write $\Xspace_\mu$ and $\Yspace_\mu$ to indicate the associated exponentially weighted H\"older spaces.  

\subsection{Linear theory} \label{linear theory section} 
Recall that
\begin{align*} 
  \F_1(u, \lambda) &:= \nabla \cdot \A(y, u, \nabla u, \lambda)  + \B(y, u, \nabla u, \lambda), \\ 
  \F_2(u, \lambda) & :=  \big( \G(y, u, \nabla u, \lambda) \big)\big|_{\partial\Omega}. 
\end{align*} 
From this it is straightforward to compute that the linearized operator at $(u,\lambda) = (0,0)$ is given by 
\begin{equation}
 \label{linear form hypothesis}
  \begin{aligned}  
    \F_{1u}(0, 0)v &= \partial_x^2v + \nabla^\prime \cdot \left( a^\prime(y) \nabla^\prime v + b(y) v \right) - b(y) \cdot \nabla^\prime v + c(y) v, \\ 
    \F_{2u}(0,0) v  &= \Bigl(- N^\prime(y) \cdot \big( a^\prime(y) \nabla^\prime  v + b(y) v \big) +  g(y) v \Big)\Big|_{\partial \base}. 
  \end{aligned} 
\end{equation}
Here, $\nabla^\prime$ and $\nabla^\prime \cdot$ indicate the gradient and divergence in $y$, respectively.   As required by assumption \eqref{F symmetry assumption}, the linearized boundary condition is co-normal.   The coefficients are all of class $C_\bdd^{M+3}(\overline{\base})$, and related to the nonlinear problem via
\[ a^\prime(y) := \A_{p^\prime}(y,0, 0,0), \quad 
 b(y) :=  \B_{p^\prime}(y,0,0,0), \quad
c(y) := \B_z(y, 0, 0, 0), \quad 
g(y) := \G_z(y,0,0,0),\]
where we are adopting the notation $p = (p_1, p^\prime) \in \mathbb{R} \times \mathbb{R}^{n-1}$.  Without loss of generality, the coefficients have been normalized so that $a_{11} \equiv 1$.  The simplicity of the linearized operator is owed to the formal self-adjointness assumption \eqref{F symmetry assumption}.
%\begin{equation}
%  \label{explicit reflection assumption}
%  \text{at $\lambda = 0$, $\A_1$ is odd in $p_1$, $\B$ and $\A_j$ are even in $p_1$ for $j = 2,\ldots,n$.}
%\end{equation}
%Thus, many of the terms in $\F_u(0,0)$ to be zero, and in particular, $a := \A_p$ has the block diagonal form evident above.

From \eqref{linear form hypothesis}, we also see that the transversal linear operator $L^\prime = (L_1^\prime, L_2^\prime)$ has the explicit form 
\begin{equation}
 \label{transverse problem}
  \begin{aligned}
    L_1^\prime w &:=  \nabla^\prime \cdot \left( a^\prime(y) \nabla^\prime w + b(y) w\right) - b(y) \cdot \nabla^\prime w+ c(y) w, \\
    L_2^\prime w & := \left( -N^\prime(y) \cdot \left( a^\prime(y) \nabla^\prime w + b(y) w \right) + g(y) w \right)\Big|_{\partial \base}, 
  \end{aligned}
\end{equation}
for all $w = w(y) \in C^{2+\alpha}(\overline{\base})$.  
\subsubsection*{Projections}

Under the spectral assumption \eqref{lambda0 assumption}, there exists a continuous orthogonal projection $P_k^\prime$ onto the eigenspace corresponding to $\varphi_k$.  That is, any function $w = w(y)$ in $C^{0+\alpha}(\overline{\base})$ admits the unique representation
\[ w(y) = \sum_{k=0}^\infty \hat{w}_k \varphi_k(y) =: \sum_{k=0}^\infty (P_k^\prime w)(y), \qquad \hat{w}_k := \left(w, \varphi_k \right)_{L^2(\Omega')}.\]
It is also important to mention that there is a variational characterization of our spectral assumption \eqref{lambda0 assumption} in terms of the Rayleigh quotient defined by  
\[ \mathcal{R}(w) := \frac{ \int_{\base} \left( -a(y) \nabla^\prime w \cdot \nabla^\prime w + c(y) w^2 \right) \, dy + \int_{\partial \base} g(y) w^2 \, dS(y)}{ \int_{\base} w^2 \, dy },\]
for all $w \in H^1(\base) \setminus \{0\}$.  It is easy to verify that any critical point of $\mathcal{R}$ is in the kernel of $L^\prime$.  By classical elliptic theory, we have further that
\begin{equation}
  \max_{\substack{w \in H^1(\base) \\ w \not\equiv 0}} \mathcal{R}(w) = \nu_0 = 0, \qquad \max_{\substack{w \in P^\prime_{\geq 1} H^1(\base) \\ w \not\equiv 0}} \mathcal{R}(w) = \nu_1 < 0,  \label{variational characterization} 
\end{equation}
where $P_{\geq 1}^\prime := 1 - P_0^\prime$.

We can therefore introduce a projection $P_k$ defined on functions $u = u(x,y)$ by 
\[ (P_k u)(x,y) := P_k^\prime u(x, \cdot).\]
Arguing as in \cite[Chapter 9, Lemma 3.4]{volpert2011book1}, one can confirm that $P_k$ is a bounded projection on $C^{0+\alpha}(\overline{\Omega})$, and that we are justified in writing
\begin{equation}
  u(x,y) = \sum_{k=0}^\infty (P_k u)(x,y) =: \sum_{k = 0}^\infty \hat u_k(x) \varphi_k(y).\label{eigenfunction expansion} 
\end{equation}
 Continuing the above convention, let $P_0$ denote the projection point-wise in $x$ onto the $0$ eigenvalue for the transverse problem \eqref{transverse problem} and set 
\[ P_{\geq 1} := 1-P_0 = \sum_{k \geq 1} P_k.\]

\subsubsection*{Boundedness of the partial Green's function}

%Since we are interested in solutions that are small and near critical, we introduce the linear operator $L := \F_u(0, 0)$.  
In this section, we seek to understand the solvability properties of the linearized problem $Lu = f$ for $f$ lying in an exponentially weighted H\"older space.  

By hypothesis, $L^\prime$ has a kernel.  Consider next the kernel of $L|_{\Xspace_\mu}$ for $\mu \in (0, \sqrt{|\nu_1|})$.  Suppose that $L u = 0$ for some $u \in \Xspace_\mu$.  It follows from the representation formula \eqref{eigenfunction expansion} that $\hat u_k$ satisfies the ODE
\[ \partial_{x}^2 \hat u_k = -\nu_k \hat u_k, \qquad \textrm{for all } k \geq 0. \]
Recalling that $\nu_k < 0$ for $k \geq 1$, this ensures that $\hat u_k$ grows exponentially as $x \to \infty$ with rate $\sqrt{|\nu_k|}$.  Thus, when $\mu \in (0,\sqrt{|\nu_1|})$, it must be that $\hat u_k \equiv 0$ for $k \geq 1$, and hence the kernel of $L|_{\Xspace_\mu}$ must lie in $P_0 \Xspace_\mu$.  Simply inspecting the operator $L$, we see immediately that $u$ must then take the form 
\[ u(x,y) = A \varphi_0(y) +  B x \varphi_0(y),\]
for some $A$, $B \in \mathbb{R}$.

In summary, we have proved the following.

\begin{lemma}[Kernel]\label{lem kernel}  For all $\mu \in (0, \sqrt{|\nu_1|})$, 
\[ \kernel{L|_{\Xspace_\mu}} = \left\{  (x,y) \mapsto A \varphi_0(y) + B x \varphi_0(y) :  A, B \in \mathbb{R} \right\}.\]
\end{lemma}

One consequence of the above lemma is that composing with the projection $P_{\geq 1}$ eliminates the kernel of $L|_{\Xspace_\mu}$.  Before considering the inhomogeneous problem for $L$, it will therefore be useful to define a projection on $\Yspace$ (which is then inherited by $\Yspace_\mu$) that agrees in a natural way with $P_0$.  For  $v = (v_1, v_2) \in \Yspace$, let   
\[ (Q_0 v)(x,y) := \left( (P_0 v_1)(x,y) + \varphi_0(y) \int_{\partial \base} v_2(x,s) \varphi_0(s) \, dS(s),  \, 0 \right).\]
Thus $Q_0 \Yspace \subset (P_0 \Yspace_1) \times \{0\} \subset \Yspace$, and, for any $u \in \Xspace$,
\begin{align*}
Q_0 Lu & = Q_0( L_1 u   , \, L_2 u) = \left( P_0 L_1 u + \varphi_0(y) \int_{\partial \base} (L_2 u)(x, s) \varphi_0(s) \, dS(s), \, 0  \right).
\end{align*}
But, recalling the definition of $P_0$, we have 
\begin{align*} P_0 L_1 u &= P_0 \partial_x^2 u + P_0 L^\prime u = \varphi_0(y) \left( \int_{\base} (\partial_x^2 u)(x,s) \varphi_0(s) \, ds + \int_{\base} (L^\prime u)(x,s) \varphi_0(s) \, ds \right) \\
& = \varphi_0(y) \left( \int_{\base} (\partial_x^2 u)(x,s) \varphi_0(s) \,ds - \int_{\partial \base} (L_2 u)(x,s) \varphi_0(s) \, dS(s) \right).
\end{align*}
Combining this with the line above yields
\begin{equation}
  Q_0 L u = \left( \varphi_0(y) \int_{\base} (\partial_x^2 u)(x,s)\varphi_0(s) \, ds , \, 0 \right) = L P_0 u. \label{QL PL identity} 
\end{equation}
In keeping with the notation above, let $Q_{\geq 1} \colon \Yspace \to \Yspace$ be defined by 
\[ Q_{\geq 1}v := (1-Q_0) v \qquad \textrm{for all } v  \in \Yspace.\]

With $Q$ and $Q_{\geq 1}$ in hand, we now establish the following ``elementary'' fact about the solvability of $Lu = f$ when the data $f \in Q_{\geq 1} \Yspace_\mu$.

\begin{proposition}[Partial Green's function] \label{partial green lemma} For any $\mu \in [0, \sqrt{|\nu_1|/2})$ and $f = (f_1, f_2) \in \Yspace_\mu$ such that $P_{0} f_1 = 0$, there exists a unique $u \in P_{\geq 1} \Xspace_\mu$ such that $Lu = f$.  Moreover, 
\[ \| u \|_{\Xspace_\mu} \lesssim \| f \|_{\Yspace_\mu}, \]
with the implied constant above uniform in $\mu$ as $\mu \to 0$.  Equivalently, for all $\mu \in [0, \sqrt{|\nu_1|/2})$, 
\[ L|_{P_{\geq 1} \Xspace_\mu} \colon P_{\geq 1} \Xspace_\mu \to Q_{\geq 1} \Yspace_\mu \]
is invertible with bounded inverse $G \colon Q_{\geq 1} \Yspace_{\mu}  \to P_{\geq 1} \Xspace_\mu$ that we call the \emph{partial Green's function}.
\end{proposition}

\begin{proof}  Fix $\mu$ as above and let $f \in \Yspace_\mu$ be given with $P_{0} f_1 = 0$.  Following the general strategy of \cite[Theorem 3.1]{amick1989small}, we introduce a smooth partition of unity $\{ \zeta^{(m)} \}_{m \in \mathbb{Z}}$ on $\mathbb{R}$ such that  
\[ \zeta \in C_{\even}^\infty(\mathbb{R}), \quad \supp{\zeta} \subset [-2,2], \quad \zeta = 1 \textrm{ on } [-1,1], \quad \zeta - \tfrac{1}{2} \textrm{ odd about } x=\tfrac{3}{2} \textrm{ on } [1,2 ],\]
and taking $\zeta^{(m)} := \zeta(\placeholder - 3m)$.  

For each $m \in \mathbb{Z}$, consider the cut-off problem 
  \begin{equation}
    L u^{(m)} = f^{(m)} := \zeta^{(m)} f.\label{cut-off PDE} 
  \end{equation}
Observe that, because the projectors are pointwise in $x$, 
\[ P_0 f_1^{(m)} = \zeta^{(m)} P_0 f_1 = 0.\]
Thus, $f^{(m)} \in Q_{\geq 1} \Yspace_\mu$, and the commutation identity \eqref{QL PL identity} implies that any solution $u^{(m)}$ of \eqref{cut-off PDE} necessarily lies in $P_{\geq 1} \Xspace_\mu$.  

As a starting point, we show that there exists weak solutions to \eqref{cut-off PDE} by introducing the Hilbert space $\Hspace := P_{\geq 1} H^1(\Omega)$ endowed with the standard $H^1$ inner product.  Letting $\mathscr{B}$ be the bilinear form associated to \eqref{cut-off PDE},
\begin{align*} \mathscr{B}[u,v] &:= \int_{\Omega} \left( -a  \nabla u \cdot \nabla v -  u (b \cdot \nabla^\prime v) + v (b\cdot \nabla^\prime u)  + c u v \right) \, dx \, dy + \int_{\partial \Omega} g u v \, dS, \end{align*}
for all $u, v \in \Hspace$, 
the weak formulation of \eqref{cut-off PDE} is 
\[ \mathscr{B}[u^{(m)}, \psi] = ( f_1^{(m)}, \psi )_{L^2(\Omega)} + (f_2^{(m)}, \psi )_{L^2(\partial \Omega)} \qquad \textrm{for all } \psi \in \Hspace. \]
Notice that, because $f^{(m)}$ is compactly supported, the right-hand side above does indeed represent an element of $\Hspace^*$ acting on $\psi$.   As the coefficients are $C_\bdd^{M+3}$, it is obvious that $\mathscr{B}$ is bounded. On the other hand, 
\begin{align*}
\mathscr{B}[u, u] &= \sum_{j, k = 1}^\infty \mathscr{B}[\hat u_j \varphi_j, \, \hat u_k \varphi_k] = -\sum_{k = 1}^\infty \nu_k^2 \|\hat u_k\|^2_{L^2(\mathbb{R})}.
\end{align*}
It follows from the variational characterization of $\nu_0 = 0$ in \eqref{variational characterization} that $\mathscr B$ is coercive on $\Hspace$, and thus Lax--Milgram implies that there exists a weak solution $u^{(m)} \in \Hspace$ to the cut-off problem \eqref{cut-off PDE} for each $m \in \mathbb{Z}$.   

We must now improve this to classical solutions and estimate their norm in $\Xspace_\mu$.  Let an integer $\ell \in \mathbb{Z}$ be given and put 
\[ \Xspace^{(\ell)}_\mu := C_\mu^{2+\alpha}([\ell, \ell+1] \times \overline{\base})  , \qquad \Yspace_\mu^{(\ell)} := C_\mu^{0+\alpha}([\ell, \ell+1] \times \overline{\base}) \times C_\mu^{1+\alpha}([\ell, \ell+1] \times \partial\overline{\base}).\]
The next stage of the argument involves deriving a priori estimates for $u^{(m)}$ in $\Xspace_\mu^{(\ell)}$.   This will follow by elliptic regularity theory, but first we must expand the class of admissible test  functions to all of $H^1(\Omega)$.  In particular, observe that if $\psi \in C_c^1(\overline{\Omega})$, we may use the splitting above to write 
\[ \psi = \hat \psi_0(x) \varphi_0(y) + \psi_{\geq 1} \in P_0 C_c^1(\overline{\Omega}) \oplus P_{\geq 1}C_c^1(\overline{\Omega}).\]
It is easy to verify that $\mathscr{B}[u^{(m)}, \placeholder]$ extends to a bounded linear functional on $C_c^1(\overline{\Omega})$, and indeed 
\begin{align*}
\mathscr{B}[u^{(m)}, \psi] &= \mathscr{B}[u^{(m)}, \psi_{\geq 1} ] = (f_1^{(m)}, \psi_{\geq 1})_{L^2(\Omega)} + ( f_2^{(m)}, \psi)_{L^2(\partial \Omega)} \\
& =  (f_1^{(m)}, \psi )_{L^2(\Omega)} + ( f_2^{(m)}, \psi)_{L^2(\partial \Omega)} - (f_1^{(m)}, \hat \psi_0 \varphi_0)_{L^2(\Omega)} \\
& = (f_1^{(m)}, \psi )_{L^2(\Omega)} + ( f_2^{(m)}, \psi)_{L^2(\partial \Omega)},
\end{align*}
since $P_0 f_1^{(m)} = 0$, by hypothesis.  Thus $u^{(m)}$ is a weak solution of \eqref{cut-off PDE} in the $H^1$ sense.

Let $\Omega^{(\ell)} := [\ell-1/4, \ell+5/4] \times \base$, which is a slight enlargement of the domain associated to $\Xspace_\mu^{(\ell)}$.  There are precisely two integers $m$ for which $\supp{\zeta^{(m)}}$ and $\Omega^{(\ell)}$ have non-empty intersection; as they are consecutive, let us call them $\tilde m$ and $\tilde m+1$. Conjugating \eqref{cut-off PDE} with the exponential weight $\sech{(\mu x)}$, and applying standard elliptic regularity theory on bounded domains (see, for example, \cite[Chapter 8]{gilbarg2001elliptic}), we infer that $u^{(m)} \in \Xspace_\mu^{(\ell)}$.  Moreover, it obeys the bound
\begin{equation}
 \| u^{(m)} \|_{\Xspace_\mu^{(\ell)}} \lesssim \left\{ 
  \begin{aligned} 
    &e^{-\mu |\ell| } \| u^{(m)} \|_{L^2(\Omega^{(\ell)})} & &\qquad \textrm{for } m \in \mathbb{Z} \setminus \{ \tilde m , \tilde m +1\} \\
  &e^{-\mu |\ell|} \| u^{(m)} \|_{L^2(\Omega^{(\ell)})} + \|f\|_{\Yspace_\mu} & &\qquad \textrm{for } m = \tilde m, \, \tilde m+1. \end{aligned} \label{u^m a priori bound} \right. 
\end{equation}
 Here, the constants are uniform in $\mu$, $\ell$, and $m$.  In order to complete the argument we must justify the convergence of the series $\sum_m u^{(m)}$ in $\Xspace_\mu$.  Looking at \eqref{u^m a priori bound}, it is apparent that this hinges on having sufficiently refined bounds on the $L^2(\Omega^{(\ell)})$ norm of $u^{(m)}$. 
 
  First, suppose that $m < \tilde m$, so that $L u^{(m)} = 0$ on $\Omega^{(\ell)}$.  In fact, this holds on the semi-infinite strip $(3m+2,\infty) \times \base$, and so we may apply elliptic regularity again to conclude that $u^{(m)} \in C^{2+\alpha}_\bdd$ on this set.  By construction, $u^{(m)}$ is also in $H^1(\Omega)$, so in particular this also ensures that $u^{(m)}$ and $\nabla u^{(m)}$ decay to $0$ as $x \to +\infty$.  
 
  We are therefore justified in taking the equation $L u^{(m)} = 0$, multiplying by $u^{(m)}$, and then integrating over the strip $(x, \infty) \times \base$.  This procedure yields the identity
  \begin{align*} \frac{1}{2} \partial_x \int_{\base} |u^{(m)}(x,y) |^2 \, dy &= \int_x^\infty \!\!\! \int_{\base} a^\prime(y) \nabla^\prime u^{(m)}(s,y) \cdot \nabla^\prime u^{(m)}(s,y) \, dy \, ds \\
  & \qquad - \int_x^\infty \!\!\! \int_{\base} c(y) |u^{(m)}(s,y)|^2 \, dy \, ds \\
  & \qquad - \int_x^\infty \!\!\! \int_{\partial\base} g(y) |u^{(m)}(s,y)|^2 \, dS(y) \, ds,
  \end{align*}
  which holds for all $x > 3m+2$.  Using the Rayleigh--Ritz characterization of $\nu_1$ in \eqref{variational characterization}, this furnishes the integro-differential inequality 
  \[ \frac{1}{2} \partial_x \| u^{(m)}(x, \placeholder)\|_{L^2(\base)}^2 \leq \nu_1 \int_x^\infty \| u^{(m)}(s, \placeholder) \|_{L^2(\base)}^2 \, ds \qquad \textrm{for all } x > 3m+2.  \]
  From this, we may further estimate that
  \[ -\int_x^\infty \| u^{(m)}(s, \placeholder) \|_{L^2(\base)}^2 \, ds \leq e^{\sqrt{2|\nu_1|} (3m+2-x)} \int_{3m+2}^\infty \| u^{(m)}(s, \placeholder) \|_{L^2(\base)}^2 \, ds, \]
  for all $x > 3m+2$.  Thus,
  \begin{equation}
    \| u^{(m)} \|_{L^2(\Omega^{(\ell)})} = \left(\int_{\ell-\frac{1}{4}}^{\ell + \frac{5}{4}} \| u^{(m)}(s, \placeholder) \|_{L^2(\base)}^2 \, ds\right)^{1/2} \lesssim e^{\sqrt{|\nu_1|/2} (3m-\ell)} \| u^{(m)} \|_{L^2(\Omega)},  \label{first u^m L^2 estimate} 
  \end{equation}
  where recall we have assumed $m < \tilde m$.  Now, $u^{(m)}$ is bounded in $L^2(\Omega)$ in terms of the data $f^{(m)}$ via Lax--Milgram.  Relating this back to $f$, we find that
  \[ \| u^{(m)} \|_{L^2(\Omega)} \lesssim e^{3 \mu |m| } \| f \|_{\Yspace_\mu} \qquad \textrm{for all } m \in \mathbb{Z}.\]
 Combining this with \eqref{first u^m L^2 estimate} yields
 \[ \| u^{(m)} \|_{L^2(\Omega^{(\ell)})} \lesssim e^{\sqrt{|\nu_1|/2} (3m-\ell) + 3\mu|m|} \| f \|_{\Yspace_\mu} \qquad \textrm{for all } m < \tilde m. \]
 
 The same type of reasoning applied to the case $m > \tilde m+1$ gives a similar bound.  For the exceptional values $m = \tilde m, \tilde m + 1$, we may use \eqref{u^m a priori bound}, so that in total 
\[ \| u^{(m)} \|_{L^2(\Omega^{(\ell)})} \lesssim e^{-\sqrt{|\nu_1|/2} |3m-\ell| + 3\mu|m|} \| f \|_{\Yspace_\mu} \qquad \textrm{for all } m \in \mathbb{Z}.\]
Returning to the preliminary a priori estimate \eqref{u^m a priori bound}, we can now conclude that
\[ \sum_{m = -M}^M u^{(m)} \longrightarrow u \text{ in $\Xspace_\mu^{(\ell)}$ as $M \to \infty$} \qquad \text{for all $\ell \in \mathbb{Z}$},\]
and 
\[ \| u \|_{\Xspace_\mu^{(\ell)}} \leq C \| f \|_{\Yspace_\mu},  \]
with a constant $C$ that depends only on $ \sqrt{|\nu_1|/2} -\mu$.  As $\ell$ on the left-hand side above is arbitrary, this gives the desired bounds on $G$.  
\end{proof}

In applications, it will often be convenient to use alternative projections onto the kernel of $L$.  For instance, looking at the statement of Theorem~\ref{reduction theorem}, we see that the coefficients $A$ and $B$ are found by evaluating $u$ and $u_x$ at $(0,0)$.    With that in mind, suppose that $\FSproj$ is a given bounded projection from $\Xspace_\mu$ to $\kernel{L}$ which is independent of $\mu$.  As in the partial Green's function analysis, we expect that $L$ is invertible on the kernel of $\FSproj$.  To make this precise, we adopt the approach of Faye--Scheel \cite{faye2016center} and considered a so-called ``bordered" operator where one appends $\FSproj$ to $L$.  The result is the following.
  
\begin{lemma}[Bordered operator]\label{lem bordering}
The bordered operator 
\begin{equation*}
  %u \in \Xspace_\mu \longmapsto \left( Lu, \, \FSproj u \right) \in \Yspace_\mu \times \kernel{L}.
  (L, \FSproj):\ \Xspace_\mu \longrightarrow \Yspace_\mu \by \kernel{L}, \qquad u \longmapsto (Lu,\ \FSproj u)
  % \label{bordered operator definition}
\end{equation*}
is invertible with a bounded inverse.
\end{lemma}
\begin{proof}
From Lemma~\ref{lem kernel} and Proposition~\ref{partial green lemma}, we see that $L : \Xspace_\mu \to \Yspace_\mu$ has Fredholm index $2$.  A standard dimension counting argument shows that the bordered operator has Fredholm index $0$; see, for example, \cite[Lemma 4.4]{scheel2007morse}.  Now, if $u\in \Xspace_\mu$ satisfies $(Lu, \FSproj u)= 0$, then in particular $u \in \kernel L$. On the other hand, $\FSproj u = 0$, and so it must be that $u = 0$. Thus, the bordered operator is injective and Fredholm index $0$.  It follows that it is invertible, and the boundedness of its inverse is a consequence of Proposition~\ref{partial green lemma}.
\end{proof}

\subsection{Reformulation as a fixed point} \label{reformulation fixed point section}

Now, let us return to the full nonlinear problem.    The abstract operator equation \eqref{abstract nonlinear equation} can be rewritten as
\begin{equation}
  \label{nonlinear problem}
  L w = \nl(w, \lambda),  
\end{equation}
where $\nl = (\nl_1, \nl_2)$ is defined by
\[ \nl(w, \lambda) := L w - \F(w, \lambda).\]
Thus, $\nl$ is ``flat'' with respect to $(w, \lambda)$ in the sense that $\nl(0,0)=0$, $\nl_u(0, 0) = 0$, and $\nl_\lambda(0,0) = 0$.   The commutation identity \eqref{QL PL identity} permits us to perform a spectral decomposition in both the domain and codomain to rewrite \eqref{nonlinear problem} as the system
\begin{align*}
  \left\{
  \begin{aligned}
    L (P_0 u) &=  Q_0 \nl(P_0 u + P_{\ge 1} u,\lambda)\\
    L (P_{\ge 1} u) &=  Q_{\ge 1} \nl(P_0 u + P_{\ge 1} u,\lambda).
  \end{aligned}
  \right.
\end{align*}
Applying the partial Green's function $G$ of Proposition~\ref{partial green lemma} to the second equation then gives 
\begin{align*}
  P_{\ge 1} u =  G Q_{\ge 1} \nl(P_0 u + P_{\ge 1} u,\lambda),
\end{align*}
while, recalling the explicit form of $P_0$ and $L$,
we see that
\begin{align}
  \label{eqn:intme}
  \partial_x^2 P_0 u &= Q_0 \nl(P_0 u + P_{\ge 1} u,\lambda).
\end{align}
Integrating \eqref{eqn:intme} twice we get the full system
\begin{align*}
  \left\{
  \begin{aligned}
    P_0 u &= \xi_1 \varphi_0 + \int_0^x (Q_0 u)_x(s,y) \, ds \\
    P_0 u_x &= \xi_2 \varphi_0 + \int_0^x Q_0 \nl(P_0 u + P_{\ge 1} u,\lambda)(s,y) \, ds\\
    P_{\ge 1} u &=  G Q_{\ge 1} \nl(P_0 u + P_{\ge 1} u,\lambda).
  \end{aligned}
  \right.
\end{align*}
for some constants $\xi_1,\xi_2 \in \R$ (the ``initial data''). Introducing a parameter $\beta$ (representing a rescaling of the axial variable), defining
\begin{equation}
  \label{spectral splitting equation}
  (P_0 u)(x,y) =: U_1(x) \varphi_0(y),
  \qquad 
  (P_0 u_x)(x,y) =: \beta U_2(x) \varphi_0(y),
  \qquad 
  R := P_{\ge 1} u,
\end{equation}
and scaling $\xi_2$, 
we finally obtain the following integro-differential fixed-point equation in the spirit of Amick and Turner \cite{amick1994center}:
\begin{align}
  \label{eqn:fixed}
  \left\{
  \begin{aligned}
    U_1(x) &= \xi_1 + \beta \int_0^{x} U_2(s) \, ds \\
    U_2(x) &= \xi_2 + \frac 1\beta \int_0^x \!\!\! \int_{\base} \varphi_0(y) \nl_1(U_1\varphi_0 + R,\lambda)(s,y) \, dy\, ds \\
    & \qquad + \frac 1\beta \int_0^x \!\!\! \int_{\partial\base} \varphi_0(y) \nl_2(U_1\varphi_0 + R,\lambda)(s,y) \, dS(y)\, ds  \\
    R &=  G Q_{\ge 1} \nl(U_1\varphi_0 + R,\lambda).
  \end{aligned}
  \right.
\end{align}
In terms of regularity, we ultimately seek solutions of \eqref{eqn:fixed} with 
\[ (U_1,U_2,R) \in C_\bdd^{2+\alpha}(\mathbb{R}) \times C_\bdd^{1+\alpha}(\mathbb{R}) \times C_\bdd^{2+\alpha}(\overline{\Omega}) =: \atXspace_\bdd.\]
Unraveling definitions, this will imply that $u \in \Xspace_\bdd$.  In view of Lemma~\ref{partial green lemma}, define $\overline{\mu} := \sqrt{|\nu_1|/2}.$   In order to obtain a fixed point, we cannot work directly in $\atXspace_\bdd$, but must instead consider the problem posed in the corresponding exponentially weighted space $\atXspace_\mu$ for $\mu \in (0,\overline{\mu})$.  

%For later use we define semi-norms $\abs\placeholder_{\homX}$ and $\abs\placeholder_{\homY}$,
% which are analogous to the $X,Y$ norms above, except with respect to the homogeneous H\"older spaces $\homholder_\mu^{k+\alpha}$.

\subsection{Analysis of the nonlinear term} \label{at hypothesis section}
%We will need to truncate the nonlinear terms in \eqref{eqn:fixed}, and
%for this we need to write the operator $\nl$ more explicitly.
We wish to eventually apply the fixed-point theorem for systems of the type \eqref{eqn:fixed} given by Amick and Turner, which is recalled in Appendix~\ref{at appendix}.  Towards that end, it is necessary to look more closely at the form of the nonlinear terms $\nl$.  

Splitting into bulk and boundary operators as usual, we have
$\nl=(\nl_1,\nl_2)$ where
\begin{align*}
  \nl_1(u,\lambda) &= 
  \nabla \cdot \big(\A(y,u, \nabla u,\lambda)-\A_{p}(y,0,0)\nabla u\big)
  \\&\qquad 
  + \B(y,u,\nabla u,\lambda)-\B_z(y,0,0,0)u-\B_{p}(y,0,0,0)\nabla u,\\
  \nl_2(u,\lambda) &=  \big(\G(y,u,\nabla u,\lambda)-\G_{p}(y,0,0)\nabla u\big) + 
  \big(\G(y,u,\lambda)-\G_z(y,0,0)u\big).
\end{align*}
Substituting 
\begin{align*}
  u=U_1 \varphi_0 + R,
  \qquad 
  u_x = \beta U_2 \varphi_0 + R_x,
  \qquad 
  \nabla^\prime u = U_1 \nabla^\prime \varphi_{0} + \nabla^\prime R,
\end{align*}
we can rewrite this as 
\begin{align*}
  \nl_1 &= 
  \nabla \cdot \tilde\A\big(y,U, \nabla R,\lambda,\beta\big)
  + \tilde\B\big(y,U,R,\nabla R,\lambda,\beta\big),\\
  \nl_2 &=  \tilde\G\big(y,U,R,\nabla R, \lambda,\beta\big),
\end{align*}
for some functions $\tilde\A,\tilde\B,\tilde \G$ that are $C^{M+4}$ in all of their arguments.   Moreover, they are flat
with respect to $(U,R,\nabla R,\lambda)$ in that their Taylor expansions in these variables at the origin contain only quadratic and higher-order terms.  Note that here, and in the sequel, we write $U := (U_1,U_2)$ to shorten the equations.

With this in mind, let's analyze the terms in the fixed-point equation \eqref{eqn:fixed} and ensure they take the required form \eqref{at nonlinear terms form} for Theorem~\ref{at fixed point theorem}.    

\subsubsection*{Equation for \texorpdfstring{$U_1$}{U\_1}.}
Looking at \eqref{eqn:fixed}, we see the right-hand side of the equation for $U_1$ is $\beta \mathfrak L_1 (U,R)$ where $\mathfrak L_1(U,R) = \int_0^x U_2\, dx$. Since the operator $f \mapsto \int_0^x f \, ds$ is bounded and linear $C_\mu^{1+\alpha} \to C_\mu^{2+\alpha}$ and $C_\bdd^{1+\alpha} \to \mathring{C}_\bdd^{2+\alpha}$, $\mathfrak L_1$ satisfies the first component of \eqref{at A1}. In particular, $\atF_1$ in \eqref{at general equation} has the form \eqref{at nonlinear terms form} with $\mathfrak H_1 = 0$.

\subsubsection*{Equation for \texorpdfstring{$U_2$}{U\_2}.}
In the equation for $U_2 \in C^{1+\alpha}$ in \eqref{eqn:fixed}, consider the term
\begin{align*}
  \frac 1\beta \int_0^x \!\!\! \int_{\base} \varphi_0 \left(\nabla \cdot \tilde\A\big(y,U,\nabla R,\lambda,\beta\big)
  + \tilde\B\big(y,U,R,\nabla R,\lambda,\beta\big) \right)\, dy\, ds.
\end{align*}
The boundary integral can be handled in a similar fashion.  
Writing $\tilde \A = (\tilde \A_1,\tilde \A^\prime)$, the contribution due to
$\tilde \A_1$ can be rewritten as
\begin{align*}
  \frac 1\beta \int_0^x \!\!\! \int_{\base} \varphi_0 \partial_x
  \tilde\A_1\big(y,U, \nabla R,\lambda,\beta\big)\, dy\, ds 
  &=
  \frac 1\beta \int_{\base} \varphi_0 
  \tilde\A_1\big(y,U,\nabla R,\lambda,\beta\big) \, dy\bigg|^x_0.
\end{align*}
Stripping away the evaluations bars, we recognize this as having the form
\begin{align*}
  % \label{eqn:comp1}
  \frac{1}{\beta} \mathfrak{I} S_g(\mathfrak{D}(U,R); \lambda,\beta),
\end{align*}
where
\begin{align*}
  \mathfrak{D}(U,R) := (U, \nabla R),
  \qquad 
  \mathfrak{I} f:= \int_{\base} \varphi_0(y) f(\placeholder, y) \, dy,
\end{align*}
and $S_g$ is the superposition operator defined by \eqref{at superposition def} for the function
\[ g(x, y, u, r, \lambda, \beta)  := \tilde{\A}_1(y, u, r, \lambda, \beta).
\]
%The truncated version will then be
%\begin{align*}
%  A_3 f(A_2(U_1,U_2,R),\lambda,\beta,r)
%  &= \int_\base \varphi_0 
%  \tilde\A_1\big(y,\eta_r(U_1),\eta_r(U_2),\eta_r(R_x),\eta_r(R_y),\lambda,\beta\big)dy.
%\end{align*}
As $\mathfrak{D}$ only evaluates derivatives in the $R$ variables, it is easy to confirm that 
\[ \mathfrak{D} \textrm{ bounded and linear } \atXspace_\mu \longrightarrow  C_\mu^{1+\alpha}(\mathbb{R}; \mathbb{R}^2) \times C_\mu^{1+\alpha}(\overline{\Omega}; \mathbb{R}^2) =: \atYspace_\mu, \]
for any $\mu \geq 0$, with bounds uniform in $\mu$ on compact subsets of $[0,\overline{\mu})$.    Clearly, then, $\mathfrak{D}$ satisfies \eqref{at A2}.
The function $g$ is $C^{M+3}$ and flat as required by \eqref{at g assumptions} in light of the regularity assumed on the coefficients and the presence of the trivial solution family in \eqref{F trivial solution assumption}.  Finally, $\mathfrak{I}$ simply integrates in the transversal direction, and hence 
\[ \mathfrak{I} \textrm{ bounded and linear } \atYspace_\mu \to C_\mu^{1+\alpha}(\mathbb{R}) \textrm{ and }  \atYspace_\bdd \to C_\bdd^{1+\alpha}(\mathbb{R}),\]
with bounds uniform in $\mu$ on compact subsets of $(0,\overline{\mu})$.  In particular, the structural assumption \eqref{at A3} is satisfied.    

% Now let's count
%derivatives argument $(U_1,U_2,R)$ of $A_2$ lies in $C^{2+\alpha} \by C^{1+\alpha}
%\by R$ and so each component of $A_2(U_1,U_2,R)$ certainly lies in
%$C^{1+\alpha}$. This regularity is preserved by the application of the
%function $f$, and $A_3$ simply sends the resulting $C^{1+\alpha}$ function
%to a $C^{1+\alpha}$ function of a single variable. Clearly $A_2,A_3$ are
%also bounded mappings between the relevant weighted and unweighted
%spaces. In particular, $A_3$ is bounded with respect to the
%$*$-seminorms. It also seems immediate that $A_3|^x_0$ is bounded
%between weighted spaces and with respect to the $*$-seminorms. 

Now let us move on to the contribution of $\tilde\A^\prime$ to the $U_2$
equation:
\begin{align*}
  \frac 1\beta \int_0^x \!\!\! \int_{\base} \varphi_0 \nabla^\prime \cdot \tilde\A^\prime\big(y,U, \nabla R,\lambda,\beta\big)\,
  dy\, ds.
\end{align*}
Stripping off the $1/\beta$ and integrating by parts in $y$, we get
\begin{align*}
  - \int_0^x \!\!\! \int_{\base} \nabla^\prime \varphi_0 \cdot \tilde\A^\prime\big(y,U, \nabla R,\lambda,\beta\big)\,
  dy\, ds
  + \int_0^x \!\!\! \int_{\partial\base} \varphi_0 N^\prime \cdot \tilde\A^\prime\big(y,U, \nabla R,\lambda,\beta\big)\,
  dS(y) \, ds.
\end{align*}
These are analogous to the contribution of $\tilde\A_1$ considered
above, except that the operator $\mathfrak{I}$ is post-composed with $f \mapsto
\int_0^x f\,ds $. While this is \emph{not} a bounded map from $C^{1+\alpha}_\bdd(\mathbb{R}) \to
C^{1+\alpha}_\bdd(\mathbb{R})$, it \emph{is} bounded $C^{1+\alpha}_\bdd(\mathbb{R}) \to
\homholder^{1+\alpha}_\bdd(\mathbb{R})$, which is all that is required in \eqref{at A3}. The contribution from $\tilde\B$ is
treated in exactly the same manner; indeed, it is even simpler since
no integration by parts is needed.

\subsubsection*{Equation for \texorpdfstring{$R$}{R}.}
The work for the $R$ equation has mostly been done through the study
of the operator $G$. We know in particular that $G$, and hence the
composition $\mathfrak{I} := G Q_{\ge 1}$, are bounded $\Yspace_\mu \to P_{\geq 1} \Xspace_\mu$ for any $\mu \in [0, \overline{\mu})$ by Proposition~\ref{partial green lemma}. The
argument of $GQ_{\ge 1}$ is the interior and boundary components of
$\nl(U_1\varphi_0 + R,\lambda)$, each of which can be written as
\begin{align*}
   S_g(U,\partial_x U_2,R, \nabla R,D^2R; \lambda,\beta)
\end{align*}
for some suitably flat  $g$ that is independent of $x$. Thus for $\mathfrak{D}$ we can take
\begin{align*}
  \mathfrak{D}(U,R) :=  
  (U, \partial_x U_2,R,\nabla R,D^2R),
\end{align*}
which satisfies 
\[ \mathfrak{D} \textrm{ is bounded and linear } \atXspace_\mu \longrightarrow C_\mu^{\alpha}(\mathbb{R} ; \mathbb{R}^3 ) \times C_\mu^\alpha(\overline{\Omega}; \mathbb{R}^{1+n+n^2}) \]
for any $\mu \in [0,\overline{\mu})$ with bounds uniform on compact subsets of this interval.  This will certainly satisfy \eqref{at A2}, and we have 
\[ (U, R) \mapsto S_g (\mathfrak{D}(U,R); \lambda, \beta) \textrm{ bounded } \atXspace_\mu \to C_\mu^\alpha(\overline{\Omega}),\]
for all $\mu \in [0,\overline{\mu})$ and uniformly on compact subsets.  Applying Proposition~\ref{partial green lemma}, we conclude that $\mathfrak{I}$ will then satisfy \eqref{at A3}:
\[ \mathfrak{I} \textrm{ is bounded and linear } C_\mu^\alpha(\overline{\Omega}) \to C_\mu^{2+\alpha}(\overline{\Omega}) \textrm{ and } C_\bdd^\alpha(\overline{\Omega}) \to C_\bdd^{2+\alpha}(\overline{\Omega}),\]
for all $\mu \in (0, \overline{\mu})$.

\subsection{Truncation and fixed point mapping} \label{truncation section}
We have verified all of the hypotheses of Theorem~\ref{at fixed point theorem}. As it stands, however, this only tells us about solutions to a truncated version of \eqref{eqn:fixed} where the nonlinear terms have been precomposed with cutoff functions. Undoing the various changes of variable, this leads us to a cut-off version $\F^r$ (in the sense of \eqref{restriction composition}) of the nonlinear elliptic operator $\F$, where $r > 0$ measures the scale of the cut-off function.  An advantage of $\F^r$ is that it is defined as a mapping $\Xspace_\mu \by \R \to \Yspace_\mu$ between weighted spaces. If we increase the weight on the target space relative to the domain, then we can arrange for $\F^r$ to have any finite degree of smoothness: $\F^r \in C^{M+3}(\Xspace_\mu \by \R; \Yspace_{(M+6)\mu})$.  This rather technical fact is a consequence of \cite[Theorem 2.1]{amick1994center}.   A slight complication is that the operator $\F^r$ is no longer local, since it is defined with reference to the spectral splitting \eqref{spectral splitting equation}. But this splitting only occurs in the transverse variable $y$, and so $\F^r$ \emph{is} local in $x$. Moreover, $\F$ and $\F^r$ agree in a sufficiently small ball in $\Xspace_\bdd$.
%Ming: Check the weights for F^r.
\begin{lemma}\label{lem cutoff F}
  In the setting of Theorem~\ref{reduction theorem}, suppose that $\| u\|_{\Xspace} < r$. Then
  \begin{enumerate}[label=\rm(\roman*)]
  \item \label{F = Fr near 0} $\F^r(u,\lambda) = \F(u,\lambda)$.
  \item \label{Fu = Fru} $\F^r_u(u,\lambda) \big|_{\Xspace_\bdd} = \F_u(u,\lambda)$.
  \item \label{DlDuF = DlDuFr} $D^{\ell}_u D^k_\lambda \F^r(u,\lambda) \big|_{\Xspace^{\ell}_\bdd} = D^{\ell}_u D^k_\lambda \F(0,0)$.
  \end{enumerate}
\end{lemma}
\begin{proof}
Part~\ref{F = Fr near 0} is obvious by the definition of $\F^r$ in the sense of \eqref{restriction composition}. For~\ref{Fu = Fru}, we know that for any $v \in \Xspace_\bdd$ and for $\|u\|_{\Xspace} < r$ one can find $\varepsilon$ sufficiently small so that $\|u + \varepsilon v\|_{\Xspace} < r$. Thus
\[
\F^r_u(u,\lambda) v = \left.{d\over d\varepsilon}\right|_{\varepsilon = 0}\F^r(u+\varepsilon v, \lambda) = \left.{d\over d\varepsilon}\right|_{\varepsilon = 0}\F(u+\varepsilon v, \lambda) = \F_u(u,\lambda) v.
\]
  It is also easy to see that $D^k_\lambda \F^r(u,\lambda) = D^k_\lambda \F^r(u,\lambda)$ for $\|u\|_{\Xspace} < r$ and $k\le M$. Repeatedly differentiating with respect to $u$, \ref{DlDuF = DlDuFr} then follows by induction on $\ell$.
\end{proof}

In terms of $\F^r$, the result of applying Theorem~\ref{at fixed point theorem} is recorded in the following lemma. 
\begin{lemma}[Existence of a fixed point] \label{fixed point lemma} 
  For any integer $M \geq 2$, there exists $\mu \in (0,\overline{\mu})$, $r > 0$, $\beta \in (0,1]$, and a $C^{M+1}$ mapping 
\begin{align}
    \label{fixed point formula with xi}
    \FP \maps \R^2 \by \R  \longrightarrow \atXspace_\mu \qquad (\xi_1,\xi_2, \lambda) \longmapsto  (U_1,U_2,R)    \end{align}
so that, for all $(\xi_1, \xi_2, \lambda)$, the function $u \in \Xspace_\mu$ defined by 
  \begin{equation}
    \label{u from W}
    u(x,y) = U_1(x)\varphi_0(y) + R(x,y)
  \end{equation}
is the unique solution to the truncated problem $\F^r(u,\lambda)=0$ that satisfies the initial conditions 
  \begin{align*}
    % \label{something implicit here}
    \xi_1 = U_1(0) = \int_{\base} u(0,y) \varphi_0(y)\, dy,
    \qquad \xi_2 = U_2(0) = \frac 1\beta \int_{\base} u_x(0,y) \varphi_0(y)\, dy.
  \end{align*}
\end{lemma}

\subsection{Proof of main results} \label{proof of center manifold section}

We are now ready to prove the main results of this section.
\begin{proof}[Proof of Theorem~\ref{reduction theorem}]
  Our first step is to change variables from the initial data $\xi = (\xi_1,\xi_2)$ in Lemma~\ref{fixed point lemma} to 
  \begin{align*}
    a = (a_1,a_2) := \left(\frac{u(0,0)}{\varphi_0(0)}, \frac{u_x(0,0)}{\beta\varphi_0(0)} \right).
  \end{align*}
  Towards that end, fix $(\xi_1,\xi_2) \in \mathbb{R}$, and suppose that $u$ is given by \eqref{u from W} and \eqref{fixed point formula with xi} in Lemma~\ref{fixed point lemma}. Then we calculate
  \begin{align*}
    a_1 = \frac{u(0,0)}{\varphi_0(0)} u(0,0) = U_1(0) + \frac 1{\varphi_0(0)} R(0,0)
     = \xi_1 + \frac 1{\varphi_0(0)} \FP_3(\xi_1,\xi_2,\lambda)(0,0),
  \end{align*}
  and, similarly, 
  \begin{align*}
    a_2 = \frac{u(0,0)}{\beta \varphi_0(0)} = \frac 1\beta U_1(0) + \frac 1{\beta\varphi_0(0)} R_x(0,0)
     = \xi_2 + \frac 1{\beta\varphi_0(0)} \dell_x\FP_3(\xi_1,\xi_2,\lambda)(0,0).
  \end{align*}
  Thanks to the estimates in Lemma~\ref{lem Rstructure}, the mapping $\xi \mapsto a$ is a $C^{M+1}$ near-identity change of variables. In particular, it has a $C^{M+1}$ inverse $\xi = a + g(a, \lambda)$ for some function $g\in C^{M+1}$ which is flat in that $g(0,\lambda) = g_a(0,\lambda) = 0$. Introducing the scaled variables $A := \varphi_0(0) a_1$ and $B := \beta\varphi_0(0) a_2$, we further rewrite this as 
  \[
      \xi_1 = \frac A{\varphi_0(0)} + G_1(A,B,\lambda), \qquad
      \xi_2 = \frac B{\beta\varphi_0(0)} + G_2(A,B,\lambda),
  \]
  for some $G_1,G_2$ that are flat with respect to $(A,B)$.  

  The Faye--Scheel reduction function $\Psi$ can now be explicitly defined by
  \begin{align*}
    \Psi(A,B,\lambda)(x,y) &= G_1(A,B,\lambda)\varphi_0(y) - Bx\varphi_0(y) \\
    &\qquad + \FP\bigg(
    \frac A{\varphi_0(0)} + G_1(A,B,\lambda),
    \frac B{\beta\varphi_0(0)} + G_2(A,B,\lambda), \lambda\bigg)(x,y),
  \end{align*}
  All of its properties are straightforward to check, and we obtain the formula \eqref{FS expression} with $x_0 \ne 0$ by appealing to the translation invariance of the problem.

  It remains to derive the ODE \eqref{reduced ODE} for $v := u(\placeholder,0)$.  Differentiating \eqref{FS expression} twice with respect to $x$ we obtain
  \begin{equation*}
    u_{xx}(x + \tau, y) =  \dell_\tau^2 \Psi(v(x), v'(x), \lambda)(\tau,y).
  \end{equation*}
  Setting $y=0$ and $\tau=0$ this becomes
  \begin{equation*}
    v''(x) = u_{xx}(x, 0) =  \frac {d^2}{d\tau^2}\bigg|_{\tau=0} \Psi(v(x), v'(x), \lambda)(\tau,0)
    =: f(v(x),v'(x),\lambda),
  \end{equation*}
  as desired.
\end{proof}

\begin{proof}[Proof of Theorem~\ref{expansion theorem}]
  Our proof of Theorem~\ref{reduction theorem} ensures the existence of the Faye--Scheel reduction function $\Psi$ which is uniquely defined by $\FSproj\Psi = 0$ and 
  \begin{align}
    \label{eqn:useme tilde}
    \F^r\big(A\varphi_0 + Bx\varphi_0 + \Psi(A,B,\lambda),\lambda\big) = 0.
  \end{align}
  From the regularity and flatness properties \eqref{Psi flatness} of $\Psi$, we know that $\Psi$ admits an expansion of the form \eqref{Psi taylor expansion}, and \ref{expansion theorem projection} follows from applying $\FSproj$ to it. Next we differentiate \eqref{eqn:useme tilde} to obtain
  \begin{equation}\label{cutoff derivatives}
    \partial^i_A \partial^j_B \partial^k_\lambda \Big|_{(A,B,\lambda) = 0} \F^r \left(A\varphi_0 + Bx \varphi_0 + \Psi(A,B,\lambda) \right) = 0
  \end{equation}
  for $i + j + k \le M$. Since the implied partials of $\F^r$ are all being evaluated at $(u,\lambda) = (0,0)$, we can then use Lemma~\ref{lem cutoff F} to replace them with the desired partials of $\F$, proving \ref{expansion theorem Gateaux}.

  It remains to show that the $\Psi_{ijk}$ are uniquely determined by these properties. Plugging \eqref{Psi taylor expansion} into \eqref{Gateaux derivative} and recalling that $L = \F_u(0,0)$, we find that 
  \begin{equation}\label{solve Psi}
    L \Psi_{ijk} + \mathscr R_{ijk} = 0,
  \end{equation}
  where $\mathscr R_{ijk}$ depends on $\Psi_{i'j'k'}$ for $i' \le i, \ j'\le j,\ k'\le k$ and $i' + j' + k' \le i + j + k -1$. Lemma~\ref{lem bordering} then allows one to solve $\Psi_{ijk}$ uniquely from $\{ \Psi_{i'j'k'} \}$.
\end{proof}

\subsection{General strategy} \label{general strategy sec}

In the course of proving Theorem~\ref{expansion theorem}, we have shown that each term $\Psi_{ijk}$ can indeed be uniquely determined by iteratively solving a hierarchy of equations of the form \eqref{solve Psi}, where the terms of the right-hand side involve information about various Fr\'echet derivatives of $\F$ at $(0,0)$. In this section, we briefly illustrate how this process is carried out in practice, and also how the reduced equation \eqref{reduced ODE} can be rescaled to obtain homoclinic or heteroclinic solutions.

%  idea about how to actually calculate them.

\subsubsection{Iteration}\label{subsubsec iteration}   

The smoothness of $\F^r$ and Lemma~\ref{lem cutoff F} allow us to write
\begin{equation}\label{expansion F}
    \F^r(u,\lambda) = \sum_{1 \le \ell + k \le K} \lambda^k D^{\ell}_u D^k_\lambda \F^r(0,0)[u,u,\ldots,u] + O\bigg(\sum_{\ell+k=K+1} \n u_{\Xspace_\mu} ^{\ell}\abs \lambda^k\bigg) 
\end{equation}
in $\Yspace_{(K+4)\mu}$ for any integer $K \le M$, where the $D^{\ell}_u D^k_\lambda \F^r(0,0)$ are symmetric bounded $\ell$-linear mappings $\Xspace_\mu^\ell \rightarrow \Yspace_{(K+4)\mu}$. 
%Ming: Check the weights!
For $i+j+k \le K$, the remainder terms in \eqref{expansion F} do not contribute to $\mathscr R_{ijk}$ in \eqref{solve Psi}. Therefore, when solving \eqref{solve Psi}, it is sufficient to work with the truncated version of \eqref{expansion F} that results from simply setting these remainder terms to zero.

For an integer $K \ge 1$ and a smooth function $g = g(A, B, \lambda)$, we define $\mathcal T_K g$ to be the $K$-th order Taylor expansion of $g$ at 0, that is,
\begin{equation*}
  \mathcal T_K g(A, B, \lambda) := \sum_{i + j + k \le K} \partial^i_A \partial^j_B \partial^k_\lambda g(0,0,0) A^i B^j \lambda^k.
\end{equation*} 
Plugging \eqref{expansion F} and \eqref{Psi taylor expansion} into \eqref{eqn:useme tilde} we see that, for $1\le K \le M$,
\begin{equation}\label{truncated Psi eqn}
\mathcal T_K  \sum_{1 \le \ell + k \le K} \lambda^k D^{\ell}_u D^k_\lambda \F(0,0)[u^{(K)},\ldots,u^{(K)}]  = 0,  %\quad \text{at} \quad O\bigg( \sum_{i+j+k=K+1} \abs A^i\abs B^j \abs \lambda^k\bigg),
\end{equation}
where 
  \begin{equation*}
  \begin{split}
    %\label{eqn:uN}
    u^{(K)}(x,y;A,B,\lambda) & := \mathcal T_K\left[ A\varphi_0(y) + Bx\varphi_0(y) + \Psi(A, B, \lambda)(x,y) \right] \\
    & = A\varphi_0(y) + Bx\varphi_0(y) + \sum_{\substack{2 \le i+j+k \le K \\ i+j\ge1}} \Psi_{ijk}(x,y) A^i B^j \lambda^k. 
    \end{split}
  \end{equation*}
%is a truncated approximation for $\Psi$ with polynomial dependence on $A, B, \lambda$.

More explicitly, at $K = 1$, the definition of $u^{(K)}$ reads simply $u^{(1)}(x,y) = A\varphi_0(y) + Bx\varphi_0(y)$.  For $K \ge 2$, we may use the facts that 
\begin{equation}\label{sim truncation}
D^k_\lambda \F(0,0) = 0, \quad \text{and} \quad u^{(K)} = u^{(K-1)} +  \sum_{\substack{i+j+k = K \\ i+j\ge1}} \Psi_{ijk}(x,y) A^i B^j \lambda^k
\end{equation}
to derive the equations for $\Psi_{ijk}$ when $i+j+k = K$. Below we give two example calculations for $K = 2,3$. As all of the derivatives of $\F$ are evaluated at $(0,0)$, the base point will be suppressed for notational convenience.  

When $K = 2$, \eqref{truncated Psi eqn} and \eqref{sim truncation} imply that 
\[
L \bigg(\sum_{\substack{i+j+k = 2 \\ i+j\ge1}} \Psi_{ijk} A^i B^j \lambda^k \bigg) = - \lambda \F_{u\lambda} u^{(1)} - \F_{uu}[u^{(1)}, u^{(1)}],
\]
from which $\left\{ \Psi_{ijk}:\ i+j+k=2 \right\}$, and hence $u^{(2)}$, can be uniquely solved by applying Lemma~\ref{lem bordering}.  At $K=3$, a similar calculation gives
\begin{align*}
  L \bigg(\smash{\sum_{\substack{i+j+k = 3 \\ i+j\ge1}}} \Psi_{ijk} A^i B^j \lambda^k \bigg) & = -L u^{(2)} - \lambda \F_{u\lambda} u^{(2)} - \F_{uu}[u^{(2)}, u^{(2)}] \\
& \qquad - \lambda^2 \F_{u\lambda\lambda} u^{(1)} - \lambda \F_{uu\lambda}[u^{(1)}, u^{(1)}] - \F_{uuu}[u^{(1)},u^{(1)},u^{(1)}].
\end{align*}
The right-hand side is explicit. Grouping like terms and applying Lemma~\ref{lem bordering} we may determine $u^{(3)}$.

  This process repeats at each stage: we have to iteratively solve linear equations
  of the form
  \begin{equation}\label{eqn for Psi}
    L\Psi_{ijk} = F_{ijk},
    \qquad \FSproj\Psi_{ijk} = 0,
  \end{equation}
  where $i+j+k=K$ and $F_{ijk}$ depends only on $u^{(K-1)}$.

In summary the calculation for $\Psi_{ijk}$ can be explained in the following
  way:
  \begin{enumerate}[label=\rm Step \arabic*.]
  \item\label{step1} Taylor expand the terms in $\F$ to order $M$ to obtain
    a Taylor-truncated operator which is naturally defined on
    weighted spaces.
  \item\label{step2} The composition of the $K$-th order Taylor-truncated operator with $u^{(K)}$ is a polynomial in
    $A,B,\lambda$ whose $\Yspace_{K\mu}$ coefficients depend on the $\Psi_{ijk}$.
  \item\label{step3} Setting the coefficients of $A^i B^j \lambda^k$ for $i+j+k \le K$ equal
    to zero, we obtain a series of equations for the $\Psi_{ijk}$.
  \item\label{step4} Working in order of increasing $i+j+k$, this becomes a sequence
    of linear problems \eqref{eqn for Psi} where $F_{ijk}$ is known. Lemma~\ref{lem bordering} ensures that these equations can be solved uniquely.
  \end{enumerate}

\subsubsection{Anticipated scaling}\label{subsubsec scaling} 
The reduced ODE \eqref{reduced ODE} always admits two degrees of freedom:  we may select a length scale for the $x$-variable as well as an amplitude scale for the unknown.   
%in terms of the parameter-dependent scaling: the length scale and the amplitude scale. 
Making intelligent choices can vastly simplify the expansion procedure.  For example, if we hope to find a heteroclinic solution, the reduced ODE must have a certain structure, and this leads to an anticipated scaling.  

%If we have in mind the desire to obtain heteroclinic or homoclinic solutions, we can make parameter-dependent choices for these scales that 
%n practice, 
%
%On the other hand, the iteration described above assumes no information about the scaling. In practice, one may expect to simplify the procedure by implementing certain anticipated scales. %(see Section \ref{sec ap} -- Section \ref{sec ww}).

By design, \eqref{reduced ODE} always has an equilibrium at the origin. In applications we are interested in cases  where the linearized problem there is nondegenerate in that $f_{(A,B)}(0,0,\lambda)$ has no zero eigenvalue for $0 < |\lambda| \ll 1$. Treating $\lambda$ as fixed and performing a double expansion in $(A,B)$ we have
\begin{equation}\label{double expansion f}
\begin{split}
f(A,B,\lambda) & = f_A(0,0,\lambda) A + f_B(0,0,\lambda)B + {1\over2} f_{AA}(0,0,\lambda) A^2 + f_{AB}(0,0,\lambda) AB \\
& \quad + {1\over2}f_{BB}(0,0,\lambda) B^2 + {1\over 6} f_{AAA}(0,0,\lambda) A^3 + O\left( (|A| + |B|)^2 |B| \right).
\end{split}
\end{equation}
Note that the nondegeneracy condition forces $f_A(0,0,\lambda) \neq 0$.  For nontrivial heteroclincic or homoclinic solutions, we need a second rest point, which in terms of the above expansion translates to the right-hand side of \eqref{double expansion f} being nonlinear in $A$. Therefore, let us assume that there is a least integer $m\ge 2$  such that $\partial^m_A f(0,0,\lambda) \ne 0$.

Now, we introduce a rescaling of the axial variable $X = \kappa x$ and amplitude $v = aV$.  Given the above discussion, we want $v''$, $v$, and $v^m$ to appear as $O(1)$ terms in the corresponding rescaled version of the reduced ODE \eqref{reduced ODE}.  This balancing forces the inverse length scale $\kappa$ and the amplitude scale $a$ to satisfy 
\begin{equation}\label{scale k a}
a \kappa^2 \sim a f_A(0, 0, \lambda) \sim a^m \partial^m_A f(0,0,\lambda) \quad \text{as} \quad \lambda \to 0.
\end{equation}
Clearly, then,  $\kappa$ and $a$ involve roots of $f_A$ and $\partial^m_A f$. To avoid this inconvenience, we may reparameterize $\lambda = \lambda(\varepsilon)$, and consider
\begin{equation}\label{reparametrization equation}
  \lambda = \lambda_p \varepsilon^p, \qquad \kappa = \kappa_n \varepsilon^n, \qquad a = a_q \varepsilon^q
\end{equation}
for some $p, n, q \in \mathbb N$.  It then follows from \eqref{scale k a} that
\begin{equation}\label{rescaling relation}
\varepsilon^{2n} \sim f_A(0,0,\varepsilon^p) \sim \partial^m_A f(0,0,\varepsilon^p) \varepsilon^{(m-1)q} \qquad \text{as} \quad \varepsilon \to 0.
\end{equation}
In particular, this implies that when we carry out the iteration procedure of Section \ref{subsubsec iteration}, $A$, $B$, and $\lambda$ have differing orders of magnitude.  It therefore suffices to compute the $\Psi_{ijk}$ for $i,j,k$ in the index set 
\begin{equation*}%\label{lattice}
  \mathcal J = \left\{ (i,j,k) \in \mathbb N^3: \ qi + (q+n) j + k \le q+2n,\ i + j + k \ge 2,\ i + j \ge 1  \right\}.
\end{equation*}

Notice that we have not taken into account the contribution of $f_B(0,0,\lambda)B$ in the expansion \eqref{double expansion f}. This can be justified, for instance, when the system has the reversal symmetry $(v(x),v^\prime(x)) \mapsto (v(-x), -v^\prime(-x))$.  However, if $f_B(0,0,\lambda) \ne 0$, the length scale will be over-determined since there is a linear term in $B$  in \eqref{double expansion f} which also suggests a choice of $\kappa$.   For this to be compatible with \eqref{scale k a}, we must therefore have
% of the length $\l \sim {|f_B(0,0,\lambda)|}$ so that in the rescaled \eqref{reduced ODE} the terms $v''$ and $v'$ all appear at $O(1)$.
%Comparing \eqref{scale k a} and this new scaling yields a compatibility condition 
\begin{equation}
  \label{compatibility}
  |f_A(0,0,\lambda)| \sim {|f_B(0,0,\lambda)|^2} \ \text{ as }\  \lambda \to 0.
\end{equation}
With enough parameters in the problem, one can always arrange for this to hold; see, for example,  Section \ref{sec fkpp}.

%To summarize, by taking advantage of the scaling properties of the reduced ODE, one can substantially reduce the number of iterations from the most na\"ive way in \ref{step4} down to a subset of the inded points in $\mathcal J$ given in \eqref{lattice}. 

\subsection{Extensions} \label{extensions section}

\subsubsection*{Other boundary conditions}

In formulating Theorems~\ref{reduction theorem} and \ref{expansion theorem}, we chose to focus on problems that linearize to co-normal boundary conditions.  This is not essential: looking at the proof, it is clear that one can just as easily impose nonlinear Dirichlet conditions of the form 
\[ \G(y,u,\lambda) = 0 \qquad \textrm{on } \partial \Omega,\]
for $\G $ that is $C^{M+4}$ in its arguments. Naturally, for this case the codomain of $\F$ should be redefined to be
\[ \Yspace := C^{0+\alpha}(\overline{\Omega}) \times C^{2+\alpha}(\partial \Omega),\]
and likewise for $\Yspace_\mu$ and $\Yspace_\bdd$.  In place of the obliqueness assumption \eqref{nonlinear obliqueness}, we now require that
\begin{equation*}
  \G_z(y, 0,0) \neq 0 \qquad \textrm{for all } y \in \base. %\label{Dirichlet assumption} 
\end{equation*}
The proof of Proposition~\ref{partial green lemma} then proceeds as before, only using a priori estimates for linear elliptic PDE with homogeneous Dirichlet conditions.   The fixed point argument is essentially unchanged.

Likewise, if $\partial \Omega$ has two or more connected components, one can freely impose either Dirichlet or co-normal conditions on each, adjusting the regularity of $\Yspace$ accordingly.  

\subsubsection*{Internal interfaces and free boundaries} 

We can also expand the scope of the reduction theorem to treat nonlinear transmission problems.  Suppose that $n = 2$ and the base $\base$ is the union of two open intervals:
\[ \base = \base_1 \cup \base_2, \qquad \base_1 := (-1,0), \, \base_2 := (0,1).\]
Let $\Omega := \mathbb{R} \times \base$ be the (slitted) cylinder, and say $\Omega_i = \mathbb{R} \times \base_i$.  

Physically, one can for instance imagine this as representing two immiscible fluids confined to a channel with rigid walls; the interface between them is the line $\Gamma := \mathbb{R} \times \{ 0 \}$. Of course this interface is only flat once we have performed a change of variables, and this may introduce terms in the interior equation relating to traces (or derivatives of traces) of quantities on the boundary.   With that in mind, we consider the following quite general quasilinear elliptic problem: 
\begin{equation}
  \left\{ \begin{aligned} 
    \nabla \cdot \A(y, u, \nabla u, u|_\Gamma,  u_x |_\Gamma, \lambda)  & = 0 \qquad \textrm{in } \Omega \\
    \G(u_1, u_2, \nabla u_1, \nabla u_2, \lambda)& = 0 \qquad \textrm{on } \Gamma \\
    \K(u_1, u_2, \lambda) & = 0 \qquad \textrm{on } \Gamma \\
    u & = 0 \qquad \textrm{on } \{ y = \pm 1\},
  \end{aligned} \right. \label{transmission elliptic PDE} 
\end{equation}
where $\A_1$, $\A_2$, $\G$, $\K$ are $C^{M+4}$ in their arguments.  Here, we are breaking convention slightly by writing $u_i := u|_{\Omega_i}$ and likewise for $\A$.  As before, assume that $\A_i$ is uniformly elliptic \eqref{nonlinear ellipticity}.  In place of obliqueness \eqref{nonlinear obliqueness}, we now ask that 
\[
N(y) \cdot \left( \G_{p_1}(z_1, z_2, p_1, p_2) - \G_{p_2}(z_1, z_2, p_1, p_2) \right) > \chi \quad \textrm{for all } y \in \Omega^\prime,~ p_1, p_2 \in \mathbb{R}^n,~ z_1, z_2, \lambda \in \mathbb{R}.
\]

The elliptic problem \eqref{transmission elliptic PDE} can be rewritten as an operator equation $\F(u,\lambda) = 0$, with $\F = (\F_1, \F_2, \F_3)$ corresponding to the first three equations and the homogeneous Dirichlet condition incorporated into the definition of the space.  The main restriction is that the \emph{linearized} problem is of transmission type, that is, 
\begin{equation}
  \begin{aligned}
    \F_{1u}(0,0) v & := \nabla \cdot \left( a(y) \nabla v \right),   \\
    \F_{2u}(0,0) v & := -\jump{N(y) \cdot a(y) \nabla v} + g_1(y) v_1 + g_2(y) v_2, \\
    \F_{3u}(0,0) v & := \jump{v}. 
  \end{aligned}
  \label{linear tranmission F} 
\end{equation}
Here $\jump{\placeholder} := (\placeholder)_2 - (\placeholder)_1$ denotes the jump of a quantity over $\Gamma$, and the coefficients $a$, $g_i$ are obtained from $\A_i$ and $\G$ in the obvious way.    

In typical applications, one asks for solutions whose restriction to $\Omega_i$ is smooth up to the boundary.  We therefore set 
\begin{equation}
  \Xspace := \left\{ u  : u|_{\Omega_i} \in C^{2+\alpha}(\overline{\Omega_i}), ~ u|_{y=\pm1} = 0 \right\},
  \label{definition X transmission}
\end{equation}
and take as the codomain for the corresponding nonlinear mapping $\mathscr{F}$ the space
\begin{equation}
  \Yspace := \left\{ w : w|_{\Omega_i} \in C^{0+\alpha}(\overline{\Omega_i})  \right\} \times C^{1+\alpha}(\Gamma) \times C^{2+\alpha}(\Gamma).
  \label{definition Y transmission}
\end{equation}
%Note that by standard elliptic regularity theory, any solution $u \in \Xspace$ to \eqref{transmission elliptic PDE}   enjoys the improved regularity $u \in C^{0+\alpha}(\overline{\Omega})$.

While the jump conditions on $\Gamma$ in \eqref{linear tranmission F} are somewhat exotic, there is a well-established literature regarding them, including the Schauder estimates \cite{ladyzhenskaya1968linear} that we require.  
%  In fact, \eqref{transmission elliptic PDE} admits an elegant weak formulation:  
%\begin{equation}
%  \int_\Omega \left( \A(y, \nabla u, \lambda) \cdot \nabla \varphi - \B(y, u, \nabla u) \varphi \right) \, dx \, dy  = \int_{\Gamma} \G(u, \lambda) \varphi \, dx,
%  \label{weak transmission formulation}
%\end{equation}
%for all test function $\varphi \in C_c^\infty(\overline{\Omega})$.    This exploits the fact that the jump of the co-normal derivative naturally arises from the distributional divergence of $\A(y, \nabla u, \lambda)$ over the surface of discontinuity.   
It is then quite straightforward to generalize Theorem~\ref{reduction theorem} and Theorem~\ref{expansion theorem} to the setting of \eqref{transmission elliptic PDE}.  Indeed, Amick and Turner explicitly mention how their fixed point theory accommodates spaces similar to \eqref{definition X transmission} and \eqref{definition Y transmission} (see, \cite[Remark 2.2, Remark 3.2]{amick1994center}), and in \cite{amick1989small} they apply it to a transmission problem in hydrodynamics that is a special case of what we consider in Section~\ref{sec ww}.

\subsubsection*{Diagonal elliptic systems}

Another possibility is to study systems of quasilinear elliptic PDE.  To do this in complete generality is beyond the scope of this paper, but, with just a minor modification, the above argument can treat a special class of systems that are important for the proof of Theorem~\ref{nondegeneracy theorem}.  

Letting $\Omega$ again be any connected cylinder as in Section~\ref{introduction section}, we consider solutions $u = (u^1, u^2)$ to   
\begin{equation}
  \left\{ \begin{aligned} 
    \nabla \cdot \A^i(y, \nabla u^i, \lambda)  + \B^i(y, u, \nabla u, \lambda) & = 0 \qquad \textrm{in } \Omega \\
    -N(y) \cdot  \A^i(y, \nabla u^i, \lambda) + \G^i(y, u, \lambda) & = 0 \qquad \textrm{on } \partial\Omega,   \end{aligned} \right. \label{diagonal elliptic PDE} 
\end{equation}
for $i = 1, 2$.  We assume that the coefficients $\A^i$, $\B^i$, $\G^i$ are $C^{M+4}$ in their arguments, and also that uniform ellipticity \eqref{nonlinear ellipticity} and obliqueness \eqref{nonlinear obliqueness} hold.  Suppose further that the linearized problem at $(u, \lambda) = (0,0)$ is \emph{diagonal} in the sense that 
%\begin{align*}
%\nabla_p\A_{1}(y, 0,0) & = \nabla_p \A_{2}(y, 0,0), \\
% \nabla_{(z,p,q)} \B_{1}(y,0,0,0) &= \nabla_{(z,p,q)}\B_{2}(y,0,0,0), \\
% \nabla_z \G_{1}(y,0,0) &= \nabla_z \G_{2}(y,0,0) \end{align*}
% for all $y \in \base$.  
%% If we then define 
%% \[ L : u_i \in \mathcal{X}_\mu \to \left( \, \right) \in \Yspace_\mu,\]
% The
 \eqref{diagonal elliptic PDE} can be rewritten as 
\begin{equation}
  \left\{
    \begin{aligned}
      L u^1 & = \mathcal{N}^1(u^1, u^2, \lambda) \\
      L u^2 & = \mathcal{N}^2(u^1, u^2, \lambda), 
    \end{aligned} \right.
  \label{definition diagonal system}
\end{equation}
where $L : \Xspace_\bdd \to \Yspace_\bdd$ is a bounded linear operator, and each $\mathcal N^i$ is a divergence form nonlinear operator that is (i) $C^{M+3}$ as a mapping $\Xspace_\bdd \times \Xspace_\bdd \times \mathbb{R} \to  \Yspace_\bdd$, and (ii) satisfies $\mathcal N^i(0,\lambda)=0$ and $\mathcal{N}^i_u(0,0) = 0$.  Arguing exactly as in Section~\ref{reformulation fixed point section}, this problem can be reformulated as a fixed-point equation of the form \eqref{eqn:fixed} but with six components --- three each for $u^1$ and $u^2$.  Amick and Turner's theory also applies in this setting, and so we recover Theorem~\ref{reduction theorem} and Theorem~\ref{expansion theorem}, with the reduction function now taking values in the product space $\Xspace_\mu \times \Xspace_\mu$. 

%{Linearization of the reduced equation}

\subsubsection*{Commuting linearization and reduction} 
With the above center manifold theory for diagonal systems at our disposal, we are now prepared to prove Theorem~\ref{nondegeneracy theorem}.

\begin{proof}[Proof of Theorem~\ref{nondegeneracy theorem}]
Suppose that we are in the setting of Theorem~\ref{reduction theorem}.  Throughout the argument, we will work with a fixed value of $\lambda$ that is taken sufficiently small.  For convenience, it will therefore be suppressed.

Consider the following (truncated) \emph{augmented problem}
\begin{equation}
  \label{nondegen augmented system} 
  \mathscr{G}^r(u, \dot u) := \Big( \F^r(u),\ \F_u^r(u) \dot u \Big) = 0.
\end{equation}
Recall that $\F^r$ denotes the truncated nonlinear mapping in the sense of Lemma~\ref{lem cutoff F}.   Naturally, \eqref{nondegen augmented system} is a (truncated) diagonal system satisfying \eqref{diagonal elliptic PDE} and \eqref{definition diagonal system}, and so we may apply the modified version of Theorem~\ref{reduction theorem} to classify its small bounded solutions.  In particular, there exists neighborhoods $U \subset \Xspace_\bdd \times \Xspace_\bdd$ and $V \subset \mathbb{R}^4$ of the origin, and a reduction function $(\Phi,  \Upsilon) \in C^{M+1}(\mathbb{R}^4, \Xspace_\mu  \times \Xspace_\mu)$ so that $(w,\dot w) \in  U$ solves \eqref{nondegen augmented system} if and only if
\begin{equation}
  \left\{
    \begin{aligned}
      w &= ( A + B x ) \varphi_0(y) + \Phi( A, B, \dot A, \dot B) \\
      \dot w &= (  \dot A +  \dot B x ) \varphi_0(y)  + \Upsilon(A, B, \dot A, \dot B),
    \end{aligned}
  \right.
  \label{nondegen uw reduction}
\end{equation}
for some $(A,B,\dot A,\dot B) \in V$.   We are recycling notation here somewhat, as $\Phi$ above is not the same as the one occurring in Remark~\ref{Phi remark}.    Let us now \emph{define} 
\[ \Psi(A,B) := \Phi(A,B, 0, 0).\]
It is easy to check that this function has all the properties of the reduction function furnished by  Theorem~\ref{reduction theorem} to the original (truncated) problem.  In particular, this means that any sufficiently small  $w \in \Xspace_\bdd$ satisfying $\F^r(w) = 0$ can be written 
\begin{equation}
  w(x,y) = \left(A + Bx\right) \varphi_0(y) + \Psi(A,B)(x,y) \qquad \textrm{in } \Omega,  \label{nondegen u reduction} 
\end{equation} 
for some $A$, $B \in \mathbb{R}$.   Moreover, $v := w(\placeholder,0)$ solves the reduced ODE \eqref{reduced ODE}, with $f$ defined by \eqref{reduced f definition}.  For simplicity, let us also normalize $\varphi_0(0) = 1$, which implies $A = v(0)$ and $B = v^\prime(0)$.  Note also that, by construction, the range of $\Phi$ and $\Upsilon$ lie in the kernel of the projection $\FSproj$ onto $\kernel{L}$.  Consequently, the coefficients $A$ and $B$ in \eqref{nondegen u reduction} and \eqref{nondegen uw reduction} must indeed coincide. 
%\[ \Phi(0,0, 0, 0) = 0,~ \Phi_{(A,B)}(0,0,0,0) = 0, ~\Upsilon(0,0,0,0) = 0, ~  \Upsilon_{(\dot A,\dot B)}(0,0,0,0) = 0.\] 
%This ensures that the coefficients $A$ and $B$ in \eqref{nondegen u reduction} and \eqref{nondegen uw reduction} must indeed coincide.
% Miles: Is it really the flatness? Can't we just apply \FSproj?
%  Moreover, uniqueness of the reduction gives 
%\[ \Psi(A,B) = \tilde \Psi(A,B, 0,0) \qquad \textrm{for all } (A,B) \in V.\]
%
%
%Recall that $\mathcal{F}^r$ denotes the truncated version of the operator $\mathcal{F}$.

%Let $(u,\lambda) \in \mathcal{X}_\bdd \times \mathbb{R}$ be a solution of the PDE \eqref{main elliptic PDE} lying in the neighborhood $U$ of $(0,0)$.    We then have that 
%
%
%Now, 
Fix a small solution $u \in \Xspace_\bdd$ to $\F(u) = \F^r(u) = 0$, and let $\dot u \in \Xspace_\bdd$ be a solution of the linearized problem as in the statement of Theorem~\ref{nondegeneracy theorem}.   Even though we do not assume that $\|\dot u\|_{\Xspace}$ is small, the structure of the augmented problem problem ensures that, for all $\delta > 0$ sufficiently small, $(u, \delta \dot u)$ lies in $U$ and \[ \mathscr{G}(u,\delta \dot u) = \mathscr{G}^r(u,\delta \dot u) =  0.\]
In that case, we can use \eqref{nondegen uw reduction} with $(u,\delta \dot u)$ in place of $(w,\dot w)$ to see that
\begin{align*}
\delta \dot u(x,y) & = \delta (\dot A + \dot B x ) \varphi_0(y) + \Upsilon(A, B, \delta \dot A, \delta \dot B).
\end{align*}
As $\Upsilon$ is $C^{M+1}$, an expansion of the right-hand side above in $\delta$ yields
\begin{equation}
  \dot u = ( \dot A + \dot B x ) \varphi_0(y) + \Upsilon_{(\dot A,\dot B)}(A, B, 0, 0) \cdot (\dot A, \dot B).
  \label{nondegen second w reduction}
\end{equation}

We claim further that 
\begin{equation}
  \Psi_{(A,B)}(A,B) = \Upsilon_{(\dot A,\dot B)}(A,B, 0,0).
  \label{nondegen critical identity} 
\end{equation}
To see this, first observe that the construction of the reduction functions $\Psi$, and $(\Phi, \Upsilon)$ ensure that
\begin{equation*}
  \begin{aligned}
    \F^r \left( (A+Bx) \varphi_0 + \Phi(A,B,\dot A,\dot B) \right) & = 0, \\ 
    \mathscr{G}^r\big( (A+Bx) \varphi_0 + \Phi(A,B,\dot A, \dot B ), (\dot A + \dot B x) \varphi_0 + \Upsilon(A,B,\dot A,\dot B)\big) & = 0,
  \end{aligned}
\end{equation*}
for all $(A,B,\dot A,\dot B) \in V$. Note that $\F^r$ is $C^2$ as a mapping $\Xspace_\mu \to \Yspace_{4\mu}$.  This permits us to differentiate the first equation with respect to $(A,B)$, and upon evaluating at $(A,B,0,0)$ we find that 
\[
\F_u^r(u) \left[ \varphi_0 + \Psi_A(A,B)  \right] = 0, \qquad
\F_u^r(u) \left[ x \varphi_0 + \Psi_B(A,B) \right] = 0.
\]
Likewise, the second component of $\mathscr{G}^r$ is  $(u,\dot u) \mapsto \F_u^r(u) \dot u$, which is $C^1$ as a mapping $\Xspace_\mu  \times \Xspace_\mu \to \Yspace_{4\mu}$.  Taking its derivative with respect to $(\dot A,\dot B)$ and evaluating at $(A,B,0,0)$ leads to the identities
\[ \F_{u}^r(u)\left[ \varphi_0 + \Upsilon_{\dot A}(A,B,0,0) \right]  = 0, \qquad \F_u^r(u) \left[ x \varphi_0 + \Upsilon_{\dot B}(A,B,0,0) \right] = 0. \]
% Assuming that $u$ in \eqref{nondegen u reduction} is sufficiently small in $\Xspace_\bdd$, Lemma~\ref{lem cutoff F} tells us that $\F^r(u) = \F(u)$ and $\F_u^r(u) = \F_u(u)$.  
Combining the two identities above we conclude
\[ \F_u^r(u) \left[ \Psi_A(A,B) - \Upsilon_{\dot A}(A,B,0,0) \right] = 0, \qquad \F_u^r(u) \left[ \Psi_B(A,B) - \Upsilon_{\dot B}(A,B,0,0) \right] = 0.\]
On the other hand, by construction 
\[ \FSproj \left[ \Psi_A(A,B) - \Upsilon_{\dot A}(A,B,0,0) \right] = 0, \qquad \FSproj \left[ \Psi_B(A,B) - \Upsilon_{\dot B}(A,B,0,0) \right] = 0.\]
We know from Lemma~\ref{lem bordering} that the bordered operator $w \mapsto (\F_u^r(0) w, \FSproj w)$ is invertible $\Xspace_\mu \to \Yspace_\mu \times \kernel{L}$. Moreover, if $u \in \Xspace_\bdd$ has $\|u\|_{\Xspace}$ sufficiently small, 
the same is true for $w \mapsto (\F_u^r(u)w, \FSproj w)$ by a perturbation argument. Hence we have proved the key identity \eqref{nondegen critical identity}, at least when $A=v(0)$ and $B=v'(0)$ correspond to a sufficiently small solution $u \in \Xspace_\bdd$. The uniform smallness of $u$ in particular means that, say by Lemma~\ref{lem cutoff F}, we do not have to worry about the cut-off functions when performing this perturbative argument.

 Theorem~\ref{nondegeneracy theorem} follows almost immediately.  From \eqref{reduced ODE},  \eqref{nondegen second w reduction}, and \eqref{nondegen critical identity} we see that $\dot v := \dot u(\placeholder,0)$ solves the reduced equation
\[ \dot v^{\prime\prime} = g(v, v^\prime, \dot v, \dot v^{\prime}),\]
for
\[ g(A,B,\dot A,\dot B) := \frac{d^2}{dx^2}\Big|_{x=0}\left[ (\dot A,\dot B) \cdot \Psi_{(A,B)}(A,B)(x,0) \right]. \]
But, recalling the definition of $f$ \eqref{reduced f definition}, this becomes exactly the claimed ODE \eqref{linearized reduced ODE}. \end{proof}

\section{Anti-plane shear}\label{sec ap} 

Consider a homogeneous, incompressible, isotropic elastic cylinder $\mathcal D = \Omega \times \R$ with generators parallel to $z$-axis and cross section $\Omega \subset \R^2$ in $(x,y)$-plane. \emph{Anti-plane shear} describes the situation where the deformation takes the form 
\begin{equation}\label{deformation}
  {\id} + u(x, y) e_3,
\end{equation}
where $e_3$ is the standard basis vector $(0,0,1)^T$. That is, the displacement of each particle is parallel to the generators of the cylinder and independent of its axial position.

For an isotropic elastic solid, the strain energy density $\strainW$ is a function of the three principal invariants $I_1, I_2, I_3$ of the Cauchy--Green tensor. In this section, we will consider a polynomial rubber elastic model, which corresponds to the case where $\strainW$ is a polynomial in $I_1$ and $I_2$ \cite{Rivlin1951}. Thus we can write
\begin{equation}
  \label{ap W}
  \strainW(I_1, I_2) := \sum_{i+j=1}^N C_{ij} (I_1 - 3)^i (I_2 - 3)^j. %= \sum_{i=1}^N \tilde C_{i} |\nabla u|^{2i}, \quad \tilde C_1  = W'(3) > 0.
\end{equation}
 Note that when $N=1$, $C_{01} = 0$, this reduces to the standard neo-Hookean solid model \cite{Treloar1948}. Values of $N > 2$ are rarely used in practice because it is difficult to fit such a large number of material properties to experimental data. Therefore, we restrict our attention to the quadratic case $N = 2$; this will result in a quasilinear PDE with a 4-Laplacian term, cf.~\eqref{ap simpler}.

Imposing the anti-plane shear ansatz \eqref{deformation} and assuming incompressibility, we know that the principal invariants satisfy $I_1 = I_2 = 3 + |\nabla u|^2 =: I$, and $I_3 = 1$.  Hence, we may identify $\strainW$ with the function $\strainW(I) = \strainW(3 + |\nabla u|^2)$; see, for example, \cite{knowles1976finite,horgan1983finite}. At infinitesimal deformations, the shear modulus is given by $2\strainW'(3)$ which is supposed to be positive. For simplicity, we normalize $\strainW^\prime(3) = 1$.  Then the quadratic rubber model \eqref{ap W} becomes
\begin{equation}
  \label{ap W simple}
  \strainW(I) = (I-3) + w_1 (I -3)^2, \qquad \strainW'(3+|\nabla u|^2) = 1 + 2w_1 |\nabla u|^2,
\end{equation}
where $w_1 := \strainW^{\prime\prime}(3)/2$ is a material constant.

Following Healey and Simpson \cite{healey1998global}, we suppose that the body is subjected to a parameter-dependent ``live" body force $b=b(\lambda,  u)$. As in, e.g., \cite{HK81,horgan1983finite}, we consider the geometrical setting  where $\Omega = \R \times (-\pi/2,\pi/2)$ is an infinite strip and homogeneous Dirichlet boundary conditions are imposed on $\{y = \pm\pi/2\}$. 

A static equilibrium then satisfies 
\begin{equation}\label{ap simple}
    \left\{ 
    \begin{aligned} %\label{momentum}
      \nabla \cdot \left( \strainW'(3+ |\nabla u|^2) \nabla u \right) - b( u,\lambda)  & =  0 & &  \text{in } \Omega \\
      u & = 0 \quad & & \text{on } \partial\Omega.
    \end{aligned} 
    \right. 
\end{equation}
The system \eqref{ap simple} carries a variational structure with the energy 
\begin{equation*}%\label{ap energy}
E(u) := \int_\Omega \left[ \strainW(3+|\nabla u|^2) + B( u,\lambda) \right] \,dx \, dy,
\end{equation*}
where $B_u = {1\over2} b$. Note that \eqref{ap simple} is invariant under the ``reversibility'' reflection $u(x,y) \mapsto u(-x,y)$ about the $(y,z)$-plane.
We will assume in addition that it is invariant under the reflection $u \mapsto -u$, which forces
\begin{equation}
  \label{ap symmetry}
  b(\placeholder, \lambda) \text{ is odd, and hence } B(\placeholder, \lambda) \text{ is even.}
\end{equation}

The eigenvalue problem for the linearized transversal operator corresponding to  \eqref{ap simple} is simple to compute:
\begin{equation*}
 %\label{linear transverse ap} 
\left\{
  \begin{aligned}
    w_{yy} - b_u(0,0) w & = \nu w  & &\qquad \textrm{in } (-\pi/2,\pi/2) \\
    w & = 0 & & \qquad \textrm{on } \{ y = -\pi/2, \pi/2 \}.
  \end{aligned} \right. 
\end{equation*}
If the body force $b$ satisfies
\begin{equation}
  \label{ap force1}
  b_u(0,0) = -1,
\end{equation}
then $\nu = 0$ is a simple eigenvalue, and the rest of the spectrum is negative. The kernel of the linearized operator is generated by
\[
  \varphi_0(y) := \cos y.
\]

To make things concrete, we introduce a specific ansatz for the body force:
\begin{equation}
  b(u, \lambda) = -u  + \lambda b_1 u + b_2 u^3.
  \label{ap expansion b}
\end{equation}
Note that this satisfies both \eqref{ap symmetry} and \eqref{ap force1}.  One can of course add higher-order terms in $u$ if desired; see Appendix~\ref{subsec ap iteration}. Following \eqref{reparametrization equation}, we reparametrize $\lambda = \lambda_2 \varepsilon^2$. The model \eqref{ap simple} then becomes
\begin{equation}\label{ap simpler}
    \left\{ 
    \begin{aligned} %\label{momentum}
      \Delta u + 2 w_1\nabla \cdot \left( |\nabla u|^2 \nabla u \right) + u - b_1\lambda_2\varepsilon^2 u  & =  0 & &  \text{in } \Omega \\
      u & = 0 \quad & & \text{on } \partial\Omega.
    \end{aligned} 
    \right. 
\end{equation}
\begin{figure}%[h]
  \centering
  \includegraphics[scale=1.1]{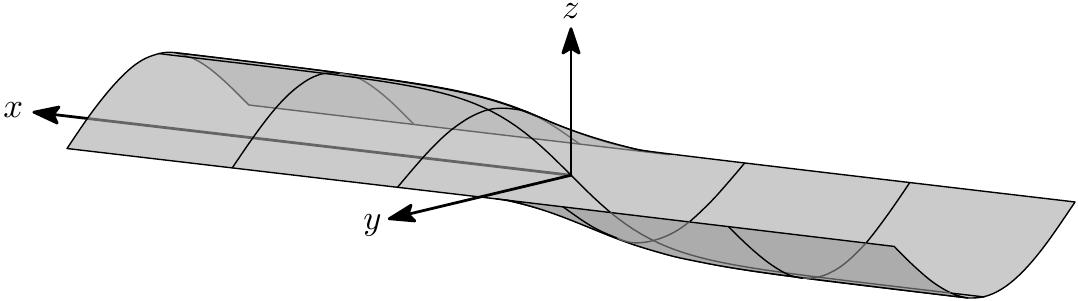}
  \caption{Leading-order approximation of the front-type solutions in Theorem~\ref{thm ap}\ref{thm ap front}. The graph $z=u(x,y)$ is the image of the strip $\{z=0, \abs y < \pi/2 \}$ under the anti-plane deformation \eqref{deformation}.}
  \label{fig:elasticity}
\end{figure}

\begin{theorem}[Fronts in anti-plane shear deformation] \label{thm ap}
  Consider the anti-plane shear problem \eqref{ap simpler} with strain energy $W$ given by \eqref{ap W simple} and a live body force $b$ of the form \eqref{ap expansion b} with $b_2 = 0$. 
  \begin{enumerate}[label=\rm(\alph*)]
  \item \label{thm ap front} When $b_1\lambda_2 < 0$, $w_1 > 0$, there exists $\varepsilon_0 > 0$ and a family of front-type solutions 
    \[  \left\{ (u^\varepsilon, \varepsilon) \in C_\bdd^{2+\alpha}(\overline{\Omega}) \times \R: -\varepsilon_0 < \varepsilon < \varepsilon_0 \right\} \]
  bifurcating from the unforced state $(u,\varepsilon) = (0,0)$.  It exhibits the asymptotics:
  \begin{equation}
    u^\varepsilon(x,y) = a_1\varepsilon \tanh{\left( {\kappa_1 \varepsilon x} \right)} \cos(y) + O(\varepsilon^2) \quad \textrm{in } C_\bdd^{2+\alpha}(\overline{\Omega}),
    \label{ap asymptotics1}
  \end{equation}
  where $\displaystyle a_1 = \sqrt{{-2b_1 \lambda_2 \over 3w_1}}, \ \ \kappa_1 = \sqrt{{-b_1\lambda_2 \over 2}}$.
\item \label{thm ap homoclinic} When $b_1 \lambda_2  > 0$ and $w_1 < 0$, there exists $\varepsilon_0 > 0$ and a family of homoclinic-type solutions 
  \[  \left\{ (u^\varepsilon, \varepsilon) \in C_\bdd^{2+\alpha}(\overline{\Omega}) \times \R: -\varepsilon_0 < \varepsilon < \varepsilon_0 \right\} \]
  bifurcating from the unforced state $(u,\varepsilon) = (0,0)$.  It exhibits the asymptotics:
  \begin{equation}
    u^\varepsilon(x,y) = a_1 \varepsilon \sech{\left( \kappa_1 \varepsilon x \right)} \cos(y) + O(\varepsilon^2) \quad \textrm{in } C_\bdd^{2+\alpha}(\overline{\Omega}),
    \label{ap asymptotics2}
  \end{equation}
  where $\displaystyle a_1 = \sqrt{{-4b_1 \lambda_2 \over 3w_1}}, \ \kappa_1 = \sqrt{b_1 \lambda_2 }$.
  \end{enumerate}
\end{theorem}
See Figure~\ref{fig:elasticity} for an illustration of the solutions in case \ref{thm ap front}.
\begin{remark}
  It is worth emphasizing that more detailed information about $u^\varepsilon$ can be obtained by combining Remark~\ref{Phi asymptotics} and the form of the reduced ODE \eqref{reduced ODE} found in Section~\ref{subsec ap truncate} below. For instance, it is possible to check that $u^\varepsilon$ inherits the monotonicity properties (in the axial variable $x$) of its leading-order approximation in \eqref{ap asymptotics1} or \eqref{ap asymptotics2}.
\end{remark}
\begin{remark}\label{rk ap cubic force}
Including the cubic term in \eqref{ap expansion b} for the body force allows one to treat more general rubber elastic material. In that setting, there exist families of front-type solutions \eqref{ap asymptotics1} when $b_1 \lambda_2 < 0$ and $b_2 + 2w_1 > 0$, and homoclinic solutions of the form \eqref{ap asymptotics2} when $b_1 \lambda_2 > 0$ and $b_2 + 2w_1 < 0$; see Appendix~\ref{subsec ap iteration}.
\end{remark}

\subsection{Center manifold reduction} 
The linearized operator of \eqref{ap simple} at $(u,\varepsilon) = (0,0)$ with assumptions \eqref{ap W simple} and \eqref{ap force1} is simply
\begin{equation*}%\label{ap L}
L := 1+\Delta \colon \Xspace_\mu \to \Yspace_\mu,
\end{equation*}
where 
\[ \Xspace_\mu := \left\{ u \in C_\mu^{2+\alpha}(\overline{\Omega}) : u|_{\partial \Omega} = 0 \right\}, \qquad \Yspace_\mu := C_\mu^{0+\alpha}(\overline{\Omega}).\]
Here, we are exploiting the fact that the boundary conditions are linear by including them in the definition of $\Xspace_\mu$.  The kernel of $L$ is the two-dimensional space
\[
\ker L = \left\{ u(x,y) = (A + Bx) \varphi_0(y): \ A, B \in \mathbb{R} \right\}.
\]
The bounds for the partial Green's function follow exactly from Proposition~\ref{partial green lemma}. As for the projection $\FSproj$ onto the kernel in Remark~\ref{projection remark}, we choose it to be
\begin{equation*}
  \FSproj u := \left( v(0) +  v^\prime(0) x \right) \varphi_0(y) \qquad \textrm{where } v(x) := u(x,0). %\label{ap def Q} 
\end{equation*}

Applying Theorem~\ref{reduction theorem}, we find that all small solutions $(u,\varepsilon)$ of \eqref{ap simpler} are of the form
\begin{align*}
  u(x,y) = v(0)\varphi_0 + v'(0)x\varphi_0 + \Psi(v(0),v'(0),\varepsilon)(x,y)
\end{align*}
for a $C^{M+1}$ coordinate map $\Psi \maps \R^3 \to C^{2+\alpha}_\mu$. The function $v$ then satisfies the reduced ODE
\begin{equation}\label{ap ODE}
  v'' = f(v, v', \varepsilon),
  \qquad 
  \text{where} \quad
  f(A, B, \varepsilon) := {d^2 \over dx^2} \Big|_{x=0} \Psi(A,B,\varepsilon)(x, 0).
\end{equation}
From the reversibility symmetry $u(x,y) \mapsto u(-x,y)$ of \eqref{ap simpler}, we deduce that
\begin{subequations}\label{ap Psi symmetry}
  \begin{align}\label{ap reversal Psi}
    \Psi(A,-B,\varepsilon)(-x,y) 
    = \Psi(A,B,\varepsilon)(x,y),
  \end{align}
  while the additional symmetry $u \mapsto -u$ implies that 
  \begin{align}
    \label{ap reflection Psi}
    \Psi(-A,-B,\varepsilon)(x,y) = \Psi(A,B,\varepsilon)(x,y).
  \end{align}
\end{subequations}
Plugging \eqref{ap Psi symmetry} into \eqref{ap ODE}, we find that $f$ has the symmetries
\begin{equation}\label{ap f symmetry}
    f(A,-B,\varepsilon) = f(A,B,\varepsilon),
    \qquad 
    f(-A,-B,\varepsilon) = f(A,B,\varepsilon).
\end{equation}

We now use Theorem~\ref{expansion theorem} to expand the coordinate map $\Psi$ and hence the function $f$. That is, we seek solutions $u \in \Xspace_\mu$ with the Faye--Scheel ansatz
\begin{equation}
  \label{ap FS ansatz}
  u(x,y) = (A + Bx) \varphi_0(y) + \sum_{\mathcal J} \Psi_{ijk}(x,y) A^i B^j \varepsilon^k + \mathcal R,
\end{equation}
where the set
\begin{equation}\label{ap iteration}
\mathcal{J} := \left\{ (i,j,k) \in \mathbb N^3: \ i+2j+k \le 3, \ i + j + k \ge 2, \ i + j \ge 1 \right\},
\end{equation}
and the error term $\mathcal R$ is of the order $O\left( (|A| + |B|^{1/2} + |\varepsilon|)^4 \right)$ in $\Xspace_\mu$. This truncation anticipates a scaling where $A \sim \varepsilon, \ B \sim \varepsilon^2$ (for more details please refer to Appendix \ref{subsec ap iteration}). Recall from Theorem~\ref{expansion theorem} that $\Psi_{ijk}(0,0) = \partial_x \Psi_{ijk}(0,0) = 0$.
% \begin{equation*}
%   %\label{proj cond}
%   \Psi_{ijk}(0,0) = \partial_x \Psi_{ijk}(0,0) = \mathcal R(0,0) = \partial_x \mathcal R(0,0) = 0.
% \end{equation*}

Plugging \eqref{ap FS ansatz} into the nonlinear term in \eqref{ap simpler}, we obtain
\[
\nabla \cdot (|\nabla u|^2 \nabla u) = -A^3 \nabla \cdot \begin{pmatrix} 0 \\ \sin^3 y \end{pmatrix} + O \left( (|A| + |B|^{1/2} + |\varepsilon|)^4  \right) \qquad \textrm{in } \Xspace_\mu.
\]
Therefore, for each $(i,j,k) \in \mathcal{J}$,  the equation for $\Psi_{ijk}$ is
\begin{equation*}
%\label{ap eqn for psi}
  \begin{cases}
    \displaystyle \sum_{\mathcal J} L(\Psi_{ijk}) A^iB^j \varepsilon^k = b_1 \lambda_2 \varepsilon^2 (A+Bx) \cos y + b_1\sum_{\mathcal J} \Psi_{ijk} A^iB^j \varepsilon^{k+2} + 6 w_1 A^3\sin^2(y) \cos y   \\
    \FSproj \Psi_{ijk} = 0.
  \end{cases}
\end{equation*}
By Lemma~\ref{lem bordering}, the above problem has a unique solution, and indeed we find:
% Recall that $\mathcal Q$ is the projection onto the kernel of $L$, and thus the above problem can be solved explicitly:
%
%Matching up the orders allows one to solve all the $\Psi$'s iteratively. In particular we obtain
\begin{align*}
& \Psi_{101} = \Psi_{011} = \Psi_{110} = \Psi_{200} = 0, \\
& \Psi_{102} = {b_1 \lambda_2 \over 2} x^2 \cos y, \\
& \Psi_{300} = {3w_1 \over 16} \left( 4x^2 \cos y - \cos y + \cos (3y) \right). 
\end{align*}

\subsection{Reduced ODE and truncation} \label{subsec ap truncate}
From Theorem \ref{reduction theorem} we know that a small solution $u$ of \eqref{ap simpler} solves the reduced ODE of the form \eqref{reduced ODE} where $v(x) = u(x, 0)$. Using the computed values of $\Psi_{ijk}$ we see that
\begin{align*}
f(A, B, \varepsilon) & = \sum_{\mathcal J} {d^2 \over dx^2} \Big|_{x=0} \Psi_{ijk}(x, 0) A^i B^j \varepsilon^k + r(A, B, \varepsilon) %\\ %O\left( (|A| + |B|^{1/2} + |\varepsilon|)^4  \right) \\
 = b_1 \lambda_2 A\varepsilon^2 + {3w_1 \over 2} A^3 + r(A, B, \varepsilon)
\end{align*}
where the error term $r \in C^{M+1}$ and
\[
r(A, B, \varepsilon) = O\left( |A|(|A| + |B|^{1/2} + |\varepsilon|)^3 + |B|(|A| + |B|^{1/2} + |\varepsilon|)^2  \right).
\]
Setting $r = 0$, we obtain the truncated reduced ODE
\[
v^0_{xx} = b_1 \lambda_2 \varepsilon^2 v^0 + {3w_1 \over 2}(v^0)^3.
\]
When $b_1 \lambda_2 < 0$ and $w_1 > 0$, this has an explicit heteroclinic orbit,
\[
  v^0(x) = a_1 \varepsilon \tanh\left( \kappa_1 \varepsilon x \right), \quad \text{where} \quad a_1 := \sqrt{{-2b_1 \lambda_2 \over 3w_1}}, \ \ \kappa_1 := \sqrt{{-b_1 \lambda_2  \over 2}}.
\]
On the other hand, when $b_1 \lambda_2 > 0$ and $w_1 < 0$, there is a homoclinic solution
\[
  v^0(x) = a_1\varepsilon \sech\left( \kappa_1 \varepsilon x \right), \quad \text{where}\quad a_1 := \sqrt{{-4b_1 \lambda_2  \over 3w_1}}, \ \ \kappa_1 := \sqrt{b_1 \lambda_2}.
\]

\subsection{Proof of existence}
It remains now to confirm that the homoclinic and heteroclinic orbits above persist for the full reduced ODE (that is, when $r$ is reintroduced).  For the heteroclinic case, it is often useful to examine invariant quantities. Here, however, the symmetry properties in  \eqref{ap f symmetry} are strong enough that a simpler argument is possible.
\begin{proof}[Proof of Theorem~\ref{thm ap}]
  Introducing the scaled variables
  \begin{align*}
    x = \varepsilon X,
    \qquad 
    v(x) = \varepsilon V(X),
  \end{align*}
  the reduced equation \eqref{ap ODE} can be written as the planar system
  \begin{equation*} \left\{ \begin{aligned} 
V_X & = W \\
W_X & = b_1 \lambda_2  V + \frac{3w_1}2 V^3 + R(V, V_X, \varepsilon),
\end{aligned} \right.
%\label{ap reduced planar ODE} 
\end{equation*}
%  \begin{align*}
%    V_{XX} = b_1 V + \frac{3a}2 V^3 + R(V, V_X, \varepsilon).
%  \end{align*}
where the rescaled error term $R(V,W,\varepsilon) = O\left( |\varepsilon| (|V| + |W|) \right)$.  At $\varepsilon = 0$, this corresponds to a rescaling of the truncated equation. 
 
Consider the situation in part~\ref{thm ap front}, where $b_1 \lambda_2 <0$, $w_1>0$.  At $\varepsilon = 0$, the explicit solution $V = a_1 \tanh(\kappa_1 X)$ crosses the $W$-axis transversely. As usual, this implies that for small nonzero $\varepsilon$, the unstable manifold of the negative equilibrium will transversely intersect the $W$-axis. Combining the reversibility symmetry $(V(X),W(X)) \mapsto (V(-X), -W(-X))$ with the reflection symmetry $(V(X),W(X)) \mapsto (-V(X), -W(X))$, we obtain existence of a (reversible) heteroclinic orbit connecting the two nontrivial equilibria.

A similar argument works for part~\ref{thm ap homoclinic}, where $b_1 \lambda_2 >0$, $w_1<0$. When $\varepsilon =0$, the explicit solution $V = a_1 \sech(\kappa_1 X)$ crosses the $V$-axis transversely. This intersection persists for small $\varepsilon$, and reversibility then guarantees the existence of a (reversible) homoclinic orbit to the origin.
\end{proof}

\section{Fronts in 2D Fisher--KPP}\label{sec fkpp}

As a second application of our general theory, we consider a reaction diffusion equation arising in mathematical biology.  The classical Fisher--KPP equation \cite{fisher1937wave,kolmogorov1937etude} is the one-dimensional problem
\begin{equation}
  v_t = v_{xx} + \sigma v(\rho^2-v),\label{fkpp 1-d} 
\end{equation}
where $v = v(t,x) : \mathbb{R}_+ \times \mathbb{R} \to \mathbb{R}$.   This models the propagation of an allele within a population; $\sigma > 0$ measures the advantageousness of the mutant gene, while $\rho^2 > 0$ describes the carrying capacity.  It is well known that Fisher--KPP supports traveling fronts moving at any wave speed greater than $2\rho \sqrt{\sigma}$.  However, it has been observed experimentally by M\"obius, Murray, and Nelson \cite{moebius2015obstacles} that, in the presence of obstacles, invasion fronts may slow down and display two-dimensional characteristics.  Recently, Minors and Dawes \cite{minors2017invasions} proposed a two-dimensional version of Fisher--KPP with certain ``reactive'' boundary conditions as a possible explanation for this phenomenon. For traveling waves, it takes the form
\begin{equation}
  \label{fkpp} 
  \left\{
  \begin{aligned}
    \Delta u + \lambda u_x + u( \rho^2 - u) & =  0 & &\qquad \textrm{in } \mathbb{R} \times (0,1) \\
    u_y & = 0 & & \qquad \textrm{on } \{ y = 0 \} \\
    u_y + \beta u & = 0 & & \qquad \textrm{on } \{ y = 1\}.
  \end{aligned} \right. 
\end{equation}
Here the unknown $u = u(x,y)$, $\beta > 0$ is an absorption constant, $\lambda$ is the wave speed, and $\rho^2 > 0$ is the carrying capacity of the allele.  Note that Minors and Dawes discuss a slightly more general problem.  For instance, we scaled the domain to be the infinite strip of unit height $\Omega := \mathbb{R} \times (0,1)$.  Also, they allow Robin or Neumann conditions to be imposed on either boundary.

In \cite{minors2017invasions}, numerical evidence is given that the two-dimensional Fisher--KPP equation \eqref{fkpp} does indeed have fronts that move arbitrarily slowly in certain regimes. As the main contribution of this section, we rigorously prove the existence of these waves via center manifold reduction.  

\begin{theorem}[$2$D Fisher--KPP fronts] \label{fkpp theorem}
Fix $\beta > 0$, let $\rho_0 > 0$ be the unique solution to $\rho_0 \tan(\rho_0) = \beta$ on $(0,\pi/2)$, and choose a positive constant $\lambda_1 > 2$.   There exists $0 < \varepsilon_0 \ll 1$, and a family of fronts solution $(u,\lambda, \rho^2)$ to the  two-dimensional Fisher--KPP equation,
\[ \left\{ (u,\lambda, \rho^2) = (u^\varepsilon, \, \lambda_1 \varepsilon,\,  \rho_0^2 + \varepsilon^2) \in C_\bdd^{2+\alpha}(\overline{\Omega}) \times \mathbb{R} \times \mathbb{R} : -\varepsilon_0 < \varepsilon < \varepsilon_0 \right\} \]
with 
\begin{equation*}
 % \label{fkpp asymptotics} 
    u^\varepsilon(x,y) = \varepsilon^2 V^\varepsilon( \varepsilon x)  \cos(\rho_0 y) + O(\varepsilon^3) \qquad \textup{in } C_\bdd^{2+\alpha}(\overline{\Omega}).
\end{equation*}
Here, $V^\varepsilon$ is to leading order a front for the one-dimensional Fisher--KPP equation \eqref{fkpp 1-d} with carrying capacity $1/\sigma$ and $\sigma$ given by \eqref{fkpp a formula}.
\end{theorem}

\subsection{Center manifold reduction}

The first step is to choose parameters so that the spectral condition \eqref{lambda0 assumption} is satisfied.   The eigenvalue problem for the transversal linearized operator at $(u,\lambda) = (0,0)$ is simply 
\begin{equation*}
 %\label{linear transverse fkpp} 
\left\{
  \begin{aligned}
    w_{yy} + \rho^2 w & = \nu w  & &\qquad \textrm{in } (0,1) \\
    w_y & = 0 & & \qquad \textrm{on } \{ y = 0 \} \\
  w_y +\beta w & = 0 & & \qquad \textrm{on } \{ y = 1\}.\end{aligned} \right. 
\end{equation*}
An elementary calculation shows that there are no eigenvalues $\nu \geq \rho^2$, and $\nu < \rho^2$ is in the spectrum if and only if 
\begin{equation}
  \label{fkpp kernel condition}
  \tan( \sqrt{\rho^2-\nu} ) = \frac{\beta\sqrt{\rho^2-\nu}}{\rho^2 -\nu}.
\end{equation}
Taking $\beta > 0$ to be fixed, the critical value for the parameter $\rho$ is defined to be the unique $\rho_0$ so that the only nonnegative solution of \eqref{fkpp kernel condition} is $\nu = 0$.  Clearly, this occurs precisely when $\tan( \rho_0) = \beta/\rho_0$, and in that case the kernel is generated by 
\[ \varphi_0(y) := \cos(\rho_0 y).\]

Now, we reconsider the full problem posed on $\Omega$.  As in the previous application, we take advantage of the linearity of the boundary conditions by encoding them directly into the definition of the space:  let
 \[ \Xspace := \left\{ u \in C^{2+\alpha}(\overline{\Omega}) : u_y|_{y=0} = 0, ~\left(\beta u + u_y\right)|_{y=1} = 0 \right\}, \qquad \Yspace := C^{0+\alpha}(\overline{\Omega}).  \]
 with the exponentially weighted counterparts $\Xspace_\mu$ and $\Yspace_\mu$, respectively. The linearized operator at $(u,\lambda) = (0,0)$ is thus 
\begin{equation}
  L:= \Delta + \rho_0^2 \colon \Xspace_\mu \to \Yspace_\mu,\label{def fkpp L} 
\end{equation}
and its kernel is the two-dimensional subspace
\begin{equation*}
  %\label{ker fkpp L}
  \kernel{L} = \left\{ u(x,y) = (A + Bx) \varphi_0(y) : A, B \in \mathbb{R} \right\}.
\end{equation*}
We have some freedom to choose a projection $\FSproj$ onto $\kernel{L}$.  As the boundary condition the bottom of the strip is simplest, a reasonable option is to take  
\begin{equation*}
  \FSproj u := \left( v(0) +  v^\prime(0) x \right) \varphi_0(y) \qquad \textrm{where } v(x) := u(x,0). %\label{def fkpp Q} 
\end{equation*}

Applying Theorem~\ref{reduction theorem}, we infer the existence of a center manifold that must contain any sufficiently small solution to \eqref{fkpp}.  To find the corresponding reduced equation, we will use Theorem~\ref{expansion theorem} and follow the general procedure outlined in Section~\ref{general strategy sec}.  That is, we seek solutions $u \in \Xspace_\mu$ with the Faye--Scheel ansatz
\begin{equation} 
  \label{fkpp FS u}
  u(x,y) = (A + Bx) \varphi_0(y) + \sum_{\mathcal{J}} \Psi_{ijk}(x,y) A^i B^j \varepsilon^k + \mathcal R(x,y),
\end{equation}
where $\varepsilon$ is a small auxiliary parameter that smoothly measures the deviation of $(\lambda, \rho^2)$ from their critical value $(0, \rho_0^2)$: $\lambda = \lambda_1 \varepsilon, \ \rho^2 = \rho_0^2 + \varepsilon^2$ as in Theorem \ref{fkpp theorem}.
%\begin{equation*}
%  %\label{exp parameter}
%  \lambda = \lambda_1 \varepsilon, \qquad \rho^2 = \rho_0^2 + \varepsilon^2.
%\end{equation*}
The sum in \eqref{fkpp FS u} ranges over %the set
\begin{equation}\label{fkpp iteration}
\mathcal{J} := \left\{ (i,j,k) \in \mathbb{N}^3 : 2i+3j+k \le 4, \ i + j + k \ge 2, \ i + j \ge 1 \right\},
\end{equation}
and the error term 
\[ \mathcal{R} = O\left( (|A|^{1/2} + |B|^{1/3} + |\varepsilon|)^5 \right) \qquad \textrm{in } \Xspace_\mu.\]
Note that, in contrast to the previous section, the truncation condition anticipates an eventual scaling where $A \sim \varepsilon^2, \ B \sim \varepsilon^3$. Computing the coefficients $\Psi_{ijk}$ can be performed according to the general strategy. The details can be found in Appendix \ref{subsec fkkp iteration}.

\subsection{Reduced ODE and truncation}
Having the coefficients $\Psi_{ijk}$ in hand, we may then apply Theorem~\ref{reduction theorem}(i) to calculate the reduced ODE. Letting $v := u(\placeholder, 0)$, we see it is given by \eqref{reduced ODE} with 
\[
f(A,B,\varepsilon) = \sum_{ \mathcal J} {d^2 \over dx^2}\Big|_{x=0} \Psi_{ijk}(x,0) A^i B^j \varepsilon^k + r(A,B,\varepsilon),
\]
where the remainder term $r \in C^{M+1}$ satisfies
\[ r(A,B,\varepsilon) = O\left( |A| ( |A|^{1/2} + |B|^{1/3} + |\varepsilon| )^3 + |B| ( |A|^{1/2} + |B|^{1/3} + |\varepsilon| )^2 \right) \]
in some neighborhood of $(0,0,0)$.  Inserting the computed values of $\Psi_{ijk}$, reveals that \begin{equation}
  v'' =  \sigma v^2 - \varepsilon^2 v -\lambda_1 \varepsilon v' + r(v, v^\prime, \varepsilon),
 \label{fkpp reduced ODE}
\end{equation}
where
\begin{equation}
  \sigma := \frac{4}{3} \frac{\sin(\rho_0) (3-\sin^2(\rho_0))}{2\rho_0+\sin(2\rho_0)} > 0,\label{fkpp a formula} 
\end{equation}
because $\rho_0 \in (0,\pi/2)$.  Rearranging \eqref{fkpp reduced ODE} slightly and truncating the remainder term, this becomes the following one-dimensional Fisher--KPP equation:
\begin{equation*}
  v^0_{xx} + \lambda_1 \varepsilon v^0_x+  \sigma v^0 \left( \frac{\varepsilon^2}{\sigma} -  v^0 \right)  = 0 .%\label{fkpp truncated ODE} 
\end{equation*} 

\subsection{Proof of existence}
In contrast to the elasticity problem in Section~\ref{sec ap}, the 2D Fisher--KPP system \eqref{fkpp} lacks reversibility and reflection symmetry. In their place, we make use of the robustness of the heteroclinic solutions to the 1D Fisher--KPP equation.
\begin{proof}[Proof of Theorem~\ref{fkpp theorem}]
Working in the scaled variables,
\[ x = \varepsilon X, \qquad v^0(x) = \varepsilon^2 V^0(X),\]
we see that $V^0$ solves
\[ -\lambda_1 V^0_X = V^0_{XX} + \sigma V^0 \left( \frac{1}{\sigma} -  V^0 \right). \] 
In the usual way, this can be converted to a first-order planar system
\begin{equation}
  \left\{ \begin{aligned} 
    V^0_X & = W^0 \\
    W^0_X & = -\sigma V^0 \left( \frac{1}{\sigma} -   V^0 \right)  -\lambda_1 W^0,
  \end{aligned} \right.
  \label{fkpp truncated planar ODE} 
\end{equation}
which has rest points $(V_+^0, W_+^0) := (0,0)$ and $(V_{-}^0, W_{-}^0) := (1/\sigma, 0)$.  A quick calculation shows that, for any $\lambda_1 > 2$, $(V_+^0,W_+^0)$ is a sink while $(V_-^0,W_-^0)$ is a saddle.  Following the classical argument of Kolmogorov, Petrovsky, and Piskunov \cite{kolmogorov1937etude}, one can show that there exists a triangular region 
\[ {\mathscr{T}^0} = \left \{ (V, W) \in \mathbb{R}^2 : W < 0,~ W+c_1 V > 0, ~W-c_2 (V- V_-^0 ) > 0 \right\}, \]
for some explicit $c_1, c_2 > 0$, so that (i) the vector field for \eqref{fkpp truncated planar ODE} enters $\mathscr{T}^0$ transversally along each of the boundary components, and (ii) the unstable manifold at $(V_-^0,W_-^0)$ enters $\mathscr{T}^0$ non-tangentially there.  As a result, $\mathscr{T}^0$ is positively invariant, and one can conclude that there exists a heteroclinic orbit $(V^0, W^0)$ contained in $\mathscr{T}^0$ and satisfying $V^0(X) \to  V_\pm^0$ as $X \to \pm\infty$. 
% It follows that 
%\[ v_0^\varepsilon(x) := \varepsilon^2 \tilde v_0(\varepsilon x),\]
%is a solution to the unscaled truncated problem \eqref{}.  
\begin{figure}%[h]
  \centering
  \includegraphics[scale=1.1]{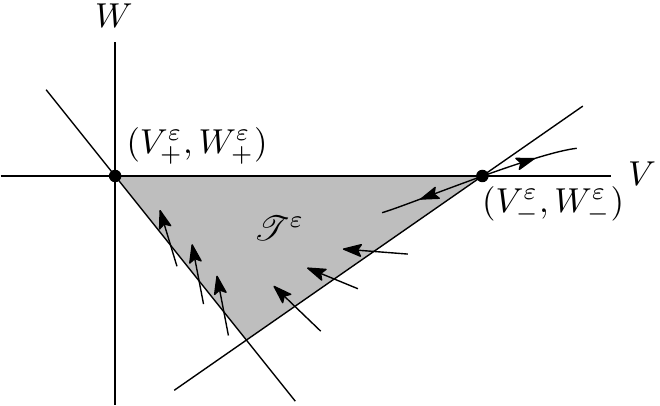}
  \caption{The positively invariant triangular region $\mathscr T^\varepsilon$} %from the proof of Theorem~\ref{fkpp theorem}.}
  \label{fig:triangle}
\end{figure}

Finally, we must show that this orbit persists for the full reduced equation \eqref{fkpp reduced ODE}.  Applying the same rescaling $x \mapsto X$ and $v \mapsto V$ gives the planar system
\begin{equation}
  \left\{ \begin{aligned} 
    V_X & = W \\
    W_X & = -\sigma V \left( \frac{1}{\sigma} -   V \right)  -\lambda_1 W + R(V,W,\varepsilon),
  \end{aligned} \right.
  \label{fkpp reduced planar ODE} 
\end{equation}
  where the remainder term $R(V, W, \varepsilon) = O(\varepsilon ( |V| + |W| ) )$.  At $\varepsilon = 0$, this is precisely the truncated problem \eqref{fkpp truncated planar ODE}.  Moreover, for each $\varepsilon \geq 0$ sufficiently small, \eqref{fkpp reduced planar ODE} has two rest points, $(V_\pm^\varepsilon, W_\pm^\varepsilon )$, with $(V_+^\varepsilon, W_+^\varepsilon) = (0,0)$, and  $(V_-^\varepsilon, W_-^\varepsilon) = ( V_-^0 + O(\varepsilon), 0)$.   It follows from the robustness of transversal intersections that there is a positively invariant triangular region $\mathscr{T}^\varepsilon$ for \eqref{fkpp reduced planar ODE} that limits to $\mathscr{T}^0$ as $\varepsilon \to 0$; see Figure~\ref{fig:triangle}.   By the same reasoning as above, we have that $\mathscr{T}^\varepsilon$ contains a heteroclinic orbit $(V^\varepsilon, W^\varepsilon)$ satisfying $V^\varepsilon \to V_\pm^\varepsilon$ as $X \to \pm\infty$.  The theorem now follows by undoing the scaling.
\end{proof}

\section{Rotational bores in a channel} \label{sec ww}
Our final application, and our initial motivation for writing this paper, pertains to water waves. Like the anti-plane shear problem in Section~\ref{sec ap}, it has a reflection symmetry in $x$, and so we expect to have to expand $f(A,B,\varepsilon)$ to third order in $A$ to obtain fronts. Unlike the anti-plane shear problem, however, there is no additional reflection symmetry in $u$. Thus the existence and persistence of heteroclinic orbits can no longer be described in terms of a transverse intersection in the plane, and we must instead introduce a second physical parameter. To solve for this auxiliary parameter in terms of $\varepsilon$, we will make heavy use of a conserved quantity called the \emph{flow force}~\cite{benjamin1984impulse}. In particular, we will investigate the so-called \emph{conjugate flow} equations which give a necessary condition for the existence of a front connecting two $x$-independent solutions~\cite{benjamin1971unified}. This analysis is quite involved, so much so, in fact, that the expressions for the Taylor coefficients of the coordinate map $\Psi$ in Theorems~\ref{reduction theorem} and \ref{expansion theorem} are too large to reproduce here. For this reason we will also highlight several important special cases where the formulas simplify drastically.

\subsection{Statement of the problem}
\begin{figure}
  \centering 
  \includegraphics[scale=1.1]{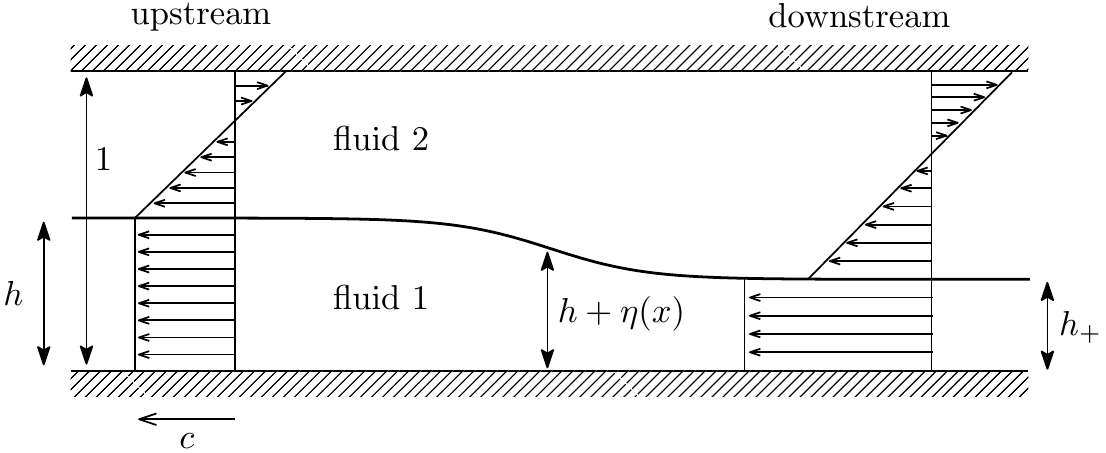}
  \caption{The class of bores under consideration. There are superposed fluid layers bounded by rigid plates at $Y=0$ and $Y=1$. The upper layer has constant density $\rho$ and constant vorticity $\omega$, while the bottom layer has unit density and zero vorticity. In the ``upstream limit'' $x \to -\infty$, the lower layer has thickness $h$, while in the downstream limit this thickness is $h_+$. At intermediate values of $x$, the layers are separated by a sharp interface a height $Y=h+\eta(x)$. In the moving frame, the upstream velocity in the lower layer is $-c$. Finally, the upstream velocity is continuous across the interface, but the downstream velocity may not be.}
  \label{fig:basicbore}
\end{figure}

Working in dimensionless variables, we consider an infinite channel bounded by horizontal walls at $Y=0$ and $Y=1$. Inside the channel there is a lower layer of fluid with density normalized to $1$, and an upper layer of lighter fluid with density $0 < \rho \le 1$. There is a sharp interface between the two layers at the height $Y=h+\eta(x)$ where $h \in (0,1)$ is a reference height to be chosen later. See Figure~\ref{fig:basicbore} for an illustration. 

Suppose that there is no surface tension along the interface and hence that the pressure is continuous across it.   For water, it is reasonable to assume that the particle velocity field is incompressible (that is, divergence free) in each fluid region.  Thus there are so-called \emph{stream functions}, $\psi_1$ in the lower fluid and $\psi_2$ in the upper fluid, so that the velocity field in the $i$-th fluid is $\nabla^\perp\psi_i := (-\dell_Y\psi_i,\dell_x\psi_i)$. Lastly, we suppose that the curl of the velocity field is some constant $\omega \in \R$ in the upper layer, but $0$ in the lower layer. Standard arguments involving Bernoulli's law then lead to the following free boundary problem for the functions $\psi_1,\psi_2,\eta$:
\begin{subequations}\label{eqn:stream}
  \begin{alignat}{2}
    \label{eqn:stream:lap1}
    -\Delta_{ x, Y}  \psi_1 &= 0 &\qquad & \fora 0 <  Y < h+ \eta, \\
    \label{eqn:stream:lap2}
    -\Delta_{ x, Y}  \psi_2 &= \omega && \fora  h+ \eta < Y < 1, \\
    \label{eqn:stream:kintop}
     \psi_2 &= - m_2  && \ona  Y=1, \\
    \label{eqn:stream:kinint}
     \psi_1 =  \psi_2 &= 0 && \ona  Y=h+ \eta, \\
    \label{eqn:stream:kinbot}
     \psi_1 &=  m_1  && \ona  Y=0, \\
    \label{eqn:stream:dynint}
    \tfrac 12 \rho\abs{\nabla_{ x, Y}  \psi_2}^2 
    - \tfrac 12 \abs{\nabla_{ x, Y}  \psi_1}^2
    + (\rho-1) \eta &=  Q && \ona Y=h+ \eta.
  \end{alignat}
  The boundary conditions \eqref{eqn:stream:kintop}--\eqref{eqn:stream:kinint} are called \emph{kinematic boundary conditions}, while \eqref{eqn:stream:dynint} is called the \emph{dynamic boundary condition}. The constants $m_1,m_2$ are the mass fluxes in each layer, while $Q$ is a Bernoulli constant. We will always consider classical solutions where the functions $\psi_1,\psi_2,\eta$ are all $C^{2+\alpha}_\bdd$ on (the closures of) their respective domains.
  
  While our methods can also be used to construct solitary wave solutions of \eqref{eqn:stream}, we will focus on the much more difficult case of fronts, sometimes called \emph{bores} in the literature. That is, we will seek solutions where $\psi_1,\psi_2,\eta$ have well-defined limits as $x \to -\infty$ (``upstream'') and as $x \to +\infty$ (``downstream'') that do not coincide. For simplicity, and because this is the case of most interest in applications, we assume that the velocity in the upstream state is continuous. The upstream limit is then is uniquely determined by requiring
  \begin{align}
    \label{eqn:stream:asym}
     \psi_{1Y}(x,0),\,\psi_{2Y}(x,0) \rightarrow -c, 
    \quad 
     \eta(x) \rightarrow 0 
    \qquad 
    \asa x \rightarrow -\infty.
  \end{align}
  Here the Froude number $c$ is a dimensionless wavespeed as measured in a reference frame where the fluid particles on the bed are stationary in the upstream limit; this is in keeping with typical conventions for periodic and solitary waves without vorticity. The second requirement that $\eta \to 0$ as $x \to -\infty$ means that $h$ is the upstream thickness of the lower fluid region.
\end{subequations}

Throughout this section we will view $\rho,\omega$ as fixed and treat $c,h$ as parameters. This is in part motivated by the fact that $\rho$ and $\omega$ are both constants of motion for the time-dependent problem.

\subsection{Main results}

Our main existence result is informally described in Theorem~\ref{thm:ww} below. A crucial part of the proof is an understanding of the so-called \emph{conjugate flow equations} which constrain the upstream and downstream depths $h,h_+$ of the lower layer and the Froude number $c$. To streamline the presentation, we defer a detailed discussion of these equations to Section~\ref{sec:conjugate} below. There, we also prove Lemma~\ref{lem:conj}, which gives sufficient conditions for the conjugate flow equations to be locally solvable for $c$ and $h_+$ in terms of $h$.

\begin{theorem}[Existence of rotational bores]\label{thm:ww}
  Consider the water wave problem \eqref{eqn:stream} with fixed density ratio $0 < \rho \le 1$ and (constant) vorticity $\omega$, and suppose that the height $h_0 \in (0,1)$ and Froude number $c_0$ satisfy the hypotheses \eqref{eqn:conjassumpt} of Lemma~\ref{lem:conj} as well as \eqref{eqn:f300pos} below. Then, for $0 < \abs\varepsilon  \ll 1$, there is a family of bore-type solutions of \eqref{eqn:stream} with upstream depths $h^\varepsilon = h_0 + \varepsilon$, Froude numbers $c^\varepsilon = c_0 + O(\varepsilon)$, and
  \begin{align}
    \label{eqn:boreexpand}
    \begin{aligned}
      \eta^\varepsilon(x) &= a_1 \varepsilon \frac{1+\tanh(\l_1\abs\varepsilon x)}2
      + O(\varepsilon^2),\\
      \psi_1^\varepsilon(x,Y) &= -c^\varepsilon (Y-h^\varepsilon) + c^\varepsilon \eta^\varepsilon(x) (1-Y) + 
      O(\varepsilon^2) ,\\
      \psi_2^\varepsilon(x,Y) &= -c^\varepsilon (Y-h^\varepsilon) 
      - \tfrac 12 \omega (Y-h^\varepsilon)^2
      + c^\varepsilon \eta^\varepsilon(x) (1+Y) + 
      O(\varepsilon^2),
    \end{aligned}
  \end{align}
  in $C^{2+\alpha}_\bdd$ of their respective domains, for some constants $a_1 \ne 0$ and $\l_1 > 0$.
\end{theorem}
\begin{remark}\label{ww monotone remark}
  The characterization of $\eta^\varepsilon$ as a solution of a second-order ODE actually furnishes much more detailed information. In particular, we can check that $\eta^\varepsilon$ inherits the strict monotonicity properties of its leading order approximation. Combining this with a maximum principle argument yields monotonicity of the full solutions; see Theorem~\ref{cats eye theorem}.
\end{remark}

The various assumptions in Theorem~\ref{thm:ww}, as well as the explicit formulas for the parameters $a_1,\l_1$ in \eqref{eqn:boreexpand}, can all be stated explicitly in terms of $h_0,c_0,\rho,\omega$. Sadly, the formulas are quite lengthy, and so it is perhaps more instructive to look at special cases. The most classical and well-studied of these is the irrotational regime where $\omega = 0$.
\begin{corollary}[Irrotational bores]\label{cor:irrot}
  The hypotheses of Theorem~\ref{thm:ww} are satisfied if we set 
  \begin{align*}
    \omega = 0, 
    \qquad 
    h_0 = \frac 1{1+\sqrt \rho}, 
    \qquad 
    c_0= \pm \frac{\sqrt{1-\rho}}{1+\sqrt \rho}. 
  \end{align*}
  The relevant family of conjugate flows $(h^\varepsilon,h_+^\varepsilon,c^\varepsilon)$ and constants $a_1,\l_1$ in \eqref{eqn:boreexpand} are given by
  \begin{align*}
    c^\varepsilon = c_0, 
    \qquad 
    h_+^\varepsilon = h_0,
    \qquad 
    a_1 = -1, 
    \qquad 
    \l_1^2 =  \frac{3(\sqrt \rho + 1)^4}{4\sqrt \rho(\rho-\sqrt \rho +1 )}.
  \end{align*}
\end{corollary}
This is the case treated by Mielke~\cite{mielke1995homoclinic}. Notice that, in particular, the solutions $(h^\varepsilon,h_+^\varepsilon,c^\varepsilon)$ have \emph{exact} formulas and that $h_+^\varepsilon$ and $c^\varepsilon$ are actually constants~\cite{laget1997numerical}. This simplifies the analysis enourmously.

When $\omega \ne 0$, interesting new phenomena can occur.  For example,  the upper fluid may contain \emph{critical layers}, curves along which $\psi_{2Y} = 0$. In the setting of Theorem~\ref{thm:ww}, there will always be such a critical layer provided $c_0$, $h_0$, and $\omega$ satisfy the inequality 
\begin{equation} \label{critical layer hypothesis}
  c_0(c_0+(1-h_0)\omega) < 0. 
\end{equation}
The upstream height of the critical layer is then $h^\varepsilon - c^\varepsilon/\omega$. Perhaps the simplest situation where this arises is when $\rho=1$ so that the fluid density is homogeneous. 
\begin{corollary}[Homogeneous-density bores]\label{cor:hom}
  The hypotheses of Theorem~\ref{thm:ww} are satisfied if we set 
  \begin{align*}
    \rho = 1,
    \qquad 
    h_0 = \frac 23
    \qquad 
    c_0 = - \frac{2\omega}9 \ne 0.
  \end{align*}
  The relevant family of conjugate flows $(h^\varepsilon,h_+^\varepsilon,c^\varepsilon)$ and constants $a_1,\l_1$ in \eqref{eqn:boreexpand} are given by
  \begin{align*}
  c^\varepsilon = c_0 + \frac \omega 3\varepsilon, 
    \qquad  
    h_+^\varepsilon = h_0 - \varepsilon,
    \qquad 
    a_1 = -2
    \qquad 
    \l_1^2 = \frac{243}{16}.
  \end{align*}
  In particular, there is an upstream critical layer at height $8/9+2\varepsilon/3$.
\end{corollary}
As with the irrotational case, we can solve the conjugate flow equations explicitly,  this time with $h_+^\varepsilon,c^\varepsilon$ both linear functions of $\varepsilon$.

Theorem~\ref{thm:ww} also applies in situations when the conjugate flow equations cannot be explicitly solved;  we then rely on Lemma~\ref{lem:conj} to guarantee the existence of solutions and also to expand them to $O(\varepsilon^2)$. We offer two examples: one with critical layers and one without.
\begin{corollary}[A ``generic'' example without critical layers]
  The hypotheses of Theorem~\ref{thm:ww} are satisfied if we set 
  \begin{align*}
    \rho = \frac {25}{52},
    \qquad 
    \omega = -\frac 9{10},
    \qquad 
    h_0 = \frac 23,
    \qquad 
    c_0 = \frac 12.
  \end{align*}
  The relevant family of conjugate flows $(h^\varepsilon,h_+^\varepsilon,c^\varepsilon)$ and constants $a_1,\l_1$ in \eqref{eqn:boreexpand} are given by
  \begin{gather*}
    c^\varepsilon = c_0 - \frac 6{29} \varepsilon + O(\varepsilon^2),
    \qquad 
    h_+^\varepsilon = h_0 - \frac{179}{725} \varepsilon + O(\varepsilon^2),
    \qquad 
    a_1 = -\frac{904}{725},
    \qquad 
    \l_1^2 = \left(\frac{226}{145}\right)^2 \frac{243}{43}.
  \end{gather*}
  None of these solutions have critical layers.
\end{corollary}
\begin{corollary}[A ``generic'' example with critical layers]
  The hypotheses of Theorem~\ref{thm:ww} are satisfied if we set 
  \begin{align*}
    \rho = \frac 1{28},
    \qquad 
    \omega = -18,
    \qquad 
    h_0 = \frac 23,
    \qquad 
    c_0 = 1.
  \end{align*}
  The relevant family of conjugate flows $(h^\varepsilon,h_+^\varepsilon,c^\varepsilon)$ and constants $a_1,\l_1$ in \eqref{eqn:boreexpand} are given by
  \begin{gather*}
    c^\varepsilon = c_0 + \frac 34 \varepsilon + O(\varepsilon^2),
    \qquad 
    h_+^\varepsilon = h_0 + \frac{11}{10} \varepsilon + O(\varepsilon^2),
    \qquad 
    a_1 = \frac 1{10},
    \qquad 
    \l_1^2 = \frac{243}{3040}.
  \end{gather*}
  In particular, there is a critical layer upstream at height $13/18 + (25/24)\varepsilon + O(\varepsilon^2)$.
\end{corollary}

\subsection{Conjugate flows}\label{sec:conjugate}
This subsection is devoted to the statement and proof of Lemma~\ref{lem:conj} on the existence of conjugate flows. Interesting in its own right, it is also one of main tools in proving Theorem~\ref{thm:ww}. 

\subsubsection*{Upstream limit and downstream limits}
Under mild regularity assumptions, the existence of the downstream and upstream limits
\begin{align*}
  \psi_1^\pm(Y):= \lim_{x \rightarrow \pm\infty}\psi_1(x,Y), 
  \qquad 
  \psi_2^\pm(Y):= \lim_{x \rightarrow \pm\infty}\psi_2(x,Y), 
  \qquad 
  \eta^\pm := \lim_{x \rightarrow \pm\infty}\eta(x)
\end{align*}
forces $(\psi_1^\pm,\psi_2^\pm,\eta^\pm)$ to each be $x$-independent solutions of \eqref{eqn:stream}. In particular, $\psi_1^\pm$ are linear in $Y$ while $\psi_2^\pm$ are quadratic. We will generally eliminate $\eta^\pm$ in favor of the upstream thickness $h$ and downstream thickness $h_+ := h+\eta^+$ of the lower fluid.

Upstream, we have the additional restrictions \eqref{eqn:stream:asym}, as well as the continuity assumption $\psi_{1Y}^-(h)=\psi_{2Y}^-(h)$ at the interface. Thus the upstream state is completely determined by the parameters $c,h,\omega$:
\begin{align}
  \label{eqn:upstream}
  \psi_1^- = -c(Y-h),
  \qquad 
  \psi_2^- = -c(Y-h)-\tfrac 12\omega (Y-h)^2.
\end{align}
Sending $x \rightarrow -\infty$ in \eqref{eqn:stream} we recover similarly explicit formulas for the fluxes $m_1,m_2$ and Bernoulli constant $Q$:
\begin{align}
  \label{eqn:mQup}
  m_1 = ch,
  \qquad 
  m_2 = c(1-h)+\tfrac 12\omega(1-h)^2,
  \qquad 
  Q = (\rho-1)\tfrac 12 c^2.
\end{align}

Now we turn to the downstream limit. In general, we cannot require it to also have a continuous velocity field, and hence the two constants
\begin{align*}
  c_1^+ := \psi_{1y}^+(h_+),
  \qquad 
  c_2^+ := \psi_{2y}^+(h_+)
\end{align*}
may differ. In terms of $c_1^+$ and $c_2^+$, the analogues of \eqref{eqn:upstream} and \eqref{eqn:mQup} are
\begin{align}
  \label{eqn:downstream}
  \begin{gathered}
    \psi_1^+ = -c_1^+(Y-h_+),
    \qquad 
    \psi_2^+ = -c_2^+(Y-h_+)-\tfrac 12\omega (Y-h_+)^2,
    \qquad 
    m_1 = c_1^+ h,\\
    m_2 = c_2^+ (1-h_+)+\tfrac 12 \omega(1-h_+)^2,
    \qquad 
    Q = \tfrac 12 \big(\rho (c_2^+)^2 - (c_1^+)^2\big) + (\rho-1)(h_+-h).
  \end{gathered}
\end{align}
Eliminating $m_1$ and $m_2$ between \eqref{eqn:mQup} and \eqref{eqn:upstream}, we can easily solve for $c_1^+$ and $c_2^+$ in terms of the other parameters. Further eliminating $Q$ we obtain a single equation relating the remaining parameters $h,h_+,c,\rho,\omega$.  Eventually, this equation simplifies to
\begin{align}
  \label{eqn:cancelfactor}
  \frac{h_+-h}
  {(1-h_+)^2 h_+^2}
  \polydyn(h,h_+,c)
  &= 0,
\end{align}
where $\polydyn=\polydyn(h,h_+,c)$ is a polynomial its arguments (as well as $\rho,\omega$),
\begin{equation}
  \begin{split}
    \label{eqn:polydyn}
    \polydyn &:=
    \omega^2 h_+^2 (h_+-h) (2-h_+-h)^2 \rho+4 h_+^2 (2 h_+^2-c^2 h_+-4 h_+-c^2 h+2 c^2+2) \rho\\
    &\qquad + 4 c \omega (1-h) h_+^2 (2-h_+-h) \rho -4 (1-h_+)^2 (2 h_+^2-c^2 h_+-c^2 h).
  \end{split}
\end{equation}
Since we are only interested in configurations where $h_+ \ne h$ and neither $h$ nor $h_+$ is $0$ or $1$, \eqref{eqn:cancelfactor} reduces to the polynomial equation $\polydyn(h,h_+,c)=0$. 

\subsubsection*{Flow force}
To obtain a second constraint on the parameters $h,h_+,c,\omega,\rho$, we introduce a quantity called the \emph{flow force}, which is related to the conservation of momentum~\cite{benjamin1984impulse}. In our variables, it takes the form
\begin{align}
  \label{eqn:flowforce}
  \begin{aligned}
    \flow(x) &:= \int_0^{h+\eta} \left( \frac 12 \abs{\nabla\psi_1}^2 - Y + \frac 12 c^2 + h \right) dY
    \\ &\qquad 
    + \rho \int_{h+\eta}^1 \left( \frac 12 \abs{\nabla\psi_2}^2 - Y - \omega\psi_2 + \frac 12 c^2 + h \right) dY.
  \end{aligned}
\end{align}
For solutions of \eqref{eqn:stream}, one can check that this quantity is independent of $x$. In particular, sending $x \to \pm\infty$ and simplifying we eventually obtain the polynomial equation
\begin{align*}
  % \label{eqn:polyflow}
  0 = \polyflow(h,h_+,c) &:= 
  \omega^2 h_+ (h_+-h) (h_++3 h-4) \rho
  +12 h_+ (h_+-c^2-1) \rho
  \\&\qquad
  +12 c \omega (h-1) h_+ \rho 
  -12 (h_+-1) (h_+-c^2).
\end{align*}
Here as above we have used our assumptions that $h_+ \ne h$ and $h,h_+ \ne 0,1$ to drop some nonzero factors. 

\subsubsection*{Constructing conjugate flows}
The equations $\polydyn = \polyflow = 0$ are called the \emph{conjugate flow equations} for our problem~\cite{benjamin1971unified}. Because of a degeneracy in this system when $h_+=h$, it will be easier to work with an equivalent system where the polynomial $\polyflow$ is replaced by 
\begin{align}
  \label{eqn:polynew}
   \polynew(h,h_+,c) := 
   \frac{2(h-1)h}{h_+-h}\bigg(\polyflow(h,h_+,c) - 
   \frac{\polyflow(h,h,c)}
  {\polydyn(h,h,c)}\polydyn(h,h_+,c)\bigg),
\end{align}
which one can verify is also a polynomial in its arguments (as well as $\omega,\rho$). We denote this ``desingularized'' set of conjugate flow equations by
\begin{align}
  \label{eqn:conj}
  \polyconj(h,h_+,c) := \big(\polydyn(h,h_+,c),  \polynew(h,h_+,c)\big) = 0,
\end{align}
where $\polydyn$ and $\polynew$ are defined in \eqref{eqn:polydyn} and \eqref{eqn:polynew} above.

Using the implicit function theorem, it is now straightforward to give conditions guaranteeing the existence of a one-parameter families of conjugate flows, that is, solutions $(h,h_+,c)$ of \eqref{eqn:conj}. We record one such result in the following lemma.

\begin{lemma}[Existence of conjugate flows]\label{lem:conj}
  For a fixed density $\rho$ and vorticity $\omega$, suppose that the depth $h_0 \in (0,1)$ and Froude number $c_0 \ne 0$ satisfy 
  \begin{subequations}\label{eqn:conjassumpt}
    \begin{align}
      \label{eqn:bif}
      \polyconj(h_0,h_0,c_0) = 0
    \end{align}
    as well as the nondegeneracy conditions
    \begin{align}
      \label{eqn:nondegen}
      \det \polyconj_{(h,h_+)}(h_0,h_0,c_0) \ne 0,
      \quad 
      \det
      \big(\polyconj_h + \polyconj_{h_+},\,
      \polyconj_c\big)
      (h_0,h_0,c_0) \ne 0.
    \end{align}
  \end{subequations}
  Then there exists a family of solutions $\{(h^\varepsilon,h_+^\varepsilon,c^\varepsilon)\}$ to the conjugate flow equations \eqref{eqn:conj} for $\abs \varepsilon < \varepsilon_0 \ll 1$ that depends analytically on $\varepsilon$ and satisfies 
  \begin{align*}
    h^\varepsilon &= h_0+\varepsilon,\\
    h_+^\varepsilon &= h_0 + h_{+,1} \, \varepsilon + h_{+,2} \, \varepsilon^2 + O(\varepsilon^3),\\
    c^\varepsilon &= c_0 + c_1 \varepsilon + c_2 \varepsilon^2 + O(\varepsilon^3).
  \end{align*}
  Moreover, $h_{+,1} \ne 1$ so that, perhaps after shrinking $\varepsilon_0$,  $h^\varepsilon \ne h_+^\varepsilon$ for $\varepsilon \ne 0$. Thus these conjugate flows are nontrivial in that the upstream and downstream states are distinct.
\end{lemma}

\subsection{Reformulating the problem}
In this subsection we reformulate \eqref{eqn:stream} as the elliptic transmission problem \eqref{eqn:final} in a \emph{fixed} domain. From now on we assume that the hypotheses of Lemma~\ref{lem:conj} are satisfied so that $h^\varepsilon, h_+^\varepsilon, c^\varepsilon$ are all defined.

\subsubsection*{Flattening the interface}
Our problem \eqref{eqn:stream} is a free boundary problem in that the interface $Y=h^\varepsilon+\eta$ between the two regions is itself an unknown.  As usual, it is helpful to switch to new coordinates where this boundary is fixed. In the absence of critical layers, one can use an elegant partial hodograph transformation in which $\psi_1,\psi_2$ are thought of as independent variables and $Y$ the dependent variable~\cite{dubreil1937theoremes}. We are interested in bores with critical layers, and therefore must allow for $\psi_2$ to be a multivalued function of $Y$. This leads us to instead make a simple piecewise-linear change of coordinates in the vertical variable $Y$: 
\begin{align}
  \label{eqn:coords}
  y &:= 
  \begin{cases}
    \displaystyle
    -1 + \frac 1{h^\varepsilon+\eta} Y  & \fora 0 < Y < h^\varepsilon+\eta \\[3ex]
    \displaystyle
    1 - \frac 1{1-h^\varepsilon-\eta} + \frac 1{1-h^\varepsilon-\eta} Y & \fora h^\varepsilon+\eta < Y < 1.
  \end{cases}
\end{align}
Thus the lower layer $0 < Y < h^\varepsilon+\eta$ is mapped onto the strip $-1 < y < 0$ while the upper layer $h^\varepsilon+\eta < Y < 1$ is mapped onto the strip $0 < y < 1$. Using subscripts $Y_1,Y_2,y_1,y_2$ to denote the vertical variables in the two layers, we have the chain rules
\begin{align}
  \label{eqn:chain}
  \begin{aligned}
    \partial_{x_1} &= \partial_x - \frac{1+y}{h^\varepsilon+\eta} \eta_x \partial_{y_1},
    &
    \partial_{x_2} &= \partial_x - \frac{1-y}{1-h^\varepsilon-\eta} \eta_x \partial_{y_2},\\
    \partial_{Y_1} &= \frac 1{h^\varepsilon+\eta}\partial_{y_1},
    &
    \partial_{Y_2} &= \frac 1{1-h^\varepsilon+\eta}\partial_{y_2},
  \end{aligned}
\end{align}
where the partials on the left hand side are with respect to the original $(x,Y)$ variables and those on the right are with respect to the transformed $(x,y)$ variables.

\subsubsection*{Subtracting off the trivial solution}
The upstream flow itself obviously solves \eqref{eqn:stream}, and so we work with normalized differences $u_1,u_2$ between our stream functions and these ``trivial'' ones:
\begin{align}
  \label{eqn:psi}
  u_1(x,y) := \frac{\psi_1(x,Y) - \psi_1^{-}(Y)}{c^\varepsilon},
  \qquad 
  u_2(x,y) := \frac {\psi_2(x,Y) - \psi_2^{-}(Y)}{c^\varepsilon}.
\end{align}
Note that the $\psi_i^{-}$ terms on the right hand side of \eqref{eqn:psi} are functions of the original variable $Y$ and not the transformed variable $y$. It is straightforward to obtain the corresponding functions of $y$ by first solving \eqref{eqn:coords} for $Y$ and then plugging into the explicit formulas \eqref{eqn:upstream}. Neither this choice nor the normalizing factor of $c^\varepsilon$ are essential, but both are convenient in later calculations. 

\subsubsection*{Final form of the equations}
\begin{subequations}\label{eqn:final}
  We now plug \eqref{eqn:psi} into \eqref{eqn:stream} and use \eqref{eqn:chain} to obtain a system of the form \eqref{transmission elliptic PDE} for $u:= (u_1,u_2)$ alone. We use one of the kinematic boundary conditions, \eqref{eqn:stream:kinint}, in order to write $\eta$ as the trace of $u_1$,
  \begin{align*}
    \eta(x) = u_1(x,0),
  \end{align*}
  thus eliminating it from the problem. Abusing notation slightly, we will nevertheless continue to write $\eta$ instead of $u_1|_\Gamma$ whenever convenient. The transformed problem is then
  \begin{alignat}{2}
    \nabla \cdot \A_1(y,u_1,\nabla u_1,u_1|_\Gamma,u_{1x}|_\Gamma,\varepsilon) &= 0 &\qquad& \ina \Omega_1 := \R \by (-1,0),\\
    \nabla \cdot \A_2(y,u_2,\nabla u_2,u_1|_\Gamma,u_{1x}|_\Gamma,\varepsilon) &= 0 &\qquad& \ina \Omega_2 := \R \by (0,1),\\
    \G(u_1,u_2,\nabla u_1, \nabla u_2,\varepsilon) &= 0 &\qquad& \ona \Gamma := \R \by \{0\},\\
    \K(u_1,u_2,\varepsilon) &= 0 &\qquad& \ona \Gamma,\\
    u_1 &= 0 && \ona \R \by \{-1\},\\
    u_2 &= 0 && \ona \R \by \{1\},
  \end{alignat}
\end{subequations}
where the functions $\A,\G,\K$ are given by
\begin{equation}\label{eqn:wwAGK}
  \begin{aligned}
    \A_1(y,u_1,\nabla u_1,u_1|_\Gamma,u_{1x}|_\Gamma,\varepsilon) &:= 
    \begin{pmatrix}
      (h^\varepsilon + \eta) u_{1x} -  (y+1) \eta_x u_{1y} \\[1ex]
      (h^\varepsilon + \eta)^{-1} ((1+y)^2\eta_x^2  + 1) u_{1y} - \eta_x (y+1) u_{1x}
    \end{pmatrix}
    \\
    \A_2(y,u_2,\nabla u_2,u_1|_\Gamma,u_{1x}|_\Gamma,\varepsilon) &:=
    \begin{pmatrix}
      (1-h^\varepsilon-\eta)u_{2x}
      -(1-y) \eta_x u_{2y}
      \\[1ex]
      (1-h^\varepsilon-\eta)^{-1} ((1+y)^2\eta_x^2+1) u_{2y}
      - (1-y)\eta_x u_{2x}
    \end{pmatrix}
    \\
    \K(y,u_1,u_2,\varepsilon) &:= u_2 - u_1 - \frac\omega{2c^\varepsilon} u_1^2\\
    \G(u_1,u_2,\nabla u_1, \nabla u_2,\varepsilon) &:=
    \frac \rho 2 \bigg(
    u_{2x}^2 - \frac{2\eta_xu_{2x} u_{2y} }{1-h^\varepsilon-\eta}
    + \frac{(1+\eta_x^2) u_{2y}^2 }{(1-h^\varepsilon-\eta)^2} 
    - \frac{2(c^\varepsilon + \omega \eta)u_{2y}}{1-h^\varepsilon-\eta} 
    \bigg)
    \\ &\qquad
    -\frac 12 \bigg(
    u_{1x}^2 - \frac{2\eta_xu_{1x} u_{1y} }{h^\varepsilon+\eta}
    + \frac{(1+\eta_x^2)u_{1y}^2}{(h^\varepsilon+\eta)^2} 
    - \frac{2c^\varepsilon u_{1y}}{h^\varepsilon+\eta} 
    \bigg)\\
    &\qquad
    + \frac{c^\varepsilon\omega \rho + \rho -1}{(c^\varepsilon)^2}\eta
    + \frac{\omega^2 \rho}{2(c^\varepsilon)^2} \eta^2.
  \end{aligned}
\end{equation}
We write \eqref{eqn:final} as $\F(u_1,u_2,\varepsilon) = 0$ where
  $\F \maps \Xspace \by \R \to \Yspace$ with $\Xspace,\Yspace$ defined in \eqref{definition X transmission} and \eqref{definition Y transmission}.

\subsection{Center manifold reduction}
The linearized operator at $(u,\varepsilon) = (0,0)$ is 
\begin{align*}
  Lu = 
  \begin{pmatrix}
    \nabla \cdot \Big(h_0 u_{1x}, h_0^{-1} u_{1y}\Big) \\
    \nabla \cdot \Big((1-h_0) u_{2x}, (1-h_0)^{-1} u_{2y}\Big) \\
    h_0^{-1} u_{1y} - \rho(1-h_0)^{-1} u_{2y}
    + c_0^{-2} (c_0 \omega \rho + \rho -1) u_1\\
    u_2-u_1
  \end{pmatrix},
  \qquad 
  L \maps \Xspace_\mu \to \Yspace_\mu
\end{align*}
which has the desired form \eqref{linear tranmission F}. Moreover, straightforward calculations using the assumption $\polydyn(h_0,h_0,c_0) = 0$ in Lemma~\ref{lem:conj} show that the spectral hypothesis \eqref{lambda0 assumption} is satisfied, and that
\begin{align*}
  \ker L = \{u(x,y) = (A+Bx)\varphi_0(y) : A,B \in \R\},
\end{align*}
where
\begin{align*}
  \varphi_0(y) := 
  \left\{
  \begin{alignedat}{2}
    1 + y &\quad& -1 &\le y \le 0 \\
    1 - y &\quad& 0 &\le y \le 1.
  \end{alignedat}
  \right.
\end{align*}
For the projection $\FSproj$ we choose
\begin{align*}
  \FSproj u := (v(0) + v'(0)x) \varphi_0(y),
  \qquad 
  \text{where } v(x) := u_1(x,0) = \eta(x).
\end{align*}
Applying Theorem~\ref{reduction theorem} as generalized in Section~\ref{extensions section}, we obtain that all small solutions $(u,\varepsilon) \in \Xspace_\bdd \by \R$ of \eqref{eqn:final} are of the form
\begin{align*}
  u(x,y) = v(0) \varphi_0 + v'(0) x \varphi_0 + \Psi(v(0),v'(0),\varepsilon)(x,y)
\end{align*}
for a $C^M$ coordinate map $\Psi \maps \R^3 \to \Xspace_\mu$. In this case the function $v$ satisfies the reduced ODE
\begin{align}
  \label{eqn:wwode}
  v'' = f(v,v',\varepsilon),
  \qquad 
  \text{where}
  \quad f(A,B,\varepsilon) := \frac{d^2}{dx^2}\Big|_{x=0} \Psi(A,B,\varepsilon)(x,0).
\end{align}
Remarkably, this is an ODE for the free surface elevation $\eta$ alone. As with the anti-plane shear problem in Section~\ref{sec ap}, the reversibility symmetry $u(x,y) \mapsto u(-x,y)$ of \eqref{eqn:final} implies that $\Psi(A,-B,\varepsilon)(x,y)= \Psi(A,B,\varepsilon)(-x,y)$ and hence that $f$ is even in $B$.

We now use Theorem~\ref{expansion theorem} to expand $\Psi$ (and thereby $f$), making the ansatz
\begin{align*}
  u(x,y) =(A+Bx)\varphi_0(y) + \sum_{\mathcal J} \Psi_{ijk}(x,y) A^i B^j \varepsilon^k + \mathcal R.
\end{align*}
Anticipating a scaling where $A \sim \varepsilon$ and $B \sim \varepsilon^2$, we work with the index set 
\begin{equation*}%\label{ww iteration}
  \mathcal{J} := \left\{ (i,j,k) \in \mathbb N^3 : i+2j+k \le 3, \ i + j + k \ge 2, \ i + j \ge 1 \right\},
\end{equation*}
so that $\mathcal R$ is $O\big( (\abs A + \abs B^{1/2} + \abs \varepsilon)^4\big)$ in $\Xspace_\mu$.

In principle, it is now straightforward to expand \eqref{eqn:final} and find the relevant $\Psi_{ijk}$ by collecting like terms and solving a sequence of linear equations of the form~\eqref{eqn for Psi}. In practice, however, these calculations are extremely tedious, partly due to the unwieldy form of the water wave problem \eqref{eqn:wwAGK} but more seriously because of the lengthy expressions for the coefficients $c_1,c_2$ in the expansion of $c^\varepsilon$ in Lemma~\ref{lem:conj}. Lastly, in order to check if complicated rational functions of $h_0,c_0$ are in fact zero, we must appeal to the highly nonlinear system of polynomial conjugate-flow equations $\polyconj(h_0,h_0,c_0) = 0$. We accomplished this latter task by transforming $\polyconj$ into a Gr\"obner basis and performing reductions using a computer algebra system. In certain situations, for instance the irrotational regime treated in Corollary~\ref{cor:irrot} and the homogeneous-density case considered in Corollary~\ref{cor:hom}, the conjugate flow equations have simple exact solutions, and so the analysis is substantially easier.

\subsection{Reduced ODE and truncation}
In the general case, we eventually find that $f(A,B,\varepsilon)$ has the form
\begin{align*}
  f(A,B,\varepsilon) &= \sum_{\mathcal J} \frac{d^2}{dx^2}\Big|_{x=0} \Psi_{ijk}(x,0) A^i B^j \varepsilon^k + r(A,B,\varepsilon) \\ 
  &= 
  f_{102} \varepsilon^2 A
  + f_{201} \varepsilon A^2
  + f_{300} A^3
  + r(A,B,\varepsilon)
\end{align*}
where the error term $r \in C^M$ and
\begin{equation*}
  r(A, B, \varepsilon) = O\left( |A|(|A| + |B|^{1/2} + |\varepsilon|)^3 + |B|(|A| + |B|^{1/2} + |\varepsilon|)^2  \right).
\end{equation*}
The coefficients are given by
\begin{align*}
  f_{300} &= \frac 32 \frac{(1-\rho)h_0^3 + c_0^2(4-5h_0)}{c_0^2 h_0^3 (1-h_0)^2 (\rho+(1-\rho)h_0)},\\
  f_{201} &=  \frac 92 \Big(c_0^2\big(1-h_0-2\rho +h_0^3 (1-\rho)^2 + h_0^2(4\rho -1)\big)-(1-h_0)^2 (3h_0^2 - 3h_0 + 2)(1-\rho)\Big)
  \\&\qquad \cdot
    \Big(
    c_0 h_0(1-h_0)^2(\rho+(1-\rho)h_0)\big(
    c_0(c_0^2 h_0^2 + (1-h_0)(c_0^2 h_0 + 2h_0 - 3 c_0^2))
  \\&\qquad\qquad
  - \omega (1-h_0)^2 h_0 (h_0 - c_0^2)\big)
    \Big)^{-1},\\
    f_{102} &= \frac{2f_{201}^2}{9f_{300}}.
\end{align*}
Using the assumptions \eqref{eqn:conjassumpt} of Lemma~\ref{lem:conj}, one can show that none of the above denominators vanish, and that $f_{300}, f_{201}, f_{102}$ are all nonzero. We additionally \emph{assume} that $f_{300} > 0$, which is equivalent to requiring that
\begin{align}
  \label{eqn:f300pos}
  h_0^3(1-\rho) + 4c_0^2(1-h_0) > c_0^2 h_0.
\end{align}
The truncated version of \eqref{eqn:wwode} is then
\begin{align}
  \label{eqn:wwode0}
  v^0_{xx} = f_{102} \varepsilon^2 v^0 + f_{201} \varepsilon (v^0)^2
  + f_{300} (v^0)^3,
\end{align}
which has the explicit solution
\begin{align}
  \label{eqn:wwv0}
  v^0(x) = a_1 \varepsilon \frac{1+\tanh(\l_1\varepsilon x)}2,
\end{align}
where
\begin{align*}
  a_1 = -\frac {2f_{201}}{f_{300}},
  \qquad 
  \l_1^2 = \frac{f_{201}^2}{18 f_{300}}.
\end{align*}

\subsection{Flow force on the center manifold}

Arguing as in Section~\ref{sec ap}, we can show that many features of the phase portrait of the truncated ODE \eqref{eqn:wwode0} persist in the full equation \eqref{eqn:wwode0}. In particular, there are three equilibria: saddles at $0$ and $a_1 \varepsilon + O(\varepsilon^2)$ and a center in between. Unfortunately, this is not enough information for the persistence of the heteroclinic orbit connecting the two saddles. For this we take advantage of the flow force $\flow$ defined in \eqref{eqn:flowforce}. 

Performing the various changes of variable, we can think of the flow force at a fixed $x$ as a functional of $(u,\varepsilon)$: $\flow = \flow(u,\varepsilon;x)$. Subtracting off its (constant) value at the trivial solution $u = 0$ and setting $x=0$, we consider the difference
\begin{align*}
  \flowdiff(u,\varepsilon) = \flow(u,\varepsilon;0) - \flow(0,\varepsilon;0).
\end{align*}
Since $\flowdiff$ only involves the values of $u$ and $\nabla u$ at $x=0$, it is a smooth function both $\Xspace_\bdd \by \R \to \R$ and $\Xspace_\mu \by \R \to \R$. We record the useful formula 
\begin{align}
  \label{eqn:Hu0}
  \flowdiff_u(0,\varepsilon) \dot u = 
  \rho(c^\varepsilon)^2(\dot u_2 - \dot u_1)(0,0)
\end{align}
for its Fr\'echet derivative at $u=0$.

When $(u,\varepsilon)$ corresponds to a solution on the center manifold, we can write
\begin{align*}
  \flowdiff(u,\varepsilon) = \flowred(v(0),v'(0),\varepsilon) ,
  \qquad 
  \text{where}
  \quad
  \flowred(A,B,\varepsilon) := \flowdiff\big((A+Bx)\varphi_0 + \Psi(A,B,\varepsilon),\, \varepsilon\big).
\end{align*}
Moreover, $\flowred(v,v',\varepsilon)$ will be constant for solutions of \eqref{eqn:wwode}. We now claim that $\flowred$ has the expansion
\begin{align}
  \label{eqn:flowredexp}
  \begin{aligned}
    \flowred(A,B,\varepsilon) 
    &= \flowred_{400} A^4 + \flowred_{301} A^3 \varepsilon + \flowred_{202} A^2 \varepsilon^2 + \tilde r(A,B,\varepsilon) \\
    &= 2 \flowred_{020}\left(\frac 12 B^2 - \frac{f_{300}}4 A^4 - \frac{f_{201}}2 A^3 \varepsilon - \frac{f_{102}}2 A^2 \varepsilon^2\right) + \tilde r(A,B,\varepsilon),
  \end{aligned}
\end{align}
where
\begin{align}
  \label{eqn:flow020}
  \flowred_{020} = - \tfrac 16c_0^2(\rho + (1-\rho)h_0) < 0
\end{align}
and the $C^M$ error term $\tilde r$ satisfies
\begin{align*}
  \tilde r(A,B,\varepsilon) = O\left( |A|(|A| + |B|^{1/2} + |\varepsilon|)^4 + |B|(|A| + |B|^{1/2} + |\varepsilon|)^3  \right).
\end{align*}
In particular, up to the nonzero factor $2 \flowred_{020}$, the truncation of $\flowred$ is precisely the Hamiltonian for the truncated ODE \eqref{eqn:wwode0}. 

Using the reversibility symmetry, we check that $\flowred$ is even in $B$, and so the smoothness of $\flowred$ implies
\begin{align}
  \label{eqn:sufficient}
  \frac{\flowred_B(A,B,\varepsilon)}B = 2\flowred_{020} + O(\abs A + \abs B + \varepsilon),
\end{align}
where $\flowred_{020} = \frac 12 \flowred_{BB}(0,0,0)$. Determining $\flowred_{020}$ in principle requires the coefficient $\Psi_{020}$ in the expansion of $\Psi$. When we actually go about calculating this coefficient and plugging it into \eqref{eqn:Hu0}, however, we see that it actually does not contribute, and that \eqref{eqn:flow020} holds. It is then straightforward to obtain \eqref{eqn:flowredexp} by combining $\flowred_{020} \ne 0$, \eqref{eqn:sufficient}, and the fact that $\flowred$ is a conserved quantity.

% It is worth emphasizing that we have by no means shown that the full ODE \eqref{eqn:wwode} or the PDE system \eqref{eqn:final} have any Hamiltonian structure. Indeed, we are unaware of any \emph{local} Hamiltonian formulation of  \eqref{eqn:stream} which allows for critical layers.

\subsection{Proof of existence}
Combining the previous subsection with Section~\ref{sec:conjugate}, we are now in a position to prove Theorem~\ref{thm:ww}.

\begin{proof}[Proof of Theorem~\ref{thm:ww}]
  Introducing the scaled variables
  \begin{equation*}
    x = {\abs\varepsilon}^{-1} X, 
    \quad 
    v(x) = \varepsilon V(X),
    \quad 
    v_x(x) = \varepsilon \abs \varepsilon W(X),
    \quad 
    \flowred(v,v_x,\varepsilon)
    = 2\varepsilon^4 \flowred_{020}\flowscaled(V,W,\varepsilon),
  \end{equation*}
  the reduced ODE \eqref{eqn:wwode} can be written as the planar system 
  \begin{equation*}%\label{ww scaled ode}
    \left\{
    \begin{aligned}
      V_X &= W \\
      W_X &= 
      f_{102} V + f_{201} V^2 + f_{300} V^3 
      + R(V,W,\varepsilon)
    \end{aligned}
    \right.
  \end{equation*}
  with conserved quantity
  \begin{align*}
    \flowscaled(V,W,\varepsilon) = 
      \frac 12 W^2 - \frac{f_{102}}2 V^2 - \frac{f_{201}}3 V^3 - \frac{f_{300}}4 V^4 
      + \tilde R(V,W,\varepsilon),
  \end{align*}
  and where the error terms satisfy
  \begin{align*}
     R(V,W,\varepsilon) = O(\abs \varepsilon (\abs V + \abs W)),
     \qquad 
     \tilde R(V,W,\varepsilon) = O(\abs \varepsilon (\abs V + \abs W)).
  \end{align*}
  When $\varepsilon = 0$, we have the explicit heteroclinic solution $V = a_1 (1+ \tanh(\l_1 X))/2$ connecting $(V,W) = (0,0)$ with $(V,W) = (a_1,0)$. This is the scaled version of $v^0$ in \eqref{eqn:wwv0}. Both of these equilibrium have same value, namely $0$, of the conserved quantity $\flowscaled$. For $\varepsilon \ne 0$, the equilibria at $(0,0)$ remains fixed while the equilibrium at $(a_1,0)$ persists but is perturbed. From Lemma~\ref{lem:conj} on conjugate flows, we in fact know that its \emph{exact} location is $(\varepsilon^{-1} h_+^\varepsilon -1, 0)$ and moreover that it continues to have $\flowscaled = 0$. The persistence of the full heteroclinic orbit then follows from its characterization as a nondegenerate level curve of the conserved quantity $\flowscaled$.
\end{proof} 

\subsection{Critical layers and streamline pattern} \label{streamline section}

Finally, in this section, we explore some qualitative features of the waves constructed above.  As we have seen, there are certain parameter regimes for which a streamline in the unperturbed flow is a critical layer.  For small bores, that streamline will split either upstream or downstream, opening into a  ``half cat's eye'' with its pupil at infinity. 

\begin{theorem}[Streamlines] \label{cats eye theorem}
  In the setting of Theorem~\ref{thm:ww}, suppose that \eqref{critical layer hypothesis} holds so that there is a critical layer, and suppose for concreteness that $\omega < 0$ (so that $c_0 > 0$) and $a_1\varepsilon < 0$. Perhaps shrinking $\varepsilon$ further, the streamlines of the corresponding solution $(\psi_1^\varepsilon, \psi_2^\varepsilon, \eta^\varepsilon)$ have the qualitative features of Figure~\ref{fig:cats}. Specifically,
  \begin{enumerate}[label=\rm(\alph*)]
  \item {\rm(Monotonicity)} \label{psi signs}
    The interface is strictly monotone with $\eta^\varepsilon_x < 0$. Moreover, $\psi^\varepsilon_x < 0$ for $Y \ne 0,1$ so that the vertical velocity is positive.
  \item {\rm(Critical layer)} \label{exist critical layer} 
    There is a unique $C^1$ curve $\mathcal{C}^\varepsilon$ in the interior of the upper fluid where $\psi_{2Y}^\varepsilon = 0$. Above this curve, $\psi_{2Y}^\varepsilon > 0$, and below it $\psi_{2Y}^\varepsilon < 0$. There are two streamlines, one above $\mathcal{C}^\varepsilon$ and one below, that both limit to $\mathcal{C}^\varepsilon$ upstream.  In the region they enclose (the eye), every streamline is a horizontally unbounded curve that opens to the right and has a unique turning point which is located on $\mathcal{C}^\varepsilon$. Outside the eye region, all streamlines extend from upstream to downstream.
  \end{enumerate}
\end{theorem}
\begin{remark}
  In \eqref{eye size equation} below we will see that the vertical extent of the eye is $O(\abs \varepsilon^{1/2})$. Changing the sign of $\omega$ (and hence $c_0$) changes the sign of the horizontal velocity throughout the fluid but preserves the streamline pattern. Changing the sign of $a_1 \varepsilon$ changes the signs of $\eta_x$ and $\psi_x$, reflecting the streamline pattern in Figure~\ref{fig:cats} but preserving the sign of the horizontal velocity.
\end{remark}

\begin{proof}
  We start by confirming monotonicity \ref{psi signs}. From the proof of Theorem~\ref{thm:ww}, our assumption that $a_1 \varepsilon < 0$, and Remark~\ref{ww monotone remark}, we know that $v' = \eta_x^\varepsilon < 0$. The asymptotics \eqref{eqn:boreexpand} also give $\psi_{2Y}^\varepsilon < 0$ along $Y=h^\varepsilon + \eta^\varepsilon$. Differentiating the kinematic condition \eqref{eqn:stream:kinint}, we see that this implies $\psi_{2x}^\varepsilon = -\eta^\varepsilon_x \psi_{2Y}^\varepsilon < 0$ there as well. But, $\psi_{2x}^\varepsilon$ is harmonic and vanishes on the upper boundary $\{ y = 1\}$ and at infinity. The maximum principle therefore implies that $\psi_{2x}^\varepsilon < 0$ in the interior of the upper fluid. Similarly, we find that $\psi_{1x}^\varepsilon < 0$ in the interior of the lower fluid.
 
  Now we turn to the more detailed claims in \ref{exist critical layer}. Setting $\varepsilon = 0$, we have
  \begin{align*}
    \psi_2^0(x,Y) = \psi_2^0(Y) = -c_0 (Y-h_0) - \tfrac 12 \omega (Y-h_0)^2.
  \end{align*}
  Differentiating, we find that $\psi_{2Y}^0 = 0$ at the unique height $Y_c^0 := h_0-c_0/\omega$, which lies in $(h_0,1)$ by \eqref{critical layer hypothesis}. Since $\psi_{2YY}^0 = -\omega > 0$ and
  $\psi_2^\varepsilon = \psi_2^0 + O(\varepsilon)$ in $C^{2+\alpha}_\bdd$ by \eqref{eqn:boreexpand}, the existence of $\mathcal{C}^\varepsilon$ now follows from the implicit function theorem. Indeed, it is the graph of a single-valued function of $x$. Moreover, $\psi_{2YY}^\varepsilon > 0$ for $0 < \varepsilon \ll 1$ so that $\psi_{2Y}^\varepsilon > 0$ above $\mathcal C^\varepsilon$ and $\psi_{2Y}^\varepsilon < 0$ below. From \eqref{eqn:boreexpand} we also have $\psi_{1Y}^\varepsilon < 0$ in the lower fluid.

  Examining the explicit formula for the upstream state \eqref{eqn:upstream}, we see that for $0 < -\varepsilon a_1 \ll 1$, the critical upstream height perturbs to  $Y_c^\varepsilon := h^\varepsilon - c^\varepsilon/\omega$ with the stream function value $(c^\varepsilon)^2/2\omega$. There are exactly two heights downstream at which the stream function takes on this value; the corresponding streamlines bound the eye region. Looking at \eqref{eqn:downstream}, we see that these heights are given by
   \begin{equation}\label{eye size equation}
   Y_c^0  \pm \frac 1\omega \sqrt{ \frac{2c_0 (c_0+\omega(1-h_0))}{1-h_0} a_1 \varepsilon} + O(\abs\varepsilon).
   \end{equation}
  %We want this to give two distinct real-valued $Y$ that lie in the range $(h_+^\varepsilon, 1)$.  
  From the assumptions $a_1 \varepsilon < 0$ and \eqref{critical layer hypothesis}, we have that the radicand is strictly positive and $O(\abs\varepsilon)$. 

  Pick any point inside the eye region. Applying the implicit function theorem, we see that the streamline through this point is globally parameterized as a graph $\{ x=\xi(Y)\}$ for some $C^1$ function $\xi$.  Moreover, the discussion above shows that $\xi_Y =0 $ only on $\mathcal{C}^\varepsilon$, and $\xi_{YY} > 0$ there.  The desired qualitative features of the streamline pattern inside the eye now follow.   On the other hand, outside this region, $\psi_Y^\varepsilon \neq 0$, so all streamlines must extend from $x=-\infty$ to $x = +\infty$.  
\end{proof}

\section*{Acknowledgements}

The research of RMC is supported in part by the NSF through DMS-1613375.  The research of SW is supported in part by the National Science Foundation through DMS-1812436.  

The authors also wish to thank the hospitality of the Erwin Schr\"odinger Institute for Mathematics and Physics, University of Vienna.  A portion of this work was completed during a Research-in-Teams Program generously supported by the ESI.  

%The research of MHW is supported in part by the NSF through DMS-1400926.

The authors are grateful to Carmen Chicone and Bente Bakker for helpful suggestions during the writing of this paper.

\appendix

\section{Amick--Turner fixed point theory} \label{at appendix}

In this section, we present a highly abbreviated version of Amick and Turner's center manifold construction in \cite{amick1989small,amick1994center}.  Rather than state those results in full generality, we have specialized to the case most relevant to our needs.  An effort has also been made to simplify and standardize the notation.
 
%Let $\Omega = \mathbb{R} \times \base \subset \mathbb{R}^n$ be an infinite cylinder with smooth base $\base$.  As usual, for a generic point in $\overline{\Omega}$, we write $(x,y) \in \mathbb{R} \times \overline{\base}$.  
Following the procedure in Section~\ref{reformulation fixed point section} leads us to study equations of the general form
\begin{equation}
 \label{at general equation}
\left\{ 
  \begin{aligned} 
    U_1 & = \xi_1 + \atF_1(U_1, U_2, R; \lambda, \beta)|_0^x  \\
    U_2 & = \xi_2 + \atF_2(U_1, U_2, R; \lambda, \beta)|_0^x \\
    R & = \atF_3(U_1, U_2, R; \lambda, \beta). 
  \end{aligned} \right. 
\end{equation}
Here, $(U_1, U_2, R)$ are the unknowns.  Motivated by \eqref{eqn:fixed}, where $U_2$ arises as a scaled derivative of $U_1$, we work in the space 
\[ W := (U_1, U_2, R) \in C_\mu^{k+\alpha}(\mathbb{R}) \times C_\mu^{k-1+\alpha}(\mathbb{R}) \times C_\mu^{k+\alpha}(\overline{\Omega}) =: \atXspace_\mu\]
for some $\mu \in [0,\overline{\mu})$, integer $k \geq 1$, and $\alpha \in (0,1)$. 
%Writing $\atF = (\atF_1, \atF_2, \atF_3)$, \eqref{at general equation} can be written as 
%\begin{equation}\label{short hand eqn}
%W = (\xi, 0) + \atF(W; \lambda, \varepsilon, \beta).
%\end{equation}
As before, let $\mathring{\atXspace}_\mu$ denote the corresponding homogeneous space, and $\atXspace_\bdd := \atXspace_0$.  

In \eqref{at general equation}, there are three types of parameters: $\xi = (\xi_1, \xi_2)$ is ``initial data'' for $U = (U_1, U_2)$; $\lambda \in \mathbb{R}$ is the main parameter of bifurcation; and $\beta \in (0,1]$ is, essentially, a rescaling of time needed to obtain a fixed point.  

Next, we impose some conditions on the nonlinear mappings in \eqref{at general equation}.   Assume that 
 \begin{equation}
 \label{at nonlinear terms form}
   \begin{split}
     \atF(W; \lambda, \beta) & = \beta \mathfrak{L} W   +  \mathfrak{H}(W; \lambda, \beta),
   \end{split}
\end{equation}
where $\mathfrak{L} = (\mathfrak L_1, \mathfrak L_2, 0)$ is a zeroth order linear mapping in the sense that
\begin{equation}
  \label{at A1} \mathfrak{L} \textrm{ is linear and  bounded } \atXspace_\mu \to \atXspace_{\mu} \textrm{ and } {\atXspace}_\bdd \to \mathring{\atXspace}_{\bdd}     
\end{equation} 
with bounds uniform in $\mu$ on compact subsets of $(0,\overline{\mu})$. 

Finally, suppose that each component of the nonlinear function $\mathfrak{H} = (\mathfrak{H}_1, \mathfrak{H}_2, \mathfrak{H}_3)$ takes the general form
\begin{equation}
  \label{at H term} \mathfrak{H}_i(W; \lambda, \beta) = \frac{1}{\beta^p} \mathfrak{I} S_g(\mathfrak{D} W; \lambda, \beta).
\end{equation}
Here, $p > 0$ corresponds roughly to the homogeneity of the nonlinearity created by $S_g$.   Intuitively, we think of $\mathfrak{D}$ as losing some number of derivatives, while $\mathfrak{I}$ is smoothing.   Between them is the mapping $S_g$, a general (parameter dependent) superposition operator.   Note that $p$, $\mathfrak{D}$, $S_g$, and $\mathfrak{I}$ can vary for each component, but we will  suppress this dependence to simplify notation.    Also, one can assume more generally that $\mathfrak{H}_i$ consists of a finite sum of terms of the form \eqref{at H term}. 

To state things more precisely, we introduce two (lower regularity) spaces:
\[ \atYspace_\mu := C_\mu^{j+\alpha}(\mathbb{R}; \mathbb{R}^\ell) \times C_\mu^{j+\alpha}(\overline{\Omega}; \mathbb{R}^m), \qquad 
\atZspace_\mu  := C_\mu^{j+\alpha}(\overline{\Omega}),\]
for some integers $j \geq 0$, $\ell,m \geq 1$ (again, each of these will in principle vary in $i$).  Now, assume that
\begin{equation}
  \label{at A2} \mathfrak{D} \textrm{ is linear and bounded } \atXspace_{\mu} \to \atYspace_\mu \textrm{ and } \atXspace_{\bdd} \to \atYspace_{\bdd}
\end{equation}
with bounds uniform in $\mu$ on compact subsets of $(0,\overline{\mu})$.    The superposition map $S_g$ is defined by
\begin{equation}
  \label{at superposition def} S_g (Y; \lambda, \beta)(x,y) := g(x,y, Y_1(x), \ldots, Y_\ell(x), Y_{\ell+1}(x,y), \ldots, Y_{\ell+m}(x,y); \lambda, \beta),   
\end{equation}
for all $Y \in \atYspace_\mu$ and $(x,y) \in \overline{\Omega}$.  Here, the function $g$ is assumed to satisfy 
\begin{equation}
 \label{at g assumptions}
  \begin{gathered} 
    g = g(x,y, w;  \lambda, \beta) \in C^{M+3}(\overline{\Omega} \times  \mathbb{R}^{\ell+m} \times \mathbb{R} \times  (0,1]; \mathbb{R}),  \\
  g(x,y,0;0,\beta) = 0,\quad g_w(x,y,0;0,\beta) = 0, \quad g_\lambda(x,y,0;0,\beta) = 0. 
  \end{gathered} 
\end{equation}
One can show that \eqref{at A2}--\eqref{at g assumptions} together ensure that
 \[ W \mapsto S_g(\mathfrak{D}W; \lambda, \beta) \qquad \textrm{is bounded } \atXspace_\mu \to \atZspace_\mu \textrm{ and } \atXspace_\bdd \to \atZspace_\bdd. \]  
 % Moreover, this mapping is smooth between these spaces; see \cite[Theorem 2.1]{amick1994center}. [Absolutely not true, you have to change the weights. Do we use this?]
Finally, $\mathfrak{I}$ is supposed to be smoothing in that it satisfies 
\begin{equation}
  \label{at A3} \mathfrak{I} \textrm{ is linear and bounded } \atZspace_\mu \to \atXspace_{i,\mu} \textrm{ and } \atZspace_{\bdd} \to \mathring{\atXspace}_{i,\bdd},  
\end{equation}
with bounds uniform in $\mu$ on compact subsets of $(0,\overline{\mu})$.

As is always the case with center-manifold constructions,
Amick and Turner do not treat \eqref{at general equation} directly but rather a truncated problem where each function $g$ in \eqref{at g assumptions} is replaced by
\begin{equation}\label{restriction composition}
  g^r(x,y,w_1,\ldots,w_{\ell+m};\lambda,\beta)
  := 
  g(x,y,\eta_r(w_1),\ldots,\eta_r(w_{\ell+m});\lambda,\beta)
\end{equation}
for an appropriate cutoff function $\eta_r$. We write the resulting fixed-point equations as 
\begin{equation}
  \label{at truncated equation}
  \left\{ 
  \begin{aligned} 
    U_1 & = \xi_1 + \atF_1^r(U_1, U_2, R; \lambda, \beta)|_0^x  \\
    U_2 & = \xi_2 + \atF_2^r(U_1, U_2, R; \lambda, \beta)|_0^x \\
    R & = \atF_3^r(U_1, U_2, R; \lambda, \beta). 
  \end{aligned} \right. 
\end{equation}
From \cite[Lemma 4.1,Theorem 4.1]{amick1994center} we know that, for each $M\in \mathbb N$, we can choose $\beta,r,\mu > 0$ sufficiently small so that $\atF^r$ has $M+1$ Lipschitz-continuous derivatives acting from $\atXspace_\mu \times \R^2 \times \R \times \R \to \atXspace_{(k+M+3)\mu}$.

\begin{theorem}[Center manifold] \label{at fixed point theorem}
  Consider the truncated fixed-point equation \eqref{at truncated equation} under the structural assumptions \eqref{at nonlinear terms form}--\eqref{at A3} enumerated above. Then, for any integer $M$, there exists $\mu \in (0,\overline{\mu})$, $r > 0$, and $\beta \in (0,1]$ so that the unique solution to \eqref{at truncated equation} is given by 
  \begin{align*}
    W = (U_1,U_2,R) =: \FP(\xi_1,\xi_2,\lambda) \in \atXspace_\mu
  \end{align*}
  where the mapping $\FP \maps \R^2 \by \R  \to \atXspace_\mu$ is $C^{M+1}$.
\end{theorem}
\begin{proof}
  This result is found by combining Theorem~3.1, Theorem~3.3, Remark 3.2, and Theorem~4.1 of \cite{amick1994center}.
\end{proof}

The coordinate mapping $\FP$ has flatness properties analogous to \eqref{Psi flatness}. A particular instance of this which we will need is the following.

\begin{lemma}\label{lem Rstructure}
Under the assumptions of Theorem~\ref{at fixed point theorem}, we have 
\begin{equation}\label{Rstructure f=0}
  |\FP_3(\xi,\lambda)(0,0)| + |\partial_x\FP_3(\xi,\lambda)(0,0)| \lesssim |\xi|(|\xi| + |\lambda|).
\end{equation}
\end{lemma}
\begin{proof}
  From the uniqueness of $\FP$ we have $\FP(0,\lambda) = 0$ for all $\lambda$. Moreover, differentiating the third equation in \eqref{at truncated equation} with respect to $\xi$ we discover
  \begin{align*}
    D_\xi \FP_3 = D_W\atF^r_3(\FP) D_\xi \FP.
  \end{align*}
  At $(\xi,\lambda) = (0,0)$, this becomes simply
  \begin{align*}
    D_\xi \FP_3(0,0) = \beta \mathfrak L_3 D_\xi \FP(0,0),
  \end{align*}
  where $\mathfrak L$ is the operator in \eqref{at A1}. But we have assumed that the third component $\mathfrak L_3$ of this operator vanishes, and so we simply obtain $D_\xi \FP_3(0,0) = 0$. Thus
    $\n{\FP_3(\xi,\lambda)}_{\atXspace_\mu} \lesssim \abs \xi (\abs \xi + \abs \lambda)$, which in particular implies \eqref{Rstructure f=0}.
\end{proof}
%\begin{remark}\label{rk Rstructure}
%One may compute more explicitly the structure of $\FP$ for the system \eqref{eqn:fixed} for our elliptic problem \eqref{main elliptic PDE}. In that case we have
%\begin{equation*}
%\left. D_W \atF \right|_{(W;\lambda)=0} (u_1, u_2, u_3) = \left( \beta \int^x_0 u_2(s)\,ds, 0, 0 \right),
%\end{equation*}
%and
%\begin{equation*}
%\left. D_{(\xi,\lambda)} \FP \right|_{(\xi,\lambda)=0} (\alpha_1, \alpha_2, \sigma) = (\alpha_1 + \beta\alpha_2 x, \alpha_2, 0).
%\end{equation*}
%Therefore Taylor expanding $\FP(\xi,\lambda)$ we have
%\begin{equation}\label{taylor W}
%\begin{split}
%\FP(\xi,\lambda) & = \FP(0,0) + \left. D_{(\xi,\lambda)} \FP \right|_{(\xi,\lambda)=0} (\xi,\lambda) + O(|(\xi,\lambda)|^2) \\
%& = (\xi_1+\beta\xi_2 x, \xi_2, 0) + O(|\xi|^2 + |\lambda|^2).
%\end{split}
%\end{equation}
%\end{remark}

\section{Supplemental calculations} \label{calculation appendix}
In this section, we provide more details about how to apply the general strategy explained in Section \ref{general strategy sec} to our application problems, with an emphasis on how to obtain the scaling information to reparametrize the problem.

\subsection{Iteration for anti-plane shear}\label{subsec ap iteration}
Recall the problem \eqref{ap simple} with the original parameter $\lambda$. Plugging in the ansatz \eqref{ap W simple} for the strain energy we can write the problem as
\begin{equation}\label{ap appendix pde}
\Delta u + 2w_1\nabla \cdot \left(|\nabla u|^2 \nabla u \right) - b(u, \lambda) = 0.
\end{equation} 

Here we consider a more general case where the body force $b$ is smooth in its arguments. Taylor expansion and using \eqref{ap symmetry} and \eqref{ap force1} we obtain 
\[
b(z,\lambda) = -z + b_1 \lambda z + {1\over2} b_{z\lambda\lambda}(0,0) \lambda^2 z + b_2 z^3 + O\left((|\lambda| + |z|)^4 \right), %\qquad \text{in} \ C_\bdd^{2+\alpha}(\overline{\Omega}),
\]
where
\[
b_1 := b_{z\lambda}(0,0), \quad b_2 :=  {1\over6}b_{zzz}(0,0).
\]

%From the standard results on the smoothness of Nemitsky operators (see, for example, \cite[Theorem 2.1]{amick1994center}) we can Taylor expand $b$ at $(0,0)$ (and plugging in  ) to obtain that
%\[
%b(u,\lambda) = -u + b_1 \lambda u + {1\over2} b_{u\lambda\lambda}(0,0) \lambda^2 u + b_2 u^3 + O((|\lambda| + |u|)^4),\qquad \text{in} \ C_\bdd^{2+\alpha}(\overline{\Omega}),
%\]
%where
%\[
%b_1 := b_{u\lambda}(0,0), \quad b_2 :=  {1\over6}b_{uuu}(0,0).
%\]
Because of the cubic term in \eqref{ap appendix pde}, we would like to expand the reduced ODE \eqref{reduced ODE} to third order. Following the general strategy in Section \ref{subsubsec iteration}, we can replace \eqref{ap appendix pde} by its truncation at order $K=3$,
\begin{equation}\label{ap truncated}
Lu = b_1 \lambda u + {1\over2} b_{u\lambda\lambda}(0,0) \lambda^2 u+ b_2 u^3 - 2w_1 \nabla \cdot \left(|\nabla u|^2 \nabla u \right). %+ O((|\lambda| + |u|)^4).
\end{equation}
Next, let us consider the anticipated scaling described in Section \ref{subsubsec scaling}.  Expanding $f$ as in \eqref{double expansion f} starting from \eqref{ap truncated}, we see that $f_A(0,0,\lambda) \sim \lambda$.  Moreover, if $w_1\ne 0$ or $b_2 \ne 0$, then the right-hand side of \eqref{ap truncated} is indeed cubic in $u$ so that $\partial^3_Af(0,0,\lambda) \sim 1$. Recalling \eqref{rescaling relation}, this predicts a balancing $\varepsilon^{2n} \sim \varepsilon^p \sim \varepsilon^{2q}$. We therefore take $n = q = 1$ and $p = 2$, which leads to the index set $\mathcal J$ given by \eqref{ap iteration}. Denote the reparametrization $\lambda  = \lambda_2 \varepsilon^2$. 
% so that $\lambda_1 b_1 = \pm 1$, and we call $\lambda_1 b_1$ the coefficient $b_1$ in \eqref{ap simpler}.

With the scaling settled, we make the ansatz $(A + Bx) \varphi_0 +\sum_{\mathcal J}\Psi_{ijk} A^i B^j \varepsilon^k$ for $u$ in \eqref{ap truncated}, obtaining
\begin{equation*}
  L\bigg( \sum_{\mathcal J} \Psi_{ijk} A^i B^j \varepsilon^k \bigg) = b_1 \lambda_2 A \varepsilon^2 \cos y + b_2 A^3 \cos^3 y + 2w_1 A^3 (\sin^3 y)_y.
\end{equation*}
Grouping like terms yields
%At $i+2j+k = 2$ we have $L \Psi_{101} = L \Psi_{200} = 0$.
%
%At $i+2j+k = 3$ we have
\begin{gather*}
  L \Psi_{101} = L \Psi_{200} = L \Psi_{110} = L \Psi_{011} = 0,\\
  L \Psi_{102} = b_1 \lambda_2 \cos y, \quad L \Psi_{201} = 0, \quad L \Psi_{300} = b_2 \cos^3 y + 6w_1 \sin^2(y) \cos(y).
\end{gather*}
Applying Lemma \ref{lem bordering} allows us to iteratively solve these equations, and ultimately we find that 
\begin{gather*}
  \Psi_{101} = \Psi_{200} = 0, \qquad \Psi_{110} = \Psi_{200} = \Psi_{201} = 0, \\
  \Psi_{102} = {1\over2}b_1 \lambda_2 x^2 \cos y, \quad \Psi_{300} = {3b_2 + 6w_1 \over 8} x^2 \cos y + {b_2 - 6w_1 \over 32} (\cos y - \cos(3y)).
\end{gather*}
Thus, 
\[
f(A, B, \varepsilon) = b_1 \lambda_2 A \varepsilon^2 + {3(b_2 + 2w_1) \over 4} A^3 + r(A,B,\varepsilon).
\]

\subsection{Calculation for FKPP}\label{subsec fkkp iteration}
As in the previous subsection, we write the PDE as
\[
L u = -\lambda u_x - (\rho^2 -\rho_0^2) u + u^2,
\]
where $L$ is the linearized operator \eqref{def fkpp L}.
From this we do not immediately see a length scale unless we assume certain parameter dependence on $\rho^2 - \rho_0^2$. The $u_x$ term imposes a compatibility condition \eqref{compatibility}, which, in the FKPP case, reads  $f_A(0,0,\lambda) \sim |\rho^2 - \rho_0^2 | \sim \lambda^2$. The quadratic term in the PDE suggests that $m=2$ and $f_{AA}(0,0,\lambda) \sim 1$. Plugging this in \eqref{rescaling relation} we see that $\varepsilon^{2n} \sim \varepsilon^{2p} \sim \varepsilon^q$, and hence one can pick $n = p =1$ and $q = 2$.  This choice corresponds to the reparametrization
\[
\lambda = \lambda_1 \varepsilon, \quad \rho^2 = \rho_0^2 + \rho_2 \varepsilon^2,
\]
and the index set $\mathcal J$ given by \eqref{fkpp iteration}.

Expressed in the new parameter regime, the PDE becomes
\begin{equation*}
  % \label{fkpp appendix pde}
  L u = -\rho_2 \varepsilon^2 u - \lambda_1 \varepsilon u + u^2.
\end{equation*}
Setting $u$ to be $(A + Bx) \varphi_0 +\sum_{\mathcal J}\Psi_{ijk} A^i B^j \varepsilon^k$, it follows that
\begin{align*}
L\bigg( \sum_{\mathcal J} \Psi_{ijk} A^i B^j \varepsilon^k \bigg) & = -\rho_2\varphi_0 A\varepsilon^2 - \lambda_1\varphi_0 B\varepsilon - \rho_2 x\varphi_0 B\varepsilon^2 - \rho_2 \sum_{\mathcal J} \Psi_{ijk} A^i B^j \varepsilon^{k+2} \\
& \qquad - \lambda_1 \sum_{\mathcal J} \partial_x \Psi_{ijk} A^i B^j \varepsilon^{k+1} + \bigg( (A + Bx) \varphi_0 + \sum_{\mathcal J} \Psi_{ijk} A^i B^j \varepsilon^k \bigg)^2,
\end{align*}
which results in four equations
\begin{align*}
L \Psi_{101} & = 0, \qquad  L \Psi_{011} = -\lambda_1 \varphi_0, \qquad L \Psi_{102} = -\rho_2 \varphi_0 - \lambda_1 \partial_x \Psi_{101}, \qquad L \Psi_{200} = \varphi_0^2
\end{align*}
augmented with $\FSproj \Psi_{ijk} = 0$. The unique solvability of each of these problems is ensured by Lemma~\ref{lem bordering}. In particular, one can verify immediately that $\Psi_{101} = 0$ and hence $L \Psi_{102} = -\rho_2\varphi_0$. Therefore 
\[
\Psi_{011} = -{\lambda_1\over2} x^2 \cos(\rho_0 y), \qquad \Psi_{102} = -{\rho_2\over2} x^2 \cos(\rho_0 y).
\] 
Solving for $\Psi_{200}$ is much more complicated. Differentiating the equations for $\Psi_{200}$ with respect to $x$, we know that $L(\partial_x \Psi_{200}) = 0$. 
Hence $\partial_x \Psi_{200} = (c_1 + c_2 x) \varphi_0(y)$ for some constants $c_1$ and $c_2$. Antidifferentiating, this means that 
\[
\Psi_{200} = \left( c_1 x + {1\over2} c_2 x^2 \right) \cos(\rho_0 y) + g(y),
\]
for some function $g$. The constants $c_1,\ c_2$ and the function $g$ will be determined from the projection condition and the PDE. Combining these, we obtain
\begin{equation*}
\begin{cases}
g'' + \rho_0^2 g + c_2 \cos(\rho_0 y) = \cos^2(\rho_0 y) \quad \text{in }(0,1) \\
g'(0) = g'(1) + \beta g(1) = 0, \quad g(0) = c_1 = 0.
\end{cases}
\end{equation*}

This is an elementary ODE that can be solved explicitly, resulting in a somewhat complicated expression for $\Psi_{200}$. However in the reduced ODE we only need to find
\[
{d^2\over dx^2}\Big|_{x=0}\Psi_{200}(x,0) = c_2 = {4\sin (\rho_0) (3 - \sin^2(\rho_0)) \over 3 (\sin(2\rho_0) + 2\rho_0)}.
\]
Finally, this gives the expansion
\[
f(A, B, \varepsilon) = -\rho_2 A \varepsilon^2 - \lambda_1 B \varepsilon + {4\sin (\rho_0) (3 - \sin^2(\rho_0)) \over 3 ( \sin(2\rho_0) + 2\rho_0)} A^2  + r(A, B, \varepsilon).
\]

\bibliographystyle{siam}
\bibliography{projectdescription}

\end{document}